\documentclass[11pt]{article}
\usepackage[T1]{fontenc}
\usepackage[utf8]{inputenc}
\usepackage{geometry}
\usepackage{amsmath}
\usepackage{amsthm}
\usepackage{enumitem}
\usepackage{setspace}
\usepackage[unicode=true,
 bookmarks=false,
 breaklinks=false,pdfborder={0 0 1},backref=false,colorlinks=false]
 {hyperref}
\usepackage[dvipsnames]{xcolor} 
\usepackage[color=Purple!50!white, textwidth=22mm]{todonotes}
\usepackage{comment}
\allowdisplaybreaks

\makeatletter
\theoremstyle{plain}
\newtheorem{thm}{\protect\theoremname}[section]
\theoremstyle{plain}
\newtheorem{lem}[thm]{\protect\lemmaname}
\theoremstyle{plain}
\newtheorem{claim}[thm]{\protect\claimname}
\theoremstyle{plain}
\newtheorem{cor}[thm]{\protect\corollaryname}

\newtheorem{innercustomgeneric}{\customgenericname}
\providecommand{\customgenericname}{}
\newcommand{\newcustomtheorem}[2]{%
  \newenvironment{#1}[1]
  {%
   \renewcommand\customgenericname{#2}%
   \renewcommand\theinnercustomgeneric{##1}%
   \innercustomgeneric
  }
  {\endinnercustomgeneric}
}

\newcustomtheorem{customthm}{Theorem}
\newcustomtheorem{customlemma}{Lemma}

\usepackage{amssymb}

\makeatother

\providecommand{\claimname}{Claim}
\providecommand{\lemmaname}{Lemma}
\providecommand{\theoremname}{Theorem}
\providecommand{\corollaryname}{Corollary}

\setenumerate[0]{label=(\roman*), ref=\roman*}

\begin{document}
\title{Subset sums, completeness and colorings}
\author{David Conlon\thanks{Department of Mathematics, California Institute of Technology, Pasadena, CA 91125. Email:  {\tt dconlon@caltech.edu.} Research supported by NSF Award DMS-2054452.} \and Jacob Fox\thanks{Department of Mathematics, Stanford University, Stanford, CA 94305. Email: {\tt jacobfox@stanford.edu}. Research supported by a Packard Fellowship and by NSF Award DMS-1855635.} \and Huy Tuan Pham\thanks{Department of Mathematics, Stanford University, Stanford, CA 94305. Email: {\tt huypham@stanford.edu}.} }
\date{}
\maketitle

\begin{abstract}
We develop novel techniques which allow us to prove a diverse range of results relating to subset sums and complete sequences of positive integers, including solutions to several longstanding open problems. These include: solutions to the three problems of Burr and Erd\H{o}s on Ramsey complete sequences, for which Erd\H{o}s later offered a combined total of \$350; analogous results for the new notion of density complete sequences; the solution to a conjecture of Alon and Erd\H{o}s on the minimum number of colors needed to color the positive integers less than $n$ so that $n$ cannot be written as a monochromatic sum; the exact determination of an extremal function introduced by Erd\H{o}s and Graham on sets of integers avoiding a given subset sum; and, answering a question reiterated by several authors, a homogeneous strengthening of a seminal result of Szemer\'edi and Vu on long arithmetic progressions in subset sums.
\end{abstract}

\section{Introduction }

Many of the most famous problems and results in mathematics concern the representation of positive integers as the sum of elements from a sparse sequence. For example, the long open Goldbach conjecture states that every even integer at least four is the sum of two primes, while Vinogradov's theorem states that every sufficiently large odd integer is the sum of three primes (and was recently extended by Helfgott~\cite{Helfgott} to cover all odd integers at least seven). Some other notable results of this type include Lagrange's four-square theorem that every positive integer is the sum of four squares, Gauss' Eureka theorem that every positive integer is the sum of three triangular numbers and the Hilbert--Waring theorem. 

While these problems concern the representation of integers as the sum of a bounded number of terms from a particular sequence, there are many results and open problems which do not stipulate a bound on the number of terms. A prominent example of such a result is a theorem of Szemer\'edi and Vu~\cite{SV}, confirming an old conjecture of Erd\H{o}s~\cite{Er62}, which says that there is a constant $C$ such that if $A = (a_n)_{n=1}^\infty$ is an infinite increasing sequence of integers with $|A \cap [n]| \geq C \sqrt{n}$ for all sufficiently large $n$ which intersects every infinite arithmetic progression of integers, then we can represent any sufficiently large integer as a sum of distinct terms from the sequence. In this paper, we develop general methods which solve many open problems of precisely this type. 

To be more precise, given a set or a sequence $A$ of integers, we define the set of subset sums $\Sigma(A)$ to be the set of all integers representable as a sum of distinct elements from $A$. That is, 
\[\Sigma(A) = \left\{ \sum_{s \in S} s : S \subseteq A \right\}.\]
Our contribution then is to solve several open problems on conditions which guarantee that $\Sigma(A)$ contains either a particular integer or all sufficiently large integers. In particular, we answer several old questions of Burr and Erd\H{o}s~\cite{BE2} on the density of so-called Ramsey complete sequences, for whose solution Erd\H{o}s~\cite{Er95} later offered \$350. We also solve a conjecture of Alon and Erd\H{o}s~\cite{AEr} on the minimum number of colors needed to color the positive integers less than $n$ so that $n$ cannot be written as a monochromatic sum and determine exactly the answer to an extremal question first studied by Alon, Erd\H{o}s and Graham \cite{A, Er89} on the maximum size of a set avoiding a particular subset sum. Finally, answering a question reiterated by several groups of authors, including Erd\H{o}s and S\'ark\"ozy~\cite{ES}, S\'ark\"ozy~\cite{Sar3} and Tran, Vu, and Wood \cite{TVW}, we prove a homogeneous strengthening of another result of Szemer\'edi and Vu~\cite{SV} from which the Erd\H{o}s conjecture mentioned above was derived.

What unites these seemingly disparate topics is a common proof framework that allows us to show the existence of a long interval in the set of subset sums of an integer set $S$. This framework has several steps:

\begin{enumerate}
\item We partition $S$ into $\ell$ parts $S_1,\ldots,S_{\ell}$ of roughly equal size for an appropriate choice of $\ell$. 
\item We further partition each part $S_i$ into two parts $S_i'$ and $S_i''$ of appropriate size and show that, for any $s\in S_i''$, the set of subset sums of $S_i'$ modulo $s$ is large. 
\item Using step (ii), we show that $\Sigma(S_i) = \Sigma(S_i' \cup S_i'')$ is dense in some long interval.  
\item Using step (iii), we show that $\Sigma(S) = \Sigma(S_1 \cup \dots \cup S_\ell)$ contains a long interval.
\end{enumerate}  
Step (ii) is the heart of the method and must be appropriately tailored to each application, drawing variously on the probabilistic method, on structural results from additive number theory and on estimates from analytic number theory. We will say more about our methods in Section~\ref{sec:overview}. For now, we will focus on describing our main results, along with several extensions, variations and applications, in more detail.

\subsection{Ramsey completeness and density completeness} 

We say that a sequence of positive integers $A$ is \emph{complete} if every sufficiently large positive integer is in $\Sigma(A)$ and \emph{entirely complete} if every positive integer is in $\Sigma(A)$. For example, the powers of two are entirely complete, while the powers of three are incomplete.
A far less simple example, due to Birch~\cite{B}, is that the sequence $\{p^i q^j : i, j \geq 0\}$ is complete whenever $p,q \geq 2$ are coprime integers. 
For more on the rich history of complete sequences (and some open problems), we refer the interested reader to~\cite{BEGL96, EG}.

Our starting point here is with the observation that the completeness property can be surprisingly fragile. Indeed, removing any element from the powers of two turns an entirely complete sequence into an incomplete one. For this reason, Burr and Erd\H{o}s~\cite{BE1, BE2} began the study of more robust notions of completeness. We will be concerned with two such notions here, namely, robustness under partitioning, known as Ramsey completeness in the literature, and robustness under taking subsets, a new concept which we refer to as density completeness.

\subsubsection{Ramsey completeness}

Following Burr and Erd\H{o}s~\cite{BE2}, we say that a sequence of positive integers $A$ is \emph{$r$-Ramsey complete} if, whenever the sequence is partitioned into $r$ classes $A_1, A_2, \dots, A_r$, every sufficiently large positive integer is in $\bigcup_{i=1}^r \Sigma(A_i)$ and \emph{entirely $r$-Ramsey complete} if every positive integer is in $\bigcup_{i=1}^r \Sigma(A_i)$. Equivalently, $A$ is entirely $r$-Ramsey complete if, for any coloring of $A$ using $r$ colors, every positive integer can be written as a monochromatic subset sum. 

In their paper introducing these concepts, Burr and Erd\H{o}s~\cite{BE2} constructed an entirely $2$-Ramsey complete sequence $A$ with the property that $|A \cap [n]| \leq C \log^3 n$ for all $n$, where $C$ is an absolute constant. In the other direction, they were able to show that there is a constant $c > 0$ for which there is no $2$-Ramsey complete sequence with $|A \cap [n]| \leq c \log^2 n$ for all sufficiently large $n$. They also asked whether it might be possible to narrow the gap between these two estimates and Erd\H{o}s~\cite{Er95} later offered \$100 for such an improvement. 

For $r \geq 3$, the results of Burr and Erd\H{o}s clearly imply that there is no $r$-Ramsey complete sequence with $|A \cap [n]| \leq c \log^2 n$ for all sufficiently large $n$. However, even for $r = 3$, they were unable to construct an $r$-Ramsey complete sequence with $|A \cap [n]| =n^{o(1)}$. Given the lack of progress on this problem, Erd\H{o}s~\cite{Er95} later offered \$250 for any non-trivial result. Our first theorem solves both this problem and that above at once, by determining the growth rate of the sparsest possible $r$-Ramsey complete sequence up to an absolute constant factor.

\begin{thm} \label{thm:Ramsey-complete}
There is a constant $C$ such that, for every integer $r\ge2$, there is an $r$-Ramsey complete
sequence $A$ with $|A\cap[n]|\le Cr\log^{2}n$ for all $n$. Furthermore, there is a constant
$c>0$ such that no sequence $A$ with $|A\cap[n]|\le cr\log^{2}n$
for all sufficiently large $n$ is $r$-Ramsey complete. 
\end{thm}

Note that the lower bound, that is, the statement that there is a constant $c>0$ such that no sequence $A$ with $|A\cap[n]|\le cr\log^{2}n$
for all sufficiently large $n$ is $r$-Ramsey complete, already improves on Burr and Erd\H{o}s' result, which had no dependency on $r$. We note also that a standard compactness argument implies that if $A$ is an $r$-Ramsey complete sequence, then there is $n(A)$ such that, for every $r$-coloring of $A$, every positive integer at least $n(A)$ can be written as a sum of distinct monochromatic elements. We may therefore enlarge the $r$-Ramsey complete sequence $A$ constructed in Theorem~\ref{thm:Ramsey-complete} to an entirely $r$-Ramsey complete sequence by including all positive integers less than $n(A)$.

The key to proving Theorem~\ref{thm:Ramsey-complete} is a density-type result, Lemma~\ref{lem:main-ramsey}, saying that, with high probability, a random sequence of $C \epsilon^{-1} \log x$ elements chosen from those elements of the interval $[x, 2x)$ with no small prime factor has the property that any subset of size $C \log x$ contains a particular long interval in its set of subset sums. This density statement already improves a result of Spencer~\cite{Sp81} from 1981 by showing that, for any integers $r \geq 2$ and $n$ sufficiently large in terms of $r$, there is a set of integers $S$ of size $C r \log n$ with the property that any $r$-coloring of $S$ contains a monochromatic subset whose elements add to $n$. More to the point, by concatenating the sequences given by Lemma~\ref{lem:main-ramsey}, one for each dyadic interval $[x, 2x)$, it is easy to construct the sparse $r$-Ramsey complete sequence $A$ required by Theorem~\ref{thm:Ramsey-complete}.

We also study Ramsey completeness for polynomial sequences. The study of ordinary completeness for polynomial sequences has a long history, with important contributions by Sprague~\cite{Sprague}, Roth and Szekeres~\cite{RoSz54} and Cassels~\cite{Cassels}. These efforts culminated in a result of Graham~\cite{G}, who characterized all real polynomial sequences which are complete (where the definition of completeness extends to real-valued sequences without alteration). Graham first observed the well-known fact that every real polynomial of degree $k$ can be written as $P(x)=\sum_{i=0}^k \alpha_i \binom{x}{i}$, where $\binom{x}{i}$ is the polynomial $\frac{1}{i!}\prod_{j=0}^{i-1}(x-j)$ and $\alpha_i \in \mathbb{R}$ with $\alpha_k \not = 0$. He then showed that $(P(m))_{m \geq 1}$ is complete if and only if the following three properties hold: 
\begin{enumerate}
\item $\alpha_k>0$, 
\item $\alpha_i=p_i/q_i$ for each $i$, where $p_i$ and $q_i$ are relatively prime integers, and 
\item $\gcd(p_0,p_1,\ldots,p_k)=1$. 
\end{enumerate}

Given this body of work, it was a natural step for Burr and Erd\H{o}s \cite{BE2} to ask which polynomial sequences are Ramsey complete. According to Erd\H{o}s~\cite{Er95}, Burr subsequently proved that the sequence of $k^{\textrm{th}}$ powers is $r$-Ramsey complete for all $r \geq 2$, though this result was never published. Our next theorem subsumes this result, answering their question completely by showing that all complete polynomial sequences are $r$-Ramsey complete for all $r \geq 2$. In fact, it gives much more, extending the upper bound in Theorem~\ref{thm:Ramsey-complete}, which corresponds to the case $P(x) = x$, by showing that every complete polynomial sequence has a subsequence which is $r$-Ramsey complete and as sparse as possible. 
Note again that in this context we are allowing the sequence to be real-valued, rather than restricting to the integers. The definitions of completeness and Ramsey completeness should then be adjusted to facilitate this change. 

\begin{thm} \label{thm:Ramsey-complete-polynomial}
For any positive integer $k$, there is a constant $C(k)$ such that, for every polynomial $P$ of degree $k$ for which $(P(m))_{m\ge1}$ is complete and every $r \geq 2$, there is an $r$-Ramsey complete subsequence $A\subset(P(m))_{m\ge1}$ with $|A\cap[n]|\le C(k)r\log^{2}n$ for all $n$. 
\end{thm}

\subsubsection{Density completeness}

We say that a sequence of positive integers $A$ is \textit{$\epsilon$-complete} if every subsequence
$A'$ of $A$ with the property that $|A'\cap[n]|\geq\epsilon|A\cap[n]|$ for all sufficiently
large $n$ is complete. This is the natural density analogue of Ramsey completeness, though it is not at all
obvious that such sequences actually exist. Indeed, since the even integers are not complete, the set of all positive integers is
not $(\frac{1}{2}-\delta)$-complete for any $\delta>0$, an observation which might suggest that no
$\epsilon$-complete sequences exist when $\epsilon$ is small.
However, by using the result of Szemer\'edi and Vu~\cite{SV}, which we will discuss in more detail in Section~\ref{sec:homog},
that there is a constant $C$ such that any subset of $[n]$ of size at least $C\sqrt{n}$ contains an arithmetic progression of length $n$ in its set of subset sums, one can show that any sequence of primes $A$ with $|A\cap[n]|\geq2C\epsilon^{-1}\sqrt{n}$ for all sufficiently large $n$ is $\epsilon$-complete.
Thus, the correct takeaway is that the property of being $\epsilon$-complete is not monotone. More concretely, as in the example above where we looked at all positive integers, an $\epsilon$-complete sequence cannot have an $\epsilon$-proportion of its elements sharing a common divisor. 

In keeping with our results about Ramsey completeness, our main result regarding this new notion of $\epsilon$-completeness is a determination of 
how sparse an $\epsilon$-complete sequence can be. To state this result, we need some notation.
Let $F=(f_{n})_{n\geq1}$ be any sequence
of positive integers for which $f_{n}=\sum_{i\leq\epsilon n}f_{i}$
for all sufficiently large $n$. It is easy to see that any two such sequences are comparable, growing within
a constant factor of each other which depends only on the initial terms. 
In Appendix \ref{appendix:est-density-ramsey}, we will show that any such $F$ satisfies 
\[
f_{n}=e^{\left(\frac{1}{2\log(1/\epsilon)}+o(1)\right)(\log n)^{2}}
\]
or, equivalently, 
\[
|F \cap [n]|=e^{\sqrt{\left(2\log(1/\epsilon)+o(1)\right)\log n}}.
\]
The promised result now says that the fastest-growing $\epsilon$-complete sequence 
grows on the same order as $F$.

\begin{thm} \label{epsilonclean} 
Let $F=(f_{n})_{n\geq1}$ be any sequence of
positive integers for which $f_{n}=\sum_{i\leq\epsilon n}f_{i}$ for
all sufficiently large $n$. Then every $\epsilon$-complete sequence $A=(a_{n})_{n\geq1}$
must satisfy $a_{n}=O(f_{n})$ and there is an $\epsilon$-complete
sequence with $a_{n}=\Theta(f_{n})$. 
\end{thm}

Like with Ramsey completeness, we may also prove a generalization regarding $\epsilon$-complete subsequences
of complete polynomial sequences, though in this case we omit the details of the argument, only pointing to how ideas from the proofs of Theorems~\ref{thm:Ramsey-complete-polynomial} and \ref{epsilonclean} can be combined to give the required conclusion.

\begin{thm} \label{thm:poly-density-ramsey}
Let $P$ be a polynomial for which
the sequence $(P(m))_{m\geq1}$ is complete. Then there is a subsequence
$A=(a_{n})_{n\ge1}$ of $(P(m))_{m\geq1}$ with $a_{n}=\Theta(f_{n})$
which is $\epsilon$-complete. That is, any complete polynomial sequence has
an $\epsilon$-complete subsequence which is as sparse as an $\epsilon$-complete
sequence can be. 
\end{thm}

\subsection{Ensuring a given subset sum} 

So far, we have discussed problems and results on notions of completeness, where we require that all sufficiently large integers can be represented as subset sums. We now address the natural problem of ensuring that a particular integer is a subset sum, again looking at both a Ramsey variant and a density variant.

\subsubsection{Monochromatic subset sums} \label{subsect:sure}

Given a positive integer $n$, let $f(n)$ be the minimum integer $r$ for which there is an $r$-coloring of the positive integers less than $n$ with the property that $n$ cannot be written as a monochromatic sum of distinct integers. The problem of estimating $f(n)$ was raised by Erd\H{o}s many times~\cite{Er81, Er82, Er822}, culminating in a problem paper~\cite{Er952} where he stated that he could show $f(n)=o(n^{1/3})$ and asked whether $f(n)=n^{1/3-o(1)}$. 
Solving this problem, Alon and Erd\H{o}s \cite{AEr} showed that there are positive constants $c_{1}$ and $c_{2}$ such that 
\begin{equation}\label{AlonErdosbound}
\frac{c_{1}n^{1/3}}{\log^{4/3}n}\le f(n)\le\frac{c_{2}n^{1/3}(\log\log n)^{1/3}}{(\log n)^{1/3}},
\end{equation}
adding that they suspect the upper bound is closer to the truth. Using his result with Szemer\'edi~\cite{SV} on long arithmetic progressions in subset sums, 
Vu~\cite{V} later refined the lower bound, showing that $f(n)\ge\frac{c_{1}n^{1/3}}{\log n}$ for some positive $c_1$.

We improve these results further, determining $f(n)$ up to an absolute constant factor and thereby confirming Alon and Erd\H{o}s' conjecture that their upper bound is close to the true order of magnitude. As is customary, we write $\phi(n)$ for the Euler totient function, the number of positive integers less than $n$ which are coprime to $n$. 

\begin{thm} \label{thm:monochromatic-subset-sums-N}
For every positive integer $n$, the minimum number of colors $f(n)$ for which it is possible to color the positive integers less than $n$ so that $n$ cannot be written as a monochromatic sum of distinct integers satisfies 
\[
f(n)=\Theta\left(\frac{n^{1/3}(n/\phi(n))}{(\log n)^{1/3}(\log\log n)^{2/3}}\right).
\]
\end{thm}

Standard estimates imply that $n/\phi(n) \in (1,2\log \log n)$ for $n$ sufficiently large, with $n/\phi(n)$ large if and only if $n$ is divisible by many small primes. As a result, $f(n)$ is surprisingly far from being monotone, exhibiting local multiplicative fluctuations on the order of $\log \log n$. Moreover, though $f(n)$ is indeed close to the upper bound proved by Alon and Erd\H{o}s, differing by at most a $\log \log n$ factor, their upper bound is only optimal up to a constant factor when $n$ is divisible by many small primes.

To give some sense of where our improvement comes from, let us briefly describe the coloring that Alon and Erd\H{o}s use for their upper bound, using $r$ colors in total. First, they use $r/2$ colors to color all integers in $[n-1]$ larger than $2n/r$, with all integers in $[n/(j+1),n/j)$ getting color $j$. Since any $j$ distinct integers of color $j$ have sum less than $n$ and any $j+1$ distinct integers of color $j$ have sum larger than $n$, we see that $n$ is not a sum of distinct elements from any of these color classes. Second, for each of the first $r/4$ primes $p$ that are coprime to $n$, they place all remaining multiples of $p$ 
in a color class. Since each sum of multiples of $p$ is itself a multiple of $p$ and each $p$ is coprime to $n$, we see that $n$ is again not a sum of elements from any of these color classes. To complete the construction, we group the few remaining uncolored integers into color classes so that the sum of the elements in any given color class is less than $n$. A careful analysis then shows that $r$ can be taken to be the upper bound in (\ref{AlonErdosbound}).

As in the Alon--Erd\H{o}s coloring, our coloring uses $r/2$ colors to color all integers in $[n-1]$ larger than $2n/r$ and then $r/4$ colors to color the multiples of each of the first $r/4$ primes which are coprime to $n$. However, we then add an additional third step, which makes use of 
the non-uniform distribution of the remaining elements in congruence classes modulo $d$ for an appropriate choice of $d$.  Indeed, let $d$ be as large as possible so that $d$ is coprime to $n$ and $\phi(d)<r/32$, noting that the prime factors of $d$ must be among the first $r/4$ primes coprime to $n$ and so the remaining uncolored integers are all coprime to $d$. For each congruence class $t \pmod d$ with $t$ coprime to $d$, let $x_t \in [d]$ be such that  $x_t \equiv t^{-1}n \pmod d$. If a sum of elements, each congruent to $t \pmod d$, is equal to $n$, then the sum must involve either $x_t$ terms, $d+x_t$ terms or more than $d+x_t$ terms. Therefore, arguing as for the first $r/2$ colors, neither the set of integers congruent to $t\pmod d$ which are at least $n/x_t$ nor the set of integers congruent to $t\pmod d$ which are at least $n/(d+x_t)$ and less than $n/x_t$ can contain a subset sum equal to $n$. Hence, using at most $2\phi(d)<r/8$ additional colors, we may color all integers in $[n-1]$ larger than $n/d$ in such a way that $n$ is not a monochromatic sum of distinct elements. To complete the coloring, we again group the remaining uncolored integers into color classes so that the sum of the elements in any given color class is less than $n$. Worked out carefully, this then returns the upper bound in Theorem~\ref{thm:monochromatic-subset-sums-N}. For a sketch of how we prove the matching lower bound, which is the more difficult aspect of the proof, 
we refer the reader to Section~\ref{subsec:monochromatic}.

In practice, since our methods allow it, we will prove a more general result. For the statement, we need some notation. For positive integers $\rho$ and $m$, writing $p_i$ for the $i^{\textrm{th}}$ prime, we let $W(\rho)=\prod_{i=1}^{\rho}p_i$ and $\tau(\rho,m)=\phi(W(\rho)m)/(W(\rho)m)$. For $m \in [n,\binom{n}{2}]$, we then let $\rho(n,m)$ be the smallest positive integer $\rho$ such that $\rho/\tau(\rho,m) \ge n^2/\phi(m)$. Our generalization of Theorem~\ref{thm:monochromatic-subset-sums-N} is now as follows.

\begin{thm} \label{thm:monochromatic-subset-sums}
For every positive integer $n$ and any $m\in [n,\binom{n}{2}]$, the minimum number of colors $f(n, m)$ for which it is possible to color the positive integers less than $n$ so that $m$ cannot be written as a monochromatic sum of distinct integers satisfies 
\[
f(n,m)=\Theta\left(\min\left( \frac{m^{1/3}(m/\phi(m))}{(\log n)^{1/3}(\log\log n)^{2/3}},\rho(n,m)\right)\right).
\]
\end{thm}

\subsubsection{The largest set avoiding a given subset sum} 

What is the maximum size $g(n, m)$ of a subset of $[n]$ which has no subset sum equal to $m$? Variants of this natural extremal problem, interesting for any positive integers $n < m \leq \binom{n+1}{2}$, were originally raised by Erd\H{o}s and Graham (see, for instance,~\cite[Page 59]{EG} and~\cite{Er89}), although, in the exact form mentioned here, the problem was first studied in detail by Alon~\cite{A}.

If we let $\textrm{snd}(m)$ be the smallest positive integer that does not divide $m$, an easy lower bound for $g(n, m)$ is $\lfloor \frac{n}{\textrm{snd}(m)}\rfloor$, since the set of all multiples of $\textrm{snd}(m)$ below $n$ does not have $n$ as a subset sum. This simple observation of Alon~\cite{A} was later refined by Alon and Freiman~\cite{AF}, who observed that $g(n, m) \geq s(n,m) := \lfloor \frac{n}{\textrm{snd}(m)}\rfloor +\textrm{snd}(m)-2$ by augmenting the example above with $\textrm{snd}(m)-2$ additional elements, each congruent to either $1$ or $-1$ modulo $\textrm{snd}(m)$. Another simple lower bound, better than that above when $m$ is close to $\binom{n+1}{2}$, is $g(n,m)\ge \lfloor \sqrt{2m}-1/2 \rfloor$, following from the fact that the sum of the first $\lfloor \sqrt{2m}-1/2 \rfloor$ positive integers is less than $m$.

For the upper bound, Alon~\cite{A} first showed that if $n^{1+\epsilon} \leq m \leq n^2/(\log n)^2$, then $g(n,m) = O(s(n,m))$, where the implicit constant depends on $\epsilon$. He also conjectured that $g(n,m)=(1+o(1))s(n,m)$ in roughly the same range. For $Cn(\log n)^6 \le m \leq n^{1.5-o(1)}$, this conjecture was proved soon after by Lipkin~\cite{Lip}. Remarkably, around the same time, Alon and Freiman~\cite{AF} determined the function exactly for $n^{5/3+o(1)} \leq m < \frac{n^{2}}{20(\log n)^{2}}$, establishing that $g(n,m)=s(n,m)$ in this range. More than twenty years then elapsed before Tran, Vu and Wood \cite{TVW} proved Alon's conjecture in full generality by showing that $g(n, m) = (1 + o(1))s(n,m)$ for $n(\log n)^{1+o(1)} \leq m \leq \frac{n^{2}}{9(\log n)^{2}}$. We improve these results, determining the function exactly for all $Cn\log n \leq m \leq \frac{n^{2}}{(8+o(1))(\log n)^{2}}$ and asymptotically for all $Cn\log n \leq m \leq \binom{n+1}{2}$.

\begin{thm} \label{thm:Alon-conjecture} 
There is a constant $C$ such that if $n$ and $m$ are positive integers and $g(n, m)$ is the maximum size of a subset of $[n]$ with no subset sum equal to $m$, then
$$g(n,m) = s(n,m) = \left\lfloor \frac{n}{\textrm{snd}(m)}\right\rfloor +\textrm{snd}(m)-2$$ 
for $m\in\left[Cn\log n,\frac{n^{2}}{12(\log n)^{2}}\right]$ and
$$g(n,m)=\max\left(s(n,m),(1+o(1))\sqrt{2m}\right)$$
for $m\in \left[\frac{n^{2}}{12(\log n)^{2}},\binom{n+1}{2}\right]$.
\end{thm}

Since $\textrm{snd}(m) \leq (1+o(1))\log m\le (2+o(1))\log n$, the theorem in fact implies that $g(n,m) = \max\left(s(n,m),(1+o(1))\sqrt{2m}\right) = s(n,m)$ for $Cn \log n \le m\le \frac{n^2}{(8+o(1))(\log n)^2}$, as promised above. On the other hand, once $m < c n \log n$ for $c$ sufficiently small, we do not generally have the bound $g(n,m) = (1+o(1))s(n,m)$. Indeed, for $8n < m < n(\log n)/8$, we can show that there is a subset of $[n]$ of size $h=\lfloor n^2/(2m) \rfloor$ with no subset sum equal to $m$, so if $\textrm{snd}(m) \ge (\log n)/2$, then $g(n,m) \ge h > (2-o(1))s(n,m)$. To show the existence of the required subset of size $h$, choose an integer $n' \in (h+3n/4,n]$ such that $m\in [n'n/(2h),n'n/(2h)+n/4]$. Note that $n' \in (h+3n/4,n]$ as $n'\le 2mh/n\le n$ and, since $n^2/(2m) \ge h \ge n^2/(2m)-1$, we can verify that, for $m\in (8n,n(\log n)/8)$ and $n$ sufficiently large, $n' \ge \frac{2h(m-n/4)}{n} > h+3n/4$. 
Observe now that the set of subset sums of the interval $[n'-h,n']$ does not contain any element from the interval $[n'n/(2h),n'n/(2h)+n/4]$, since any sum of at most $n/(2h)$ elements from the interval is strictly smaller than $n'n/(2h)$, while any sum of at least $n/(2h)+1$ elements from the interval is strictly larger than $(n'-h)(n/(2h)+1) = n'n/(2h)+n'-h-n/2 > n'n/(2h)+n/4$. Therefore, the interval $[n'-h,n']$ has size at least $h$ and does not contain $m$ as a subset sum, as required.

To say more about how we prove Theorem~\ref{thm:Alon-conjecture}, we must first discuss the main tool used in our proof, a strengthening of the subset sums result of Szemer\'edi and Vu~\cite{SV} which is itself of independent interest.

\subsection{Long homogeneous progressions in subset sums} \label{sec:homog}

We opened this paper by mentioning Szemer\'edi and Vu's proof~\cite{SV} of a longstanding conjecture of Erd\H{o}s~\cite{Er62}. As shown by Folkman~\cite{Folk}, this is itself a corollary of the statement that there is a constant $C$ such that if $A = (a_n)_{n=1}^\infty$ is an infinite increasing sequence of integers with $|A \cap [n]| \geq C \sqrt{n}$ for all sufficiently large $n$, then $\Sigma(A)$ contains an infinite arithmetic progression. In proving this latter statement, Szemer\'edi and Vu first proved the following finite analogue, which we have already mentioned several times. Note that this result is clearly best possible, as may be seen by considering the set of all positive integers up to $\lfloor \sqrt{2n} - 1/2\rfloor$. 

\begin{thm}[Szemer\'edi--Vu \cite{SV}]\label{SVthm}
There is a constant $C$ such that if $A \subset [n]$ with $|A| \geq C\sqrt{n}$, then $\Sigma(A)$ contains an arithmetic progression of length $n$.
\end{thm} 

This theorem improved on an earlier result obtained independently by Freiman~\cite{F2} and S\'{a}rk\"{o}zy~\cite{Sar2}, who showed that there is a constant $C$ such that if $|A|\ge C\sqrt{n\log n}$, then $\Sigma (A)$ contains an arithmetic progression of length at least $n$. However, it also loses something, because the Freiman--S\'ark\"ozy result gives not only an arithmetic progression, but a {\it homogeneous progression}, an arithmetic progression $a, a + d, \dots, a + kd$ where the common difference $d$ divides $a$ and, hence, every other term in the progression. The natural question then, reiterated by several groups of authors, including Erd\H{o}s and S\'ark\"ozy~\cite{ES}, S\'ark\"ozy~\cite{Sar3} and Tran, Vu, and Wood \cite{TVW}, 
is whether there is a common strengthening of the Szemer\'edi--Vu and Freiman--S\'ark\"ozy theorems. We answer this question in the affirmative.

\begin{thm} \label{thm:Sze-Vu-2}
There is a constant $C$ such that if $A \subset [n]$ with $|A| \geq C\sqrt{n}$, then $\Sigma(A)$ contains a homogeneous progression of length $n$.
\end{thm}

For the proof of Theorem~\ref{thm:Alon-conjecture}, we need a slightly stronger version of Theorem~\ref{thm:Sze-Vu-2}. This result, Theorem~\ref{thm:Sze-Vu-2-strong}, states that if $A \subset [n]$ with $|A| \geq C \sqrt{n}$, then there exists $d$ (which is typically just $1$) such that most elements in $A$ are divisible by $d$ and the set of subset sums formed from adding at most $2^{50} n/|A|$ elements of $A$ which are divisible by $d$ contains a homogeneous progression with length $n$ and common difference $d$. 

To prove Theorem~\ref{thm:Alon-conjecture}, suppose now that $A$ is a subset of $[n]$ with $s(n, m) + 1$ elements and we wish to show that $m \in \Sigma(A)$. Using Theorem~\ref{thm:Sze-Vu-2-strong}, we may conclude that $\Sigma(A)$ contains a homogeneous progression with length $n$ and common difference $d$, where $d$ divides most elements of $A$. Moreover, if $d|m$, this progression will contain $m$, so we may assume that $d \geq \textrm{snd}(m)$. A simple counting argument then implies that $d$ must in fact equal $\textrm{snd}(m)$, as otherwise there will not be enough elements in $A$. Since $s(n,m) + 1 = \lfloor \frac{n}{\textrm{snd}(m)}\rfloor +\textrm{snd}(m)-1$, there must also be at least $\textrm{snd}(m)-1$ elements in $A$ which are not divisible by $\textrm{snd}(m)$. We complete the proof by using these additional elements to show that $m \in \Sigma(A)$, as required.

As another corollary of Theorem \ref{thm:Sze-Vu-2}, we also obtain an improved bound on an old question of Straus~\cite{Straus} (see also~\cite{EStr}) regarding the maximum size of a non-averaging subset of $[n]$, where a subset $A$ of $[n]$ is said to be {\it non-averaging} if no $a\in A$ is the average of two or more other elements of $A$. If we write $h(n)$ for the maximum size of a non-averaging subset of $[n]$, an elegant construction of Bosznay~\cite{Bosz} shows that $h(n)=\Omega(n^{1/4})$. On the other hand, if we write $H(n)$ for the maximum integer for which there are two subsets of $[n]$ of size $H(n)$ whose sets of subset sums have no non-zero common element, then a result of Straus~\cite{Straus} says that $h(n) \leq 2 H(n) + 2$. Using the Freiman--S\'ark\"ozy result on homogeneous progressions, Erd\H{o}s and S\'{a}rk\"{o}zy \cite{ES} were able to show that $H(n) = O(\sqrt{n \log n})$, which, by Straus' observation, also yields a similar upper bound on $h(n)$. By following their method, but using Theorem~\ref{thm:Sze-Vu-2} instead of the Freiman--S\'ark\"ozy result, we improve their bound to $H(n) = O(\sqrt{n})$, which is tight up to the constant factor, as may be seen by considering the sets $[1,c\sqrt{n}]$ and $[n-c\sqrt{n},n]$ for any $c<\sqrt{2}$. By Straus' inequality, it also provides an improved upper bound $h(n) = O(\sqrt{n})$ for the size of the largest non-averaging subset of $[n]$.

\begin{cor} \label{cor:nonavg}
There is a constant $C$ such that $H(n) \leq C \sqrt{n}$, where $H(n)$ is the largest integer for which there are two subsets of $[n]$ of size $H(n)$ whose sets of subset sums have no non-zero common element, and $h(n) \leq C\sqrt{n}$, where $h(n)$ is the size of the largest non-averaging subset of $[n]$.
\end{cor}

\subsection*{Organization of the paper}

In the next section, we will elaborate on our methods by giving rough outlines of the proofs of some of our main results. We then proceed to the formal proofs, proving Theorems~\ref{thm:Ramsey-complete} and \ref{thm:Ramsey-complete-polynomial} on Ramsey completeness in Section~\ref{sec:Ramsey-completeness}, Theorem~\ref{epsilonclean} on density completeness in Section~\ref{sec:density-ramsey} and Theorem~\ref{thm:monochromatic-subset-sums} on monochromatic subset sums in Section~\ref{sec:Monochromatic-subset-sums}. We turn to the proof of Theorem~\ref{thm:Sze-Vu-2}, our homogeneous strengthening of the Szemer\'edi--Vu theorem, and its consequence Corollary~\ref{cor:nonavg} in~Section \ref{sec:Sze-Vu-proof} and conclude in Section \ref{sec:extremal-sum} by proving Theorem~\ref{thm:Alon-conjecture} on the largest set avoiding a particular subset sum. Several supplementary results are consigned to the appendices.

\subsection*{Notation}

For the sake of clarity of presentation, we omit floor and ceiling signs whenever they are not essential. We also maintain the convention that all logarithms are natural logarithms unless otherwise specified. 

\section{\label{sec:overview}Overview of the proofs of the main results}

The techniques used to prove Theorems~\ref{thm:Ramsey-complete}, \ref{thm:Ramsey-complete-polynomial}, \ref{thm:monochromatic-subset-sums} and \ref{thm:Sze-Vu-2} all share some similarities. In each case, we reduce a 
problem over $\mathbb{Z}$ to the corresponding problem over $\mathbb{Z}_{m}$.
In the cyclic setting, considering the structure of the ``almost
periods'', i.e., those elements whose inclusion does not significantly expand
the subset sum, allows us to transform our questions about subset sums
into problems about iterated sumsets. The literature on iterated sumsets is extensive,
allowing us to reach our desired conclusions by combining existing results on these 
sumsets with novel arguments from probabilistic combinatorics. In this section, we say more about the specific ideas
that go into the proofs of each of our main theorems.
The detailed proofs of these theorems and the other results described in the introduction are then in subsequent sections.  

\subsection{Some useful tools} \label{sec:tools}

We will repeatedly use the following simple lemma, allowing us to extend intervals in the set of subset sums by adding new elements. It is essentially Lemma 1 of Graham \cite{G}.

\begin{lem}[Graham~\cite{G}]\label{lem:verysimple}
Let $A$ be a set such that $\Sigma(A)$ contains all integers in the interval $[x,x+y)$.
\begin{enumerate}
\item If $a$ is a positive integer with $a \leq y$ and $a\notin A$, then $\Sigma(A \cup \{a\})$ contains all integers in the interval $[x,x+y+a)$. 
\item If $a_1,\ldots,a_s$ are positive integers such that $a_i \leq y+\sum_{j<i} a_j$ and $a_i \notin A$ for $i=1,\ldots,s$, then  
$\Sigma(A \cup \{a_1,a_2,\dots,a_s\})$ contains all integers in the interval $[x,x+y+\sum_{i=1}^s a_i)$.
\end{enumerate}
\end{lem}

\begin{proof}
For the proof of the first part, note that if $u \in [x,x+y)$, then $u \in \Sigma(A) \subset \Sigma(A \cup \{a\})$. If $u \in [x+y,x+y+a)$, then $u-a \in [x,x+y) \subset \Sigma(A)$, so $u=(u-a)+a \in \Sigma(A \cup \{a\})$. The second part follows from the first part by induction on $s$. 
\end{proof}

We will also make repeated use of the following result of Lev \cite{Lev}. The importance of this result is that it allows us to find long intervals in a set of subset sums
by first finding several dense subsets of long intervals and then summing these sets. Several weaker versions 
of this result appeared earlier in the literature, many of which would also suffice for our purposes.

\begin{lem}[Lev~\cite{Lev}] \label{lem:Lev}
Suppose $\ell,q\ge1$ and $n\ge3$ are integers with
$\ell\ge2\lceil(q-1)/(n-2)\rceil$. If $S_{1},\dots,S_{\ell}$ are integer
sets each having at least $n$ elements, each a subset of an interval
of at most $q+1$ integers and none a subset of an arithmetic
progression of common difference greater than one, then $S_{1}+\cdots+S_{\ell}$
contains an interval of length at least $\ell(n-1)+1$. 
\end{lem}

In working with general cyclic groups, the following analogue of the Cauchy--Davenport
theorem, a consequence of Theorem 1.1 from \cite{COS}, will also be useful to us. Given subsets $A$ and $B$ of an abelian group $G$, we define $A+B=\{a+b:\,a\in A, b\in B\}$ and $A-B = \{a-b:\,a\in A, b\in B\}$. For $k\in \mathbb{N}$, we define the $k$-fold sumset $kA = \underbrace{A+A+\cdots+A}_{k\textrm{ times}}$. 

\begin{lem}[Cochrane, Ostergaard and Spencer~\cite{COS}]
\label{lem:Cauchy-Davenport} If $A$ is a subset of an abelian group $G$ which is not contained in a coset of a proper
subgroup of $G$ and $r,s$ are non-negative integers which are not both zero, then 
\[
|rA-sA|\ge\min\left\{ |G|,\frac{(r+s+1)|A|}{2}\right\}.
\]
\end{lem}

We will also make use of the following result of Deshouillers and Freiman \cite{DF}. The following corrected statement of the result appears in \cite{BP}, where it is also shown that the hypothesis $|A+A|\le2.04|A|$ can be weakened to $|A+A|\le2.1|A|$. 

\begin{lem}[Deshouillers and Freiman~\cite{DF}] \label{lem:Freiman-2.04}
There exists a positive constant $\xi$ such
that if $A$ is a subset of $\mathbb{Z}_{m}$ of size at most $\xi m$
with $|A+A|\le2.04|A|$, then there exists a proper
subgroup $H\subseteq\mathbb{Z}_{m}$ such that either
\begin{enumerate}
\item $A$ is a subset of an arithmetic progression of $H$-cosets of length
$\ell$ with $(\ell-1)|H|\le|A+A|-|A|$,
\item $A$ meets exactly three $H$-cosets and these three $H$-cosets are terms of an arithmetic progression of $H$-cosets of length
$\ell$ with $(\min(\ell,4)-1)|H|\le|A+A|-|A|$ or 
\item $A$ is a subset of an $H$-coset and $|A|\ge\xi|H|$.
\end{enumerate} 
Here an arithmetic progression of $H$-cosets of length $\ell$ is a set of the form $\bigcup_{i\in [\ell]}(x+id+H)$, where $x,d\in \mathbb{Z}_{m}$ and $d\notin H$.
\end{lem}

The following simple lemma is crucial in the proofs of most of our main results. 

\begin{lem}
\label{lem:modp}Let $m$ be an integer. Let $A$ be a set of integers such that $m\notin A$ and the size of $\Sigma(A)$ considered modulo $m$
is at least $h$, then
$|\Sigma(A\cup\{m\})|\ge|\Sigma(A)|+h$. 
\end{lem}

\begin{proof}
The lemma follows since each modulo $m$ class containing an element of $\Sigma(A)$ contributes
at least one new element to $(\Sigma(A)+\{m\})\setminus\Sigma(A)$. 
\end{proof}

In showing that there are many subset sums over cyclic groups, we
use the following lemma, which shows that the set of new elements
whose inclusion do not expand the set of subset sums is small. 

\begin{lem} \label{lem:double-count}
Suppose $A\subset\mathbb{Z}_{m}$ with $d<|A|<m$ and let $G_{d}$
be the set of $x\in\mathbb{Z}_{m}$ such that $|(A+x)\cup A|\le|A|+d$.
Then $|G_{d}|\le\frac{|A|^{2}}{|A|-d}$. 
\end{lem}

\begin{proof}
For each $x\in\mathbb{Z}_{m}$, $|(A+x)\cap(\mathbb{Z}_{m}\setminus A)|\le|A+x|=|A|$,
while if $x\in G_{d}$, $|(A+x)\cap(\mathbb{Z}_{m}\setminus A)|\le d$
by definition. Furthermore, 
\[
\sum_{x\in\mathbb{Z}_{m}}|(A+x)\cap(\mathbb{Z}_{m}\setminus A)|=\sum_{a\in A}|\{x\in\mathbb{Z}_{m}:a+x\in\mathbb{Z}_{m}\setminus A\}|=\sum_{a\in A}(m-|A|)=|A|(m-|A|),
\]
where the second equality follows since, for each $a\in A$, $a+x$ is
an element of $\mathbb{Z}_{m}\setminus A$ for exactly
$|\mathbb{Z}_{m}\setminus A|=m-|A|$ values of $x$. Thus, 
\[
|A|(m-|A|)\le|G_{d}|\cdot d+(m-|G_{d}|)|A|,
\]
from which we get the desired inequality by rearranging.
\end{proof}

We will often use the lemma above in combination with the following simple result.

\begin{lem} \label{lem:stable-period} 
If $A \subset \mathbb{Z}_{m}$ and $x_1,\dots,x_k \in \mathbb{Z}_m$ satisfy $|(A+x_i)\cup A|\le |A|+d$ for all $i\in [k]$, then $|(A+x_1+\cdots+x_k)\cup A|\le |A|+kd$.
\end{lem}

\begin{proof}
We will show, by induction on $i$, that $|(A+x_1+\cdots+x_i)\cup A|\le |A|+id$ for $1\le i\le k$. This is clearly true for $i=1$. For the induction step, assume that $|(A+x_1+\cdots+x_{i-1})\cup A|\le |A|+(i-1)d$. Then
\begin{align*}
|(A+x_1+\cdots+x_i)\cup A|-|A| &= |(A+x_1+\cdots+x_i)\setminus A| \\
&\le |(A+x_1+\cdots+x_i)\setminus (A+x_i)|+|(A+x_i)\setminus A|\\
&=|(A+x_1+\cdots+x_{i-1})\setminus A|+|(A+x_i)\setminus A|\\
&\le (i-1)d+d = id.
\end{align*}
Thus, $|(A+x_1+\cdots+x_i)\cup A|\le |A|+id$ for $1\le i\le k$.
\end{proof}

\subsection{Outline of the proof of the upper bounds in Theorems~\ref{thm:Ramsey-complete}
and \ref{thm:Ramsey-complete-polynomial}}\label{subsec:Ramsey-complete}

The upper bound in Theorem~\ref{thm:Ramsey-complete} states that there exists a constant $C$ such that, for every $r \geq 2$, there is an $r$-Ramsey complete sequence $A$ with $|A\cap [n]| \leq C r \log^2 n$ for all $n$. The following density-type result is the key to the proof of this statement. 

\begin{lem} \label{lem:main-ramsey}
Let $C=3840$ and $\epsilon\in (0,1/2]$. Let $x$ be a positive integer. Let $X$ be the set
of integers in $[x,2x)$ with no prime divisor at most $(\log x)/2$. If
a sequence $S$ of $C\epsilon^{-1}\log x$ elements in $X$ is chosen independently
and uniformly at random, then, with high probability (as $x \to \infty$), $S$ has distinct terms and, for any
subsequence $S'$ of $S$ of size $\epsilon|S|=C\log x$, the set $\Sigma(S')$
contains all integers in the interval $[\frac{Cx\log x}{4},\frac{7Cx\log x}{8}]$. 
\end{lem}

The upper bound in Theorem \ref{thm:Ramsey-complete} can be easily deduced from Lemma \ref{lem:main-ramsey} as follows. 

\vspace{2mm}

\noindent {\it Proof of the upper bound in Theorem \ref{thm:Ramsey-complete}.} Let $\epsilon=1/r$ and let $x_0$ be large enough that the conclusion of Lemma \ref{lem:main-ramsey} holds with positive probability for this choice of $\epsilon$ and $x\ge x_0$. Let $x_i=2^ix_0$ and $y_i=Cx_i \log x_i$. By Lemma \ref{lem:main-ramsey}, for each dyadic interval $[x_{i},x_{i+1})$ with $i\ge 0$ we can pick a sequence $S_i$ of $Cr \log x_i$ distinct elements in this interval such that the set of subset sums of any subset of $S_i$ of size at least $|S_i|/r$ contains the integers in $I_i:=[y_i/4,7y_i/8]$. Note that every $r$-coloring of $S_i$ has a color class of size at least $|S_i|/r$ and so the set of monochromatic subset sums of $S_i$ contains the integers in $I_i$. We pick the sequence $A$ to be the concatenation of the sequences $S_i$ for $i\ge 0$. Observe that, for all $n$, we have $A(n) \leq \sum_{i : n \leq x_{i+1}} |S_i| \leq Cr(\log n)^2$. Moreover, since $y_{i+1}/4<7y_i/8$, the intervals $I_i$ cover all integers at least $y_0/4$. Thus, for every $r$-coloring of $A$, every sufficiently large integer can be represented as a monochromatic subset sum. That is, the sequence $A$ is $r$-Ramsey complete. \qed

\vspace{2mm}

We now give an informal sketch of the proof of Lemma~\ref{lem:main-ramsey}, showing how it follows from an appropriate combination of the results of Section~\ref{sec:tools} with some further ideas. To begin, we observe that for any fixed set $I$ of $C\log x$ indices in $[C\epsilon^{-1} \log x]$, the elements of the subsequence $(s_i : i\in I)$ of $S$ of size $\epsilon |S|$ are independently and uniformly distributed in $X$. By taking a union bound, it will therefore suffice to show that if $S'$ is a sequence of $C\log x$ elements chosen independently and uniformly from $X$, then the probability that $\Sigma(S')$ does not contain all integers in the interval $[\frac{Cx\log x}{4},\frac{7Cx\log x}{8}]$ is sufficiently small. 

For this, for some fixed $\ell$, we take $\ell$ disjoint random subsets $S''_1, \dots, S''_\ell$ of $S'$, each of size $|S'|/(8\ell)$, with the aim being to show that, with appropriately high probability, the set of subset sums $\Sigma(S_j'')$ is a dense subset of a long interval and is not contained in an arithmetic progression with common difference larger than $1$. Lemma~\ref{lem:Lev} then allows us to conclude that $S'' = S''_1 \cup \dots \cup S''_\ell$ is such that $\Sigma(S'') = \Sigma(S''_1) + \dots + \Sigma(S''_\ell)$ contains a long interval. Note, moreover, that $S''$ only has size $|S'|/8$, so there are at least $7|S'|/8$ elements still remaining in $S'$. Using Graham's lemma, Lemma~\ref{lem:verysimple}, we can use these elements to extend the long interval in $\Sigma(S'')$ to a significantly longer interval containing all of the required elements.

It only remains to show that $\Sigma(S_j'')$ is a dense subset of a long interval with appropriately high probability (showing that it is also not contained in an arithmetic progression with common difference larger than $1$ is reasonably straightforward). For this, we split $S''_j$ randomly into two disjoint pieces $P_1$ and $P_2$. The key remaining component is to show that for every $m \in X$, the set of integers in $[x, 2x)$ with no prime factor at most $(\log x)/2$, the mod $m$ set of subset sums $\Sigma_m(P_1)$ is large with very high probability. Very roughly, this follows by exposing the elements of $P_1$ one at a time and showing that most elements expand the mod $m$ set of subset sums significantly. Though we will not give a more detailed description here, we note that this key step again relies on several results from the previous section, including the Cauchy--Davenport-type statement, Lemma~\ref{lem:Cauchy-Davenport}, as well as Lemma~\ref{lem:double-count}, which bounds the number of almost periods, those $x$ for which $(A+x)\setminus A$ is small.
Finally, once we know that $|\Sigma_m(P_1)|$ is, with high probability, large for each $m \in X$, we can apply Lemma~\ref{lem:modp} repeatedly to conclude that $|\Sigma(S''_j)| \geq \sum_{m \in P_2} |\Sigma_m(P_1)|$, which yields the required lower bound for $|\Sigma(S''_j)|$.

The proof of Theorem \ref{thm:Ramsey-complete-polynomial} follows
a similar scheme. Let $P$ be a complete polynomial. By the characterization due to Graham~\cite{G}
discussed in the introduction, we can write $P(x)=\sum_{i=0}^{k}\alpha_{i}\binom{x}{i}$
with $\alpha_{k}>0$ and $\alpha_{i}=\frac{p_{i}}{q_{i}}$, where $p_{i}$ and $q_{i}$
are relatively prime integers, $q_{i}>0$ and $\gcd(p_{0},\dots,p_{k})=1$.
If $L=\textrm{lcm}(q_{0},\dots,q_{k})$, then the polynomial $L\cdot P$ has integer
coefficients in its binomial representation and satisfies Graham's
condition, so it is also complete. Furthermore, if $((L\cdot P)(a_{n}))_{n=1}^{\infty}$
is $r$-Ramsey complete, then $(P(a_{n}))_{n=1}^{\infty}$ is $r$-Ramsey complete,
so it suffices to work with complete polynomials which have
integer coefficients in their binomial representations. From now on, we will assume that $P$ is such a polynomial. 

To prove Theorem \ref{thm:Ramsey-complete-polynomial}, we prove the following polynomial analogue of Lemma \ref{lem:main-ramsey}. For a polynomial $P$ and a sequence $T$ of integers, let $P(T)$ be the sequence where we replace each term $t$ in $T$ by $P(t)$. 

\begin{lem} \label{lem:main-poly-ramsey}
Let $P$ be a complete polynomial of degree $k$ with integer coefficients in its binomial representation and let $C(k)= k2^{k+15}$. Let $\epsilon\in (0,1/2]$. Let $x$ be a positive integer. Let $X$ be the set of
elements $y$ in $[x,(1+1/k)x)$ such that $P(y)$ has no prime divisor
at most $(\log x)^{1/2}$. If a sequence $S$ of $C(k)\epsilon^{-1} \log x$
elements in $X$ is chosen independently and uniformly at random, then,
with high probability (as $x \to \infty$), $S$ has distinct terms and, for any subsequence $S'$ of $S$ of size
$\epsilon|S|$, the set $\Sigma(P(S'))$ contains all integers in
the interval $[\frac{e}{9}P(x)|S'|,\frac{8}{9}P(x)|S'|]$. 
\end{lem}

We now show how Theorem \ref{thm:Ramsey-complete-polynomial}
follows from Lemma \ref{lem:main-poly-ramsey}, just as the upper bound in Theorem \ref{thm:Ramsey-complete} follows from Lemma \ref{lem:main-ramsey}.

\vspace{2mm}

\noindent {\it Proof of Theorem \ref{thm:Ramsey-complete-polynomial}.}  Let $\epsilon=1/r$. For each positive integer $i$, let $x_i=(1+1/k)^i$, $y_i=C(k)P(x_i)\log x_i$ and $I_i=[ey_i/9,8y_i/9]$. For $i$ sufficiently large in terms of $P$ and $r$, Lemma \ref{lem:main-poly-ramsey} implies that we can pick a subsequence $S_i$ of $C(k)r\log x_i$ distinct terms in $[x_i,(1+1/k)x_i)$ such that any subsequence $S'$ of $S_i$ with $|S_i|/r$ terms has the property that $\Sigma(P(S'))$ contains all integers in the interval $I_i$. Therefore, since every $r$-coloring of $S_i$ has a color class of size at least $|S_i|/r$, the set of monochromatic subset sums of $P(S_i)$ contains the integers in $I_i$. 
We pick the sequence $A$ to be the concatenation of the sequences $P(S_i)$ with $i$ sufficiently large. Then, for all $n$, we have $A(n) \leq \sum_{i : P(x_i) \leq n} |S_i| = O_k\left((\log n)^2\right)$. Moreover, as $x_i$ is sufficiently large, $P(x)$ is increasing for  $x \geq x_i$ and $P(x_{i+1}) = P((1+1/k)x_i) \leq eP(x_i)$. It follows that, for $i$ sufficiently large, $ey_{i+1}/9 \leq 8y_i/9$ and the intervals $I_i$ and $I_{i+1}$ are overlapping. Hence, the intervals $I_i$ cover all sufficiently large integers. Thus, for every $r$-coloring of $A$, every sufficiently large integer can be represented as a monochromatic subset sum. That is, the sequence $A$ is $r$-Ramsey complete. \qed

\vspace{2mm}
 
The proof of Lemma~\ref{lem:main-poly-ramsey} itself follows along broadly similar lines to the proof of Lemma~\ref{lem:main-ramsey}. The key additional input, arising in the analogue of the step where we showed that $|\Sigma_m(P_1)|$ is large with high probability for each $m \in X$, is the following result on iterated sumsets of a set of polynomial values, proved through a form of PET induction (see, for example,~\cite{BL96}). For further details, we refer the reader to Section~\ref{sec:Ramsey-completeness}, where the proofs of Lemmas~\ref{lem:main-ramsey}~and~\ref{lem:main-poly-ramsey} are given in full.
 
\begin{lem} \label{lem:iterated-sum-growth}
There exists a constant $C_k$, depending only on $k$, such that if $P$ is a complete polynomial of degree $k$ with integer coefficients in its binomial representation, $x$ is sufficiently large depending on $P$, $m$ is an integer in $[x,2x)$, $(\log x)^{-1}<\alpha<1/2$ and $T$ is a subset of $[x,2x)$ of size at least $\alpha x$, then the iterated sumset $2^{k-1}P(T)-2^{k-1}P(T)$ contains more than $\alpha^{C_k}P(m)$ residue classes modulo $P(m)$.
\end{lem}

\subsection{Outline of the proof of the lower bound in Theorems \ref{thm:monochromatic-subset-sums-N} and \ref{thm:monochromatic-subset-sums}}\label{subsec:monochromatic}

Recall that, for any $n \leq m \leq {n \choose 2}$, $f(n,m)$ is defined as the minimum $r$ for which there is an $r$-coloring of $[n-1]$ such that $m$ cannot be written as a sum of distinct monochromatic elements. In this section, we sketch the main ideas behind the lower bound in Theorem~\ref{thm:monochromatic-subset-sums}, which asymptotically determines the value of $f(n, m)$. For simplicity, we will focus on the case $m=n$ corresponding to Theorem \ref{thm:monochromatic-subset-sums-N}, where we wish to show that $f(n) = f(n, n) = \Theta\left(\frac{n^{1/3}(n/\phi(n))}{(\log n)^{1/3}(\log\log n)^{2/3}}\right)$. Theorem~\ref{thm:monochromatic-subset-sums} follows from an appropriate elaboration of these ideas.

We begin by sketching Vu's argument~\cite{V} (itself building on an argument used by Alon and Erd\H{o}s~\cite{AEr}), which yields the bound $f(n) \geq c_{1} \frac{n^{1/3}}{\log n}$ for some positive constant $c_1$. To this end, consider an arbitrary $r$-coloring of $[n]$ for some $r < c_{1} \frac{n^{1/3}}{\log n}$. We restrict our attention to the interval $[n^{2/3},2n^{2/3})$ and focus on the color class containing the largest number of primes from this interval. Let $Q$ be the set of primes in this color class, noting that $r < c_{1} \frac{n^{1/3}}{\log n}$ implies that $|Q|\ge Cn^{1/3}$ for a positive constant $C$ (which can be made arbitrarily large by taking $c_1$ to be sufficiently small). Partition $Q$ into three subsets $Q_{1}$, $Q_{2}$ and $Q_{3}$ of roughly equal size. Since $|Q_1| \geq \frac{C}{3} n^{1/3}$, we can apply the Szemer\'edi--Vu theorem, Theorem~\ref{SVthm}, to $Q_1$ to obtain an arithmetic progression of length at least $2n^{2/3}$ in $\Sigma(Q_{1})$. We can then complete this arithmetic progression of common difference $d$, say, to a long interval by building a complete modulo $d$ class using $Q_{2}$. Provided the parameters have been chosen appropriately, this interval will have length at least $2n^{2/3}$ and the minimum number in the interval will be smaller than $n$. Therefore, by Lemma~\ref{lem:verysimple}, adding each element of $Q_{3}$ in turn will expand the interval and, since adding all elements in $Q_{3}$ would exceed $n$, the resulting interval in $\Sigma(Q_{1}\cup Q_{2}\cup Q_{3})$ must contain $n$. 

To go further, we make two observations about this argument. First, note that we passed immediately to a subset of the primes. This was in order to avoid the situation where a color class consists entirely of numbers with a given divisor, as, otherwise, it would be impossible to write any $n$ which is not a multiple of this divisor as a sum of elements from the color class. Second, the key tool in the proof, Theorem~\ref{SVthm}, is tight up to the constant, since the set of subset sums of the set consisting of the first $\lfloor \sqrt{2n} - 1/2\rfloor$ positive integers has size less than $n$. However, this naive application of Theorem~\ref{SVthm} makes no use of the fact that our set consists entirely of primes. It is here that we are able to gain.

To illustrate the main ideas in our argument, we first restrict to the case where $n$ is prime. Suppose then that there is an $r$-coloring of $[n-1]$, where $r$ satisfies $n=\lambda r^{3}(\log r)(\log\log r)^{2}$ for a sufficiently large constant $\lambda$. If we let $y=\Theta(r^{2}(\log r)(\log\log r))$, the number of primes in the interval $[y, 2y)$ is $\Theta(r^{2} \log\log r)$, so, by the pigeonhole principle, there is a monochromatic subset $Q$ of the primes in $[y,2y)$ with $|Q| = \Omega(r \log \log r)$. As in Vu's argument, the plan from this point is to use a subset $V$ of $Q$ of size $O(r \log \log r)$ to build a large interval and then to apply Lemma \ref{lem:verysimple} to expand this interval using the remaining elements. To show that $\Sigma(V)$ contains the required interval, we partition $V$ into a bounded number of sets $V_{1},V_{2},\dots,V_{\ell}$ of roughly equal size and show that, for each $i$, $\Sigma(V_{i})$ contains a dense subset of an interval. Given this crucial input, Lemma~\ref{lem:Lev} then implies that $\Sigma(V)$ contains a long interval.

Quantitatively, for this argument to go through, we need $\Sigma(V)$ to contain an interval of length $\Omega(r^{2}(\log r)(\log\log r))$. For this to follow from Lemma~\ref{lem:Lev}, we need to have $|\Sigma(V_{i})| = \Omega(r^{2}(\log r)(\log\log r))$ for each $V_i$, themselves satisfying $|V_i| = O(r\log\log r)$. Thus, we need to show that $|\Sigma(V_{i})|/|V_{i}|\ge r\log r$, say. 
For this, we prove an inverse result, that if $|\Sigma(V_{i})|/|V_{i}| < r\log r$, then a large subset of $V_i$ must be additively structured, in the sense that this subset is contained in a set of size $O(|V_i|(\log r)/(\log \log r))$ 
which can be written as a union of long arithmetic progressions. We then use the Selberg sieve to show that, since $V_i$ consists of primes, it is impossible for a large subset of $V_i$ to have this structure. 

In practice, as in the proofs of Theorems~\ref{thm:Ramsey-complete}~and~\ref{thm:Ramsey-complete-polynomial}, we do much of our work over cyclic groups. Indeed, to show that $\Sigma(V_i)$ is large, we partition $V_{i}$ into two sets $A_{i,1}$ and $A_{i,2}$ and show that, for each $a_{2}\in A_{i,2}$, $|\Sigma(A_{i,1})\pmod {a_{2}}|$ is large. Lemma~\ref{lem:modp} then allows us to conclude that $\Sigma(V_i)=\Sigma(A_{i,1} \cup A_{i,2})$ is large.

To show that $|\Sigma(A_{i,1})\pmod {a_{2}}|$ is large, we consider an iterative building process which grows the set of subset sums modulo $a_2$ by picking elements in $A_{i,1}$ one at a time. We begin with $T_{0} = A_{i,1}$ and $\Sigma(0) = \{0\}\subseteq \mathbb{Z}_{a_2}$. In step $j\ge 1$, we choose an element $x_j$ from $T_{j-1}$ which maximizes $|(\Sigma(j-1)+x_j)\setminus \Sigma(j-1)|$, where the set $\Sigma(j-1)$ is viewed as a subset of $\mathbb{Z}_{a_2}$, and then set $\Sigma(j) = (\Sigma(j-1)+x_j)\cup \Sigma(j-1)$ and $T_j = T_{j-1}\setminus \{x_j\}$. If, for each $j \le |A_{i,1}|/2$, there is a choice of $x_j$ such that $|(\Sigma(j-1)+x_j)\setminus \Sigma(j-1)|$ is large, then $|\Sigma(A_{i,1})\pmod {a_{2}}|$ will be large, as required. If, instead, there is a step $j$ such that $|(\Sigma(j-1)+x)\setminus \Sigma(j-1)|$ is small for all $x \in T_{j-1}$, then, using Lemma~\ref{lem:Freiman-2.04} (or, rather, its corollary, Lemma~\ref{lem:structure}), we can show that $T_{j-1}$ is additively structured, in the sense that it is contained in a small set which is a union of long arithmetic progressions. By a version of the Selberg sieve, $T_{j-1}$ cannot then contain too many primes, contradicting the fact that, as a subset of $Q$, $T_{j-1}$ consists entirely of primes.

Several additional ideas are needed to handle the case where $n$ is not prime. For instance, in the prime case, we could build the required sum $n$ using only primes, but now we must use integers of the form $qu$, where $u$ is a small divisor of $n$ and $q$ is coprime to the first $r$ primes.
As before, our first step is to pass to a large monochromatic subset $Q_0$ of this set, the goal being to show that $n$ is contained in the set of subset sums of $Q_0$. In the prime case, we took a subset $V$ of $Q = Q_0$, partitioned it into sets $V_i$ and then partitioned each $V_i$ into sets $A_{i,1}$ and $A_{i,2}$, before showing that $|\Sigma(A_{i,1})\pmod {a_2}|$ is large for each $a_2\in A_{i,2}$.
However, this argument may not go through in the general case, because, when $a_2$ is not prime, we could have that $A_{i,1}$, and hence $\Sigma(A_{i,1})$, is contained in a small proper subgroup of $\mathbb{Z}_{a_{2}}$. 
 
To overcome this issue, we first apply a preprocessing step to the set $Q_0$, our aim being to find a closely related set $Q$ which is \emph{$k$-diverse}, by which we mean that, for any $d\ge 2$, there are at least $k$ elements of $Q$ which are not divisible by $d$. We obtain such a set through a simple iteration. Indeed, if we have a set which is not $k$-diverse, then there is some $d$ dividing all but $k$ elements of the set, so we can remove these elements from the set and divide the remaining elements by $d$ to form a new set. Repeating this procedure with an appropriate value of $k$, we eventually arrive at a large $k$-diverse set $Q$ such that $\{vx:x\in Q\}\subseteq Q_0$ for some $v|n$. Thus, in order to conclude that $n$ is a sum of elements in $Q_0$, we only need to show that $n/v$ is a sum of elements in $Q$. 

A crucial property of diverse sets is that random subsets of a diverse set are themselves diverse with high probability. Thus, by taking a random subset $V$ of $Q$, randomly partitioning $V$ into parts $V_i$ and then randomly partitioning each $V_i$ into $A_{i,1}$ and $A_{i,2}$, we have that, with high probability, all of the sets $A_{i,1}$ are diverse. We can also show that any common divisor of a large subset of $A_{i,1}$ must be a small divisor of $n$. Proceeding now along the same lines as the prime case, this reduces our task to showing that $|\Sigma(A_{i,1}) \pmod {a_2}|$ is large for any diverse subset $A_{i,1}$ of $A$ with the additional property that any common divisor of a large subset of $A_{i,1}$ is small. 

To show that $|\Sigma(A_{i,1}) \pmod {a_2}|$ is large, we consider a more refined version of the iterative building process used in the prime case. The details of this key step are contained in Lemma~\ref{lem:growth-}. We again begin with $T_{0} = A_{i,1}$ and $\Sigma(0) = \{0\}\subseteq \mathbb{Z}_{a_2}$ and, in step $j\ge 1$, we again choose an element $x_j$ from $T_{j-1}$ and set $\Sigma(j) = (\Sigma(j-1)+x_j)\cup \Sigma(j-1)$ and $T_j = T_{j-1}\setminus \{x_j\}$, but the process for choosing $x_j$ is more complex. To describe it, we let $d_j$ be the greatest common divisor of the elements in $T_{j-1}$. The choice of $x_j$ depends on the sets $S_u = \Sigma(j-1) \cap (u+d_j\mathbb{Z}_{a_2})$ with $u \in \mathbb{Z}_{a_2}/d_j\mathbb{Z}_{a_2}$. We refer to step $j$ as a {\it growth phase}, an {\it unsaturated phase} or a {\it saturated phase}, depending on whether there exists $u$ such that $S_u$ is non-empty and small, no non-empty $S_u$ is small and at least one is of intermediate size or all non-empty $S_u$ are large, respectively. If $j$ is a growth phase, we choose $x_j$ from $T_{j-1}$ so as to maximize $|\Sigma(d_j,j)|-|\Sigma(d_j,j-1)|$, where $\Sigma(d_j,t) =\{\sum_{h \in H} x_h \pmod{a_2} : H\subseteq [t] \cap \{h: d_j|x_h\}\}$. If $j$ is an unsaturated or saturated phase, we choose $x_j$ from $T_{j-1}$ so as to maximize $|(\Sigma(j-1)+x_j)\setminus \Sigma(j-1)|$. 

If now there is a saturated phase $j$ among the first $|A_{i,1}|/2$ steps, we can show that $|\Sigma(A_{i,1})\pmod {a_2}| \ge |\Sigma(j-1)|$ is large, as required. On the other hand, we can also show that there are only a small number of growth phases among the first $|A_{i,1}|/2$ steps. Hence, we can assume that there are many unsaturated phases. Our aim now is to show that $|\Sigma(j)|-|\Sigma(j-1)|$ is large for any unsaturated phase, since, together with the fact that there are many unsaturated phases, this will imply that $|\Sigma(A_{i,1})\pmod{a_2}|$ is large, as required. As in the prime case, this final step proceeds by first showing that if $|\Sigma(j)|-|\Sigma(j-1)|$ is not large, then $T_{j-1}$ must be additively structured, again that it is contained in a small set which is a union of long arithmetic progressions, and then using the Selberg sieve to derive a contradiction, in this case that $T_{j-1}$ cannot contain many elements of the form $qu$, where $u$ is a small divisor of $n$ and $q$ is coprime to the first $r$ primes. 

\subsection{Outline of the proof of Theorem~\ref{thm:Sze-Vu-2}}

To prove Theorem \ref{thm:Sze-Vu-2}, that there exists a constant $C$ such that any $A \subset [n]$ with $|A| \geq C\sqrt{n}$ has a homogeneous progression of length $n$ in $\Sigma(A)$, we use a variant of the ideas discussed in Subsection~\ref{subsec:monochromatic}. 
As in that subsection, we apply a preprocessing step to the set $A$ to find a set $A'$ of size comparable to $A$ which is $k$-diverse for an appropriate $k$ and for which there exists an integer $d$ such that $\{dx: x\in A'\} \subseteq A$. We also maintain a further property, that $A'$ intersects each dyadic interval in either the empty set or a large set. Having obtained the required set $A'$, we replace $A$ with this set and consider a random partition of the set into parts $X_1,Y_1,\dots,X_\ell,Y_\ell$. 

The key step in the proof is Lemma~\ref{lem:structure-1}, which roughly says that if $X_i$ satisfies an appropriate diversity condition, then $|\Sigma(X_i) \pmod b|$ is large for all $b\in Y_i$. But since $X_i$ is part of a random partition of the diverse set $A$, we can, with high probability, guarantee that $X_i$ is also diverse and, therefore, by Lemma~\ref{lem:structure-1}, that $|\Sigma(X_i) \pmod b|$ is large for all $b\in Y_i$. Then, as in the previous outlines, we apply Lemma~\ref{lem:modp}, in this case together with what we know about the distribution of $A$ in dyadic intervals, to show that $\Sigma(X_i\cup Y_i)$ is large, followed by Lemma~\ref{lem:Lev} to conclude that $\Sigma(A)$ contains a long interval. Unwinding the preprocessing step, we see that this interval corresponds to a long homogeneous arithmetic progression in the set of subset sums of the original set, as required.

At first glance, Lemma~\ref{lem:structure-1} seems to bear close resemblance to one of the key steps in the proofs of Theorems~\ref{thm:monochromatic-subset-sums-N} and \ref{thm:monochromatic-subset-sums} described in the previous subsection (and formally encapsulated in Lemma~\ref{lem:growth-}). In both cases, we wish to show that if $X$ is a sufficiently diverse set, then $|\Sigma(X) \pmod b|$ is large for all $b$ in a certain set $Y$. The difference lies in the fact that the sets $X$ considered in Theorems~\ref{thm:monochromatic-subset-sums-N} and \ref{thm:monochromatic-subset-sums} are carefully chosen so that we can hope for a stronger guarantee on the size of $\Sigma(X) \pmod b$ than in the typical case, whereas here we are concerned precisely with that typical case. The proof of Lemma~\ref{lem:structure-1} follows from a similar iterative building process to that used in the proof of Lemma~\ref{lem:growth-}, as described at the end of the last subsection. 

Because we need it for the proof of Theorem~\ref{thm:Alon-conjecture}, our result on the largest subset of $[n]$ avoiding a particular subset sum, we will actually prove a strengthening of Theorem~\ref{thm:Sze-Vu-2}, saying that we can build the required homogeneous progression using short sums, that is, sums with only a small number of terms. This strengthening requires a somewhat more careful analysis than that described above. In particular, we must start with $T_0$ equal to a large random subset of $X_i$ and $\Sigma(0) = X_i\setminus T_0 \pmod b$. 

\section{\label{sec:Ramsey-completeness}Ramsey completeness}

\subsection{Proof of the upper bound in Theorem \ref{thm:Ramsey-complete}}

The goal of this section is to prove the upper bound in Theorem \ref{thm:Ramsey-complete}, that there exists a constant $C$ such that, for every $r \geq 2$, there is an $r$-Ramsey complete sequence $A$ with $|A\cap [n]| \leq C r \log^2 n$ for all $n$. As shown in Section \ref{subsec:Ramsey-complete}, this theorem follows from another statement, Lemma \ref{lem:main-ramsey}, whose proof will occupy us in this subsection.

The next lemma, a mod $m$ analogue of Lemma~\ref{lem:main-ramsey}, is the key step in proving that lemma. Let $\Sigma_{m}(S)$ be the set of subset sums of $S$ taken modulo $m$.

\begin{lem}
\label{lem:cover} Fix $c \geq 6$ and assume that $x$ is sufficiently large. Let
$w=(\log x)/2$ and let $X$ be the set of integers in $[x,2x)$ with no
prime divisor at most $w$. Let $m \in X$. If a sequence $S$ of $c\log x$
integers is chosen uniformly and independently at random from $X$ and viewed
as a sequence of elements in $\mathbb{Z}_{m}$, then $|\Sigma_{m}(S)|<\frac{x}{4}$ with probability less than $\left(\log x\right)^{-(c-5)\log x}$.
\end{lem}
\begin{proof}
Let $W =\prod_{p\le w}p$, where the product is taken over primes, and $\tau = \phi(W)/W$. The prime number theorem implies that $W=e^{(1+o(1))w}=x^{1/2+o(1)}$. In any interval of length $W$, there are exactly $\tau W$ integers with no prime divisor at most $w$. By Merten's third theorem, $\tau=\left(e^{-\gamma}+o(1)\right)/\log w$, where $\gamma$ is the Euler--Mascheroni constant. It follows that $$|X|\ge \tau (x-W) \ge \frac{x}{2\log \log x}.$$ 

Let $q=c \log x$. Let $S=(s_{1},s_{2},\dots,s_{q})$ be a sequence of $q$ random elements of $X$. Let $S_i=(s_1,\ldots,s_i)$ denote the sequence consisting of the first $i$ elements of $S$. Let $\delta=\log\log x/\log x$. Call $i \in [2,q]$ {\it bad} if 
\begin{itemize} 
\item $|\Sigma_{m}(S_i)|\leq \frac{3}{2}|\Sigma_{m}(S_{i-1})|$ and $|\Sigma_{m}(S_{i-1})|\leq x/\log x$ or 
\item $|\Sigma_{m}(S_i)|\leq (1+\delta)|\Sigma_{m}(S_{i-1})|$ and $x/\log x < |\Sigma_{m}(S_{i-1})|<x/4$. 
\end{itemize}
The following two claims allow us to quickly complete the proof.

\vspace{2mm}

\noindent {\bf Claim 1.} The probability that $i$ is bad conditioned on the choice of $S_{i-1}$  is at most $p:=\frac{4(\log \log x)^2}{\log x}$.

\vspace{2mm}

\noindent {\bf Claim 2.} If $|\Sigma_m(S)|<x/4$, then the number of integers in $[2,q]$ which are not bad is less than $4\log x$. 

\vspace{2mm}

Assuming Claim 1, for any $B \subset [2,q]$, the probability that all elements in $B$ are bad is at most $p^{|B|}$. 
From Claim 2, if $|\Sigma_m(S)|<x/4$, then there is a set $B$ of $q-4\log x$ integers $i \in [2,q]$ which are bad. Taking a union bound over all such choices of $B$, the probability that $|\Sigma_m(S)|<x/4$ is at most 
\[{q \choose q-4\log x}p^{|B|}={q \choose 4\log x}p^{|B|}<c^{4\log x}\left(\frac{4(\log \log x)^2}{\log x}\right)^{(c-4)\log x}<(\log x)^{-(c-5)\log x}. \qedhere\]
\end{proof}

To complete the proof, it remains to verify Claims 1 and 2. 

\vspace{2mm}

\noindent {\it Proof of Claim 1.} Fix $S_{i-1}=(s_1,\ldots,s_{i-1})$. Conditioned on this choice of $S_{i-1}$, we bound the probability that $i$ is bad. If $|\Sigma_{m}(S_{i-1})| \geq x/4$, then $i$ cannot be bad (so the event that $i$ is bad has probability zero). We may therefore restrict attention to the two cases $|\Sigma_{m}(S_{i-1})| \leq x/\log x$ and $x/\log x <|\Sigma_{m}(S_{i-1})| < x/4$. 

For the first case, note, by Lemma \ref{lem:double-count}, that the number of $s$ with $|\Sigma_{m}(S_{i-1}\cup\{s\})|\le\frac{3}{2}|\Sigma_{m}(S_{i-1})|$ is at most $\frac{|\Sigma_{m}(S_{i-1})|^{2}}{|\Sigma_{m}(S_{i-1})|/2}=2|\Sigma_{m}(S_{i-1})|$. Therefore, if $|\Sigma_{m}(S_{i-1})| \leq x/\log x$, the probability that $i$ is bad conditioned on $S_{i-1}$ is at most $\frac{2|\Sigma_{m}(S_{i-1})|}{|X|}\le \frac{2x}{|X|\log x} \leq \frac{4\log \log x}{\log x}<p$.

Suppose now that $x/\log x <|\Sigma_{m}(S_{i-1})| < x/4$. For a positive integer $D$, let $G_{D}$ be the set of $s$ such that $|\Sigma_{m}(S_{i-1}\cup\{s\}) |\le |\Sigma_{m}(S_{i-1})|+D$. Let $d = \lfloor \delta |\Sigma_{m}(S_{i-1})|\rfloor$, so $i$ is bad in this case if and only if $s_i \in G_d$. Let $k=\lfloor \frac{1}{2\delta}\rfloor$, so $kd \leq |\Sigma_{m}(S_{i-1})|/2$. By Lemma~\ref{lem:stable-period}, $kG_{d}\subseteq G_{kd}$, so 
$\left|kG_{d}\right|\le|G_{kd}| \le 2|\Sigma_{m}(S_{i-1})| < \frac{x}{2}$, where the middle inequality is again by the consequence of Lemma \ref{lem:double-count} noted above. 

If  $|G_{d}|\le\frac{m}{w}$, then $|G_{d}|\le \frac{m}{w} \leq \frac{2x}{(\log x)/2} < 2\delta x$. 
Otherwise, $|G_{d}|>\frac{m}{w}$. In this case, since $m$ has no prime divisor
at most $w$, no subgroup of $\mathbb{Z}_{m}$ has size larger than
$\frac{m}{w}$. Thus, $G_{d}$ cannot be contained in a coset
of a non-trivial subgroup. By Lemma \ref{lem:Cauchy-Davenport}, since $\left|kG_{d}\right| \le \frac{x}{2} < m$, we must have 
$\left|kG_{d}\right|\ge (k+1)|G_{d}|/2 \geq |G_d|/(4\delta)$. Hence, $|G_{d}| \leq 4\delta|kG_d| \leq 4\delta x/2 =2\delta x$. Thus, in either case, conditioned on the choice of $S_{i-1}$, the probability that $i$ is bad, which is the same as the probability that $s_i \in G_d$, is at
most $\frac{|G_d|}{|X|} \leq \frac{2\delta x}{|X|} \le 4\delta \log \log x = p$. 
\qed  

\vspace{2mm}

\noindent {\it Proof of Claim 2.} As $S_{i-1}\subset S_i$ for $i\in[2,q]$, $\Sigma_m(S_{i-1}) \subset \Sigma_m(S_{i})$ and, hence, $1 \leq |\Sigma_m(S_1)| \leq \cdots \leq |\Sigma_m(S_q)|=|\Sigma_m(S)| < \frac{x}{4}$. Therefore, the number of $i$ which are not bad with $|\Sigma_m(S_{i-1})| \leq x/\log x$ and $|\Sigma_m(S_{i})| \geq \frac{3}{2}|\Sigma_m(S_{i-1})|$ is at most $\log_{3/2} x$, as we get a factor of $3/2$ for each such $i$. Moreover, since $(1+\delta)^{\delta^{-1} \log_2\log x} \geq 2^{\log_2 \log x} =\log x$, the number of elements $i$ which are not bad with $x/4>|\Sigma_m(S_{i-1})| >  x/\log x$ and $|\Sigma_m(S_i)| \geq (1+\delta)|\Sigma_m(S_{i-1})|$ is at most $\delta^{-1}\log_2 \log x = \log_2 x$, as we get a factor of $1+\delta$ for each such $i$. Therefore, the number of $i \in [2,q]$ which are not bad is at most $\log_{3/2} x+\log_2 x < 4\log x$. \qed

\vspace{2mm}

We next prove Lemma \ref{lem:main-ramsey} using Lemma \ref{lem:cover}. Let $C = 3840$, $\epsilon\in (0,1/2]$ and $X$ be the set
of integers in $[x,2x)$ with no prime divisor at most $(\log x)/2$, as in Lemma~\ref{lem:cover}. We wish to show that if a sequence $S$ of $C\epsilon^{-1}\log x$ elements in $X$ is chosen independently
and uniformly at random, then, with high probability, $S$ has distinct terms and, for any
subsequence $S'$ of $S$ of size $\epsilon|S|=C\log x$, the set $\Sigma(S')$
contains all integers in the interval $[\frac{Cx\log x}{4},\frac{7Cx\log x}{8}]$. 

\begin{proof}[Proof of Lemma \ref{lem:main-ramsey}]
By the birthday paradox, as $|S|=o(\sqrt{|X|})$, $S$ has distinct terms with high probability. Fix a choice of subset $I'$ of $[C\epsilon^{-1}\log x]$ of size $C \log x$ and consider the subsequence $S'$ of $S$ given by $(s_i)_{i\in I'}$. Let $I''$ be the smallest $|S'|/8$ elements in $I'$ and let $S''$ be given by $(s_i)_{i\in I''}$. Let
$\ell=40$. Arrange $I''$ in increasing order and partition $I''$ into $\ell$ sets $I''_1,\dots,I''_\ell$ of consecutive terms so that each set $I''_j$ for $j\in [\ell]$ has size $|I''|/\ell$. This gives a partition of $S''$ into $\ell$ subsequences $S''_{1},\dots,S''_{\ell}$, where $S''_j = (s_i)_{i\in I''_j}$. Note that $|S_{j}''|=\frac{C\log x}{8\ell}$
and $\Sigma(S''_{j})\subseteq[0,2x\frac{C\log x}{8\ell}]$. We shall
prove below that, with high probability, the sequence $S$ has the property that, for all possible choices of $I'$ and $j$, $|\Sigma(S''_{j})|\ge Cx\log x/64\ell$.
Assuming this, we can show that $\Sigma(S_{j}'')$ is not contained in
an arithmetic progression with common difference larger than $1$.
Indeed, if $\Sigma(S_{j}'')$ is contained in an arithmetic
progression with common difference $d>1$, then $d\le\frac{(2Cx\log x)/(8\ell)}{(Cx\log x)/(64\ell)-1} \leq 17$. 
Moreover, if $\Sigma(S_{j}'')$ is contained in an arithmetic progression
with common difference $d>1$, then all elements in $\Sigma(S_{j}'')$
are congruent modulo $d$, from which it follows that all elements
of $S_{j}''$ are divisible by $d$. This contradicts the fact that
no element of $S_{j}''$ has a prime factor at most $(\log x)/2 > 17
$. Hence, for each $j$, $\Sigma(S_{j}'')$ is not contained in an arithmetic progression with common difference larger than $1$. Therefore, by Lemma \ref{lem:Lev}, as $\Sigma(S'')=\Sigma(S''_{1})+\cdots+\Sigma(S''_{\ell})$, the set $\Sigma(S'')$ contains the integers in an interval of length at least $\ell\left(\frac{Cx\log x}{64\ell}-1\right)+1>2x$.
Finally, by Lemma \ref{lem:verysimple}, $\Sigma(S') = \Sigma(S''\cup(S'\setminus S''))$ contains all integers in the interval $[\frac{Cx\log x}{4},\frac{7Cx\log x}{8}]$, where we used that all elements of $S'$ are at most $2x$, the elements of $\Sigma(S'')$ are at most $2x|S''|=\frac{Cx\log x}{4}$ and the sum of the elements in $S'\setminus S''$ is at least $\frac{7}{8}|S'|x = \frac{7Cx\log x}{8}$.

It remains to show that, with high probability, the sequence
$S$ has the property that, for all possible choices of $I'$ and $j$, $|\Sigma(S''_{j})|\ge Cx\log x/64\ell$. Fix
an index $1\le j\le \ell$ and partition the index set $I''_{j}$ of $S''_{j}$ into two consecutive blocks $J_1$ and $J_2$ of equal size. Let $P_1 = (s_i)_{i\in J_1}$ and $P_2 = (s_i)_{i\in J_2}$, so $|P_{1}|=|P_{2}|=\frac{|S_{j}''|}{2}=\frac{C\log x}{640}=c\log x$
for $c=\frac{C}{640}=6$. Recall that $X$ is the
set of integers in $[x,2x)$ with no prime divisor at most $(\log x)/2$.
Consider $m\in X$. We note that when we fix the subset of indices $I'$ of $[C\epsilon^{-1}\log x]$ of size $C\log x$ and the index $j$, then $J_{1}$ is determined as a particular subsequence of $I'$.
Moreover, each element in $P_{1}$ is uniformly and independently 
distributed in $X$. 
Taking a union bound over all $x \ell \binom{C \epsilon^{-1} \log x}{C\log x}$ choices of $I'$, $j \in [\ell]$ and $m \in X$, Lemma \ref{lem:cover} implies that the probability $|\Sigma_{m}(P_{1})|<\frac{x}{4}$ for some $I'$, $j \in [\ell]$ and $m\in X$ is at most
\[x\ell\binom{C \epsilon^{-1} \log x}{C\log x}\cdot\left(\frac{1}{\log x}\right)^{(c-5)(\log x)} \le 50x(e/\epsilon)^{3840\log x}\cdot\left(\frac{1}{\log x}\right)^{\log x}=o_x(1),
\]
where $o_x(1)$ tends to $0$ as $x$ tends to infinity. Thus, with high probability, 
the sequence $S$ is such that $|\Sigma_{m}(P_{1})|\ge\frac{x}{4}$ for all choices of $I'$,
$j\in
[\ell]$ and $m\in X$. In this case, by repeated application of Lemma \ref{lem:modp}, for all $j \in [\ell]
$,
\[
|\Sigma(S_{j}'')|\ge\sum_{m\in P_{2}}\frac{x}{4} = \frac{x}{4} \cdot \frac{|S_{j}''|}{2}=\frac{Cx\log x}{64\ell}.
\]
Therefore, with high probability, the sequence
$S$ is such that $|\Sigma(S''_{j})|\ge Cx\log x/64\ell$ for all possible choices of $I'$ and
$j$, as required. 
\end{proof}

\subsection{Proof of Theorem \ref{thm:Ramsey-complete-polynomial}}

Our aim in this section is to prove Theorem \ref{thm:Ramsey-complete-polynomial}, our main result on the Ramsey completeness of complete polynomial sequences $(P(m))_{m \geq 1}$, saying that there exists a constant $C$, depending only on the degree of $P$, such that, for every $r\ge2$, there is an $r$-Ramsey complete sequence $A\subset(P(m))_{m\ge1}$ with $|A\cap[n]|\le C r\log^{2}n$ for all $n$. As remarked in Section \ref{subsec:Ramsey-complete}, we can and will assume that $P$ is a complete polynomial which has integer coefficients in its binomial representation. That is, we can write $P(x)=\sum_{i=0}^{k}\alpha_{i}\binom{x}{i}$,
with $\alpha_{k}>0$, each $\alpha_{i}$ an integer and $\gcd(\alpha_{0},\dots,\alpha_{k})=1$.

Our first goal will be to prove Lemma~\ref{lem:iterated-sum-growth}. To recall the statement, suppose that $P$ is a complete polynomial of degree $k$ with integer coefficients in its binomial representation, $m$ is an integer in $[x,2x)$, $(\log x)^{-1} < \alpha < 1/2$ and $T$ is a subset of $[x,2x)$ of size at least $\alpha x$. Then Lemma~\ref{lem:iterated-sum-growth} asserts that there is a constant $C_k$ depending only on $k$ such that, for $x$ sufficiently large, the iterated sumset $2^{k-1}P(T)-2^{k-1}P(T)$ contains more than $\alpha^{C_k}P(m)$ residue classes modulo $P(m)$. Once this lemma is in place, we will follow a scheme similar to that of the previous subsection to complete the proof. 

\begin{proof}[Proof of Lemma \ref{lem:iterated-sum-growth}]

Let $T=\{x_{0},x_{1},\dots,x_{\ell-1}\}$, where $x\le x_{0}<x_{1}<\dots<x_{\ell-1} < 2x$
and $\ell\ge \alpha x$. Let $x_{0,i}=x_{i}$ and $\ell_0=\ell$. Let $\ell_j=\ell_{j-1}(\ell_{j-1}-1)
/4x$ for $j =1, \dots, k$. For each $j \in [k]$, we recursively construct a subsequence $x_{j, 0} < x_{j,1} < \dots < x_{j, \ell_j - 1}$  of $T$ with $\ell_j$ terms, as follows.
For each $j\in [k]$, note that at least $\left(\ell_{j-1}-1\right)/2$ of the indices $0 \le i\le \ell_{j-1}-2$
satisfy $x_{j-1,i+1}-x_{j-1,i}\le2x/\ell_{j-1}$. Thus, by the pigeonhole principle, there is $y_j \in [2x/\ell_{j-1}]$ such that at 
least $\left((\ell_{j-1}-1)/2)\right/\left(2x/\ell_{j-1}\right)=\ell_{j}$ indices $0 \le i\le \ell_{j-1}-2$ satisfy $x_{j-1,i+1}-x_{j-1,i}=y_{j}$.
Let $x_{j,i}=x_{j-1,t_{i}}$ for $\ell_{j}$ increasing indices $t_0, t_1, \dots, t_{\ell_j - 1}$ such that $x_{j-1,t_{i}+1}-x_{j-1,t_{i}}=y_{j}$.
As $x/\left(\ell_{j}+1\right) \le \left(2x/\left(\ell_{j-1}+1\right)\right)^{2}$, by iterating we get 
\begin{equation}\label{ljinequal}
\frac{x}{\ell_{j}+1} \leq
 \left(\frac{x}{\ell_{0}+1}\right)^{2^{j}}2^{2(1+2+\dots+2^{j-1})}\le \alpha^{-2^{j}}2^{2^{j+1}}.
\end{equation}
In particular, by (\ref{ljinequal}) and the assumption $\alpha \ge (\log x)^{-1}$, we obtain that, for $1\le j\le k$, $y_j$ is bounded above by a polynomial function of $\log x$ depending on $k$. 

Let $P_{0}=P$ and recursively define 
\[
P_{j}(x)=P_{j-1}(x+y_{j})-P_{j-1}(x),
\]
which is a polynomial in $x$ of degree $k-j$ whose coefficients are polynomials in $y_1,\dots,y_j$. Let $z_j:=\prod_{i=1}^j y_i$. Then $z_j$ and the coefficients of $P_{j}$ are bounded in absolute value by a polynomial function of $\log x$ which depends on $k$ and the coefficients of $P$. This observation brings the following simple claim into play. 

\vspace{2mm}
\noindent {\bf Claim.} Let $Q(x) = \sum_{i=0}^k \beta_i x^i$ and $\tilde{Q}(x) = \sum_{i=0}^k \tilde{\beta}_i x^i$, where $\beta_i$ and $\tilde \beta_i$ are allowed to depend on $x$. If the $\beta_i$ and $\tilde \beta_i$ are at most a fixed polynomial function of $\log x$ in absolute value and $\beta_k=\tilde \beta_k$ is bounded below in absolute value by some positive constant depending only on $k$, then $\lim_{x\to \infty} \frac{Q(x)}{\tilde{Q}(x)}=1$. 

\vspace{2mm}

Recall that $P$ is a complete polynomial with integer coefficients in its binomial representation $P(x)=\sum_{i=0}^k \alpha_i {x \choose i}$ and the leading coefficient $\alpha_k$ is a positive integer. The coefficient of $x^{k-j}$ in $P_j(x)$ is the same as that in $\alpha_k z_j {x \choose k-j}$. To see this, note, by induction, that the coefficient of $x^{k-j}$ in $P_j(x) = P_{j-1}(x+y_{j})-P_{j-1}(x)$ is the same as the coefficient of $x^{k-j}$ in $\alpha_k z_{j-1} \left({x+y_j \choose k-j +1}-{x \choose k-j+1}\right)$ and, hence, of $\alpha_k z_{j-1} y_j {x \choose k-j} = \alpha_k z_j {x \choose k-j}$. It follows from the claim that the polynomial $P_{j}(x)$ is asymptotically equal to $\alpha_k z_j {x \choose k-j}$. 

Let $c=1/(k2^{k+2})$ and $w_{k-1}=1$. For $0 \leq j \leq k-2$, let $w_j=2^{k-j}y_{j+1}$. We choose (not necessarily disjoint) sets $I_{0},I_{1},\dots,I_{k-1}$ of indices such that $I_j \subseteq [\ell_{j}]$ and any two distinct indices in $I_j$ differ by at least $w_j$. By partitioning $[x,2x)
$ into $1/c$ intervals of length $cx$ each, we can further guarantee that $\{x_{0,i_0}:i_0\in I_0\}$ is a subset of an interval $[x',x'+cx)$ of length $cx$ that is a subinterval of $[x,2x)$. By greedily picking the elements, we can guarantee that $|I_0| \geq c\ell_0/w_0$ and $|I_j| \geq \ell_j/w_j$ for $j>0$. 

For a $k$-tuple $t=(i_{0},\dots,i_{k-1})\in I_{0}\times \cdots\times I_{k-1}$, let 
\[
F(t) = \sum_{j=0}^{k-1}P_{j}(x_{j,i_{j}}). 
\]
We claim that these numbers are distinct modulo $P(m)$. This follows from showing that (as integers) these numbers lie in an interval of length less than $P(m)$ and that they are ordered lexicographically. That is, if $t=(i_j)$ and $t'=(i'_j)$ are distinct $k$-tuples, $j_0$ is the smallest index such that $i_{j_0} \ne i'_{j_0}$ and $i_{j_0}>i'_{j_0}$, then $F(t)>F(t')$. 

We first show that the numbers $F(t)$ with $t\in I_{0}\times\cdots\times I_{k-1}$ lie in an interval of length less than $P(m)$.  
As $x$ is sufficiently large, each $P_j$ is positive and increasing in $[x,2x)$. It follows that  
\begin{equation}
P(x')+\sum_{1\le j\le k-1}P_{j}(x) \leq F(t) \leq P(x'+cx)+\sum_{1\le j\le k-1}P_{j}(2x).\label{eq:bound-F}
\end{equation} We have that 
\begin{align*}
P(x'+cx)-P(x') &\leq \alpha_k \left(\binom{x'+cx}{k}-\binom{x'}{k}\right) + \sum_{j<k} |\alpha_j| (x'+cx)^j \\
& \le  \alpha_k \left(\binom{2x}{k}-\binom{2x-cx}{k}\right) + \sum_{j<k} |\alpha_j| (2x)^j \\
&= (2^k-(2-c)^k)P(x) + R(x),
\end{align*}where $R$ is a polynomial with degree at most $k-1$ depending only on $P$. 
Thus, the difference between the upper and lower bounds for $F(t)$ in (\ref{eq:bound-F}) is, for $x$ sufficiently large, at most
\begin{align*}
&P(x'+cx)-P(x')+\sum_{1\le j\le k-1}P_{j}(2x) \\
&\leq (2^k-(2-c)^k)P(x) + R(x) + \sum_{1\le j\le k-1}P_{j}(2x) \\
&\leq ck2^{k-1}P(x) \leq \frac{1}{2}P(x) \leq \frac{1}{2}P(m),
\end{align*}
where, in the second inequality, we used that $2^k-(2-c)^k < ck2^{k-1}$, as well as the claim and the fact that $R(x) + \sum_{1\le j\le k-1}P_{j}(2x)$ is a polynomial of degree at most $k-1$ in $x$ whose coefficients are polynomials (depending only on $P$) in $y_1,\dots,y_{k-1}$, where $y_1,\dots,y_{k-1}$ are themselves bounded in absolute value by a polynomial function of $\log x$. Hence, the integers $F(t)$ all lie in an interval of length at most $P(m)/2 < P(m)$, as desired. 

We next show that the integers $F(t)$ with $t\in I_{0}\times\cdots\times I_{k-1}$ are lexicographically ordered. Indeed, suppose $t=(i_j)$ and $t'=(i'_j)$ are distinct $k$-tuples, $j_0$ is the smallest index such that $i_{j_0} \ne i'_{j_0}$ and $i_{j_0}>i'_{j_0}$. Then  
\begin{equation}\label{insum} F(t)-F(t')= \sum_{j=j_0}^{k-1}P_{j}(x_{j,i_j})-P_{j}(x_{j,i'_j}).\end{equation}
Since $x_{j,i_j}-x_{j,i'_j} \geq i_j-i'_j \geq w_j$, the first summand in (\ref{insum}), when $j=j_0$, is asymptotically at least 
$\alpha_k z_{j_0} w_{j_0} {x \choose k-j_0-1}$. If $j_0=k-1$, the rest of the sum is $0$. Otherwise, $j_0 \leq k-2$ and, since $x\leq x_{j,i_j},x'_{j,i'_j} < 2x$ and $P_j(x)$ is increasing for $x$ sufficiently large, the rest of the sum in (\ref{insum}) is at least 
\[
\sum_{j=j_0+1}^{k-1}P_{j}(x)-P_{j}(2x).
\]
By the claim, this sum is asymptotic to its first summand (when $j=j_0+1$). Therefore, this sum is asymptotically $-\alpha_k z_{j_0+1}(2^{k-j_0-1}-1) {x \choose k-j_0-1}$. As $z_{j_0+1}=y_{j_0+1}z_{j_0}$, we have $z_{j_0}w_{j_0} > 2z_{j_0+1}(2^{k-j_0-1}-1)$. Hence, as $x$ is sufficiently large, the first term in the sum in (\ref{insum}) is more than the absolute value of the sum of the other terms, so we conclude that $F(t)>F(t')$, as desired. 

As the integers $F(t)$ with $t \in I_0 \times \cdots \times I_{k-1}$ are distinct modulo $P(m)$, the number of distinct residue classes $F(t) \pmod{P(m)}$ is at least 
$$\prod_{j=0}^{k-1} |I_j| \geq c\prod_{j=0}^{k-1} \ell_j/w_j \geq c_kx^k \prod_{j=0}^{k-1} \alpha^{2^{j+1}} \geq c_k\alpha^{2^{k+1}}x^k,$$
where $c_k>0$ depends only on $k$. Here we used $y_j \leq 2x/\ell_{j-1}$, $w_j = 2^{k-j}y_{j+1}$ by the definition of $w_j$ and the bound (\ref{ljinequal}) on $\ell_j$. 

Note now that $P_0(x_{0,i})=P(x_i) \in P(T)$. We will show, inductively, that for $j\ge 1$ we have $P_{j}(x_{j,i}) \in 2^{j-1}P(T)-2^{j-1}P(T)$ for all $0 \le i \le \ell_{j}-1$. Indeed, 
\[
P_{j}(x_{j,i}) = P_{j-1}(x_{j,i}+y_{j})-P_{j-1}(x_{j,i}) = P_
{j-1}(x_{j-1,t_i+1})-P_{j-1}(x_{j-1,t_i}) \in 2^{j-1}P(T)-2^{j-1}P(T),
\] recalling that there exist indices $t_i$ such that $x_{j,i}=x_{j-1,t_i}$ and $x_{j,i}+y_{j} = x_{j-1,t_i+1}$. As each $F(t)$ is the sum of $k$ terms in which the $j^{\textrm{th}}$ term is of the form $P_j(x_{j,i})$, we have that each $F(t)$ is in the set $P(T)+\sum_{j=1}^{k-1} (2^{j-1}P(T)-2^{j-1}P(T))=2^{k-1}P(T)-(2^{k-1}-1)P(T)$. The set $2^{k-1}P(T)- 2^{k-1}P(T) = - P(T) + 2^{k-1}P(T)-(2^{k-1}-1)P(T)$ is the union of $|-P(T)|$ translates of $2^{k-1}P(T)-(2^{k-1}-1)P(T)$. Hence,
\[
\left|2^{k-1}P(T)-2^{k-1}P(T)
\right|\ge c_k\alpha^{2^{k+1}}x^{k} > \alpha^{C_k}P(m)
\]
for an appropriate constant $C_k$ depending only on $k$, completing the proof. 
\end{proof}

\noindent {\it Remark.} A Hilbert cube of dimension $k$ (or simply a $k$-cube) is a set $H(a_0,e_1,\dots,e_k)$ of the form $\{a_0+\sum_{i \in I} e_i: I \subseteq [k]\}$ with $a_0$ an integer and $e_1,\ldots,e_k$ positive integers (see~\cite{GuRo98} for more on the long history of these objects). The first step in the proof of Lemma~\ref{lem:iterated-sum-growth} was to iteratively build many Hilbert cubes of dimension $k$ consisting of elements of $T$, all with $e_j=y_j$ and where we can take $a_0$ to be any $x_{k,i}$. 
An alternative approach to this step is to build many $k$-cubes in $T$ with small $e_1,\dots,e_k$ and then to use the pigeonhole principle to show that one can pick out many such $k$-cubes with the same $e_1,\dots,e_k$. 

\vspace{3mm}

As in the previous subsection, we will deduce Lemma \ref{lem:main-poly-ramsey} from a modular analogue, which we now state.
Recall that $\Sigma_{m}(S)$ is the set of subset sums modulo $m$. 

\begin{lem}
\label{lem:cover-poly}Let $P$ be a complete
polynomial of degree $k$ with integer coefficients in its binomial
representation. Fix $c\geq k2^{k+4}$ and assume $x$ is sufficiently large (depending on $P$). Let $w=(\log x)^{1/2}$ and let $X$ be the set of $y \in [x,(1+1/k)x)$ such that $P(y)$ has no prime divisor at most $w$. Let $m\in X$. If $S=(s_1,\dots,s_{q})$ is a sequence of $q =c\log x$ elements 
chosen uniformly and independently at random from $X$ and the sequence $P(S)=(P(s_1),\dots,P(s_{q}))$ is viewed as
a sequence of elements in $\mathbb{Z}_{P(m)}$, then $|\Sigma_{P(m)}(P(S))|< P(m)/4$ with probability
at most $\left(\log x\right)^{-(c-k2^{k+3})(\log x)/(8C_k)}$, where $C_k$ is the constant defined in Lemma \ref{lem:iterated-sum-growth}. 
\end{lem}

We will need the following estimate for the proof of Lemma \ref{lem:cover-poly}. 

\begin{lem} \label{lem:estimate-nondiv}
For each positive integer $k$, there is $c_k>0$ such that the following holds. Suppose $P$ is a complete polynomial of degree $k$ with integer coefficients in its binomial representation. If $x$ is sufficiently large and $1<w<(\log x)/2$ is an integer, then the set $X$ of $y \in [x,(1+1/k)x)$ such that $P(y)$ has
no prime divisor at most $w$ satisfies $|X|\ge c_{k}(\log w)^{-k}x$.
\end{lem}

\begin{proof}
For each prime $q\le k$, let $v_q$ be the largest integer $v$ such that $q^v \mid k!$. For $i\le k$, we have $\prod_{j=0}^{i-1}(q^{2v_q}+(x-j)) = \prod_{j=0}^{i-1}(x-j) + q^{2v_q}z$ for some integer $z$, so 
\[
\binom{x+q^{2v_q}}{i}-\binom{x}{i}=\frac{1}{i!}\prod_{j=0}^{i-1}(q^{2v_q}+(x-j)) - \frac{1}{i!}\prod_{j=0}^{i-1}(x-j) = \frac{q^{2v_q}z}{i!}.
\]
Letting $v_{q,i}$ and $r_i$ be integers such that $i! = q^{v_{q,i}} r_i$ and $\gcd(r_i,q)=1$, we have, since $\binom{x+q^{2v_q}}{i}-\binom{x}{i}$ is an integer, that $r_i \mid z$. Moreover, $v_{q,i} \le v_q$, since $i! \mid k!$. Hence, 
$$\binom{x+q^{2v_q}}{i}-\binom{x}{i}=\frac{q^{2v_q}z}{i!} = q^{v_q} \cdot q^{v_q-v_{q,i}} \frac{z}{r_i} \equiv 0 \pmod{q^{v_q}}.$$
That is, $\binom{x}{i}\pmod{q^{v_q}}$ is periodic every $q^{2v_q}$ and, therefore, $P(x) \pmod{q}$ is periodic every $q^{2v_q}$. Since $P$ is complete, for each prime $q\le k$, there exists an integer $x \in [1,q^{2v_q}]$ such that $P(x)$ is coprime to $q$. Using that $\prod_{q \le k, \textrm{ $q$ prime}}q^{2v_q}=k!^2$, we have, by the Chinese Remainder Theorem, that there exists an integer $y \in [1,k!^2]$ such that $P(y)$ is coprime to all primes $q\le k$. We also have that $P(x) \pmod{q}$ is periodic every $k!^2$ for all primes $q\le k$. Moreover, for each prime $q>k \geq i$, $\binom{x}{i} \pmod{q}$ is periodic every $q$. Therefore, letting
$R_{k,w}=k!^{2}\prod_{k<q\le w,\,q\textrm{ prime}}q$, we have that $P(x) \pmod{\prod_{q\le w,\,q\textrm{ prime}}q}$ is periodic every $R_{k,w}$.

Let $W$ be the set of positive integers $y$ at most $R_{k,w}$ such that
$P(y)$ is coprime to all primes at most $w$. Let $m_{P,q}$ be the
number of roots of $P(x)\pmod q$, which is at most $k$ for each prime $q$,
where we used that $P$ has degree $k$ and is nonzero modulo $q$ by completeness. By the Chinese Remainder Theorem, the fraction of $y\in[\prod_{k<q\le w,\,q\textrm{ prime}}q]$
such that $P(y)$ is coprime to $\prod_{k<q\le w,\,q\textrm{ prime}}q$ is then
\[
\prod_{k<q\le w,\,q\textrm{ prime}}\left(1-\frac{m_{P,q}}{q}\right)\ge c'_k(\log w)^{-k}
\]
for some constant $c'_k>0$, where we used the bound $1-\varepsilon>e^{-\varepsilon-\varepsilon^2}$ for $0<\varepsilon<1/2$ and Merten's second theorem, which implies that $\sum_{q\le w,\,q\textrm{ prime}} 1/q = \log \log w +O(1)$. Furthermore, as shown above, there exists $y \in [k!^2]$ such that $P(y)$ is coprime to $\prod_{q\le k,\,q\textrm{ prime}}q$. Since $P(y) \pmod{\prod_{q\le k,\,q\textrm{ prime}}q}$ is periodic every $k!^2$ and $\gcd(k!^2,\prod_{k<q\le w,\,q\textrm{ prime}}q)=1$, the Chinese Remainder Theorem implies that the fraction of $y\in[R_{k,w}]$ such that $P(y)$ is coprime to $R_{k,w}$ is
at least $(c'_k/k!^2)(\log w)^{-k}$. Hence, $|W|\ge (c'_k/k!^2)(\log w)^{-k}R_{k,w}$. Since the integers $y$ for which $P(y)$ has no prime factor at most $w$ are periodic every $R_{k,w}$ and $R_{k,w} \leq x^{1/2 + o(1)}$ by the assumption $w \leq (\log x)/2$, we have that $|X| \geq (|W|/R_{k,w})(x/k)-|W| \geq  c_{k}(\log w)^{-k} x$ for an appropriate $c_k>0$ depending only on $k$, as required. 
\end{proof}

The proof of Lemma \ref{lem:cover-poly} now proceeds along broadly similar lines to the proof of Lemma \ref{lem:cover}.

\begin{proof}[Proof of Lemma \ref{lem:cover-poly}]
Let $S_i=(s_1,\ldots,s_i)$ denote the sequence consisting of the first $i$ terms of $S$ and let $T_{i}=(P(s_{1}),P(s_{2}),\dots,P(s_{i}))$. We also write $T$ as a shorthand for $T_q = P(S)$. Call $i\in [2,q]$ {\it bad} if 
\begin{itemize}
\item $|\Sigma_{P(m)}(T_i)|\le (1+2^{-k-1})|\Sigma_{P(m)}(T_{i-1})|$ and $|\Sigma_{P(m)}(T_{i-1})| < \frac{P(m)}{2w}$ 
or 
\item $|\Sigma_{P(m)}(T_i)|\le \left(1+\frac{1}{2^{k+1}w}\right)|\Sigma_{P(m)}(T_{i-1})|$ and $\frac{P(m)}{2w}\le |\Sigma_{P(m)}(T_{i-1})| < \frac{P(m)}{4}$.
\end{itemize} 

The following claims are the key components in the proof. Here $C_k$ is the constant from Lemma \ref{lem:iterated-sum-growth}. 

\vspace{2mm}

\noindent {\bf Claim 1.} The probability that $i$ is bad conditioned on the choice of $S_{i-1}$ is at most $p:=(\log x)^{-1/(4C_k)}$. 

\vspace{2mm}

\noindent {\bf Claim 2.} If $|\Sigma_{P(m)}(T)| < P(m)/4$, then the number of integers in $[2,q]$ which are not bad is less than $k2^{k+3}\log x$. 
 
 \vspace{2mm}
 
By Claim 1, for any $B\subset [2,q]$, the probability that all elements in $B$ are bad is at most $p^{|B|}$. By Claim~2, if $|\Sigma_{P(m)}(T)| < P(m)/4$, then there is a set $B$ of $q-k2^{k+3}\log x = (c-k2^{k+3})\log x$ integers $i\in [2,q]$ which are bad. Taking a union bound over all choices of $B$, the probability that $|\Sigma_{P(m)}(T)|<P(m)/4$ is at most 
\begin{align*}
\binom{q}{q-k2^{k+3}\log x}p^{|B|} & = \binom{q}{k2^{k+3}\log x}p^{|B|} \\ & \le (ec)^{k2^{k+3}\log x} (\log x)^{-(c-k2^{k+3})(\log x)/(4C_k)} \\ & < (\log x)^{-(c-k2^{k+3})(\log x)/(8C_k)}. \qedhere
\end{align*}
\end{proof}

Therefore, in order to complete the proof of the lemma, it suffices to prove Claims 1 and 2. It is here, in the proof of Claim 1, that Lemma~\ref{lem:iterated-sum-growth} comes into play.

\vspace{2mm}

\noindent {\it Proof of Claim 1.} Fix $S_{i-1}=(s_1,\dots,s_{i-1})$. Conditioned on this choice of $S_{i-1}$, we bound the probability that $i$ is bad. If $|\Sigma_{P(m)}(T_{i-1})|\ge P(m)/4$, then $i$ cannot be bad (so the probability that $i$ is bad is zero). The proof now splits into two cases, when $|\Sigma_{P(m)}(T_{i-1})|< \frac{P(m)}{2w}$ and when $\frac{P(m)}{2w}\le |\Sigma_{P(m)}(T_{i-1})|<P(m)/4$. 

\vspace{2mm}

\noindent {\bf Case 1.} $|\Sigma_{P(m)}(T_{i-1})|< \frac{P(m)}{2w}$. 

\vspace{2mm}

Let 
\[
V=\left\{ t\in[x,(1+1/k)x):\,|\Sigma_{P(m)}(T_{i-1}\cup\{P(t)\})|\le\left(1+\frac{1}{2^{k+1}}\right)|\Sigma_{P(m)}(T_{i-1})|\right\} .
\] Observe that $i$ is bad conditioned on $S_{i-1}$ if and only if $s_i \in V$. We will show that $|V| \leq \alpha x$, where $\alpha = w^{-1/C_k}$ and $C_k$ is again the constant from Lemma~\ref{lem:iterated-sum-growth}. 

Suppose, for the sake of contradiction, that $|V| > \alpha x$. Lemma \ref{lem:iterated-sum-growth} then implies that 
$$|2^{k-1} P(V)-2^{k-1} P(V)| > \alpha^{C_k} P(m) = P(m)/w,$$ 
where $P(V)=\{P(v):v\in V\}$. Note now that if $U\subseteq \mathbb{Z}_{P(m)}$ and $u\in \mathbb{Z}_{P(m)}$, then $|\Sigma_{P(m)}(U\cup \{u\})| =|\Sigma_{P(m)}(U\cup \{-u\})|
$. Thus, for each $z\in P(V)\cup (-P(V))$, we have 
$$|\Sigma_{P(m)}(T_{i-1}\cup\{z\})| \leq\left(1+\frac{1}{2^{k+1}}\right)|\Sigma_{P(m)}(T_{i-1})|$$
and Lemma \ref{lem:stable-period} implies that, for each $y\in 2^{k-1} P(V)-2^{k-1} P(V)$, 
\begin{equation}
|\Sigma_{P(m)}(T_{i-1}\cup\{y\})|\le \left(1+\frac{2^{k}}{2^{k+1}}\right)|\Sigma_{P(m)}(T_{i-1})| = \frac{3}{2} |\Sigma_{P(m)}(T_{i-1})|.\label{eq:boundTi-1y}
\end{equation} 
However, by Lemma~\ref{lem:double-count}, the number of $y \in \mathbb{Z}_{P(m)}$
satisfying (\ref{eq:boundTi-1y})
is at most $2|\Sigma_{P(m)}(T_{i-1})| < P(m)/w$. But this contradicts the bound $|2^{k-1}P(V)-2^{k-1}P(V)| > P(m)/w$, so we must indeed have that $|V| \leq \alpha x$. 

\vspace{2mm}

\noindent {\bf Case 2.} $P(m)/2w \le |\Sigma_{P(m)}(T_{i-1})|<P(m)/4$. 

\vspace{2mm}

Let 
\[
V =\left\{ t\in[x,(1+1/k)
x):|\Sigma_{P(m)}(T_{i-1}\cup\{P(t)\})|\le\left(1+\frac{1}{2^{k+1}w}\right)|\Sigma_{P(m)}(T_{i-1})|\right\} .
\] 
Observe again that $i$ is bad conditioned on $S_{i-1}$ if and only if $s_i \in V$.
As in Case 1, we will show that $|V| \leq \alpha x$. Indeed,
suppose, for the sake of contradiction, that $|V| > \alpha x$. Then,
by Lemma \ref{lem:iterated-sum-growth}, we again have that
$|2^{k-1}P(V)-2^{k-1}P(V)| > P(m)/w$.
By our assumption that $P(m)$ has no prime divisor at most $w$,
$2^{k-1}P(V)-2^{k-1}P(V)$ cannot be contained in a coset of a proper subgroup of
$\mathbb{Z}_{P(m)}$. Hence, by Lemma \ref{lem:Cauchy-Davenport},
\[
|2^{k-1}wP(V)-2^{k-1}wP(V)| > P(m)/2.
\]
However, again using Lemma \ref{lem:stable-period}, for all elements $y \in 2^{k-1}wP(V)-2^{k-1}wP(V)$, we have 
\[
|\Sigma_{P(m)}(T_{i-1}\cup\{y\})|\le \left(1+\frac{2^{k}w}{2^{k+1}w}\right)|\Sigma_{P(m)}(T_{i-1})|=\frac{3}{2}|\Sigma_{P(m)}(T_{i-1})|.
\]
But, by Lemma \ref{lem:double-count}, the number of such elements is at
most $2|\Sigma_{P(m)}(T_{i-1})|<P(m)/2$, a contradiction. 

\vspace{2mm}

Therefore, in either case, the set $V$ of bad choices satisfies $|V| \leq \alpha x$. By using Lemma \ref{lem:estimate-nondiv}, which says that 
$|X|\ge c_{k}(\log w)^{-k}x$ for an appropriate $c_k>0$,
this implies that the probability $i$ is bad conditioned on the choice of $S_{i-1}$ is at most 
$$|V|/|X| \leq \alpha x/\left(c_k(\log w)^{-k}x\right)= \alpha c_k^{-1}(\log w)^{k} = c_k^{-1} (\log x)^{-1/(2C_k)} \left(\frac{1}{2}\log \log x\right)^k <(\log x)^{-1/(4C_k)}=p,$$
as required.
\qed

\vspace{2mm}

\noindent {\it Proof of Claim 2.} 
As $S_{i-1}\subset S_i$ for $i\in [2,q]$, $\Sigma_{P(m)}(T_{i-1})\subseteq \Sigma_{P(m)}(T_i)$ and, hence, $1 \leq |\Sigma_{P(m)}(T_1)|\le \cdots \le |\Sigma_{P(m)}(T_q)|=|\Sigma_{P(m)}(T)| < P(m)/4$. Therefore, the number of $i$ which are not bad with $|\Sigma_{P(m)}(T_{i-1})|<P(m)/2w$ and $|\Sigma_{P(m)}(T_i)|\ge (1+2^{-k-1})|\Sigma_{P(m)}(T_{i-1})|$ is at most $\frac{\log(P(2x)/2w)}{\log(1+2^{-k-1})} \le k2^{k+2}\log x$. Moreover, the number of $i$ which are not bad with $P(m)/2w \le |\Sigma_{P(m)}(T_{i-1})|<P(m)/4$ and $|\Sigma_{P(m)}(T_i)|\ge \left(1+\frac{1}{2^{k+1}w}\right)|\Sigma_{P(m)}(T_{i-1})|$ is at most $\frac{\log(2w)}{\log(1+2^{-k-1}w^{-1})}\le \log x$, where we used that $w=(\log x)^{1/2}$. Therefore, the number of $i\in [2,q]$ which are not bad is at most $k2^{k+2}\log x+\log x<k2^{k+3}\log x$. \qed 

\vspace{2mm}

We conclude this subsection and the proof of Theorem~\ref{thm:Ramsey-complete-polynomial} by using Lemma \ref{lem:cover-poly} to prove Lemma~\ref{lem:main-poly-ramsey}. To this end, suppose that $P$ is a complete polynomial of degree $k$ with integer coefficients in its binomial representation, $C(k)= k 2^{k+15}$, $\epsilon\in (0,1/2]$ and $X$ is the set of $y \in [x,(1+1/k)x)$ such that $P(y)$ has no prime divisor at most $(\log x)^{1/2}$. Our aim is to show that if a sequence $S$ of $C(k)\epsilon^{-1} \log x$ elements in $X$ is chosen independently and uniformly at random, then, with high probability, $S$ has distinct terms and, for any subsequence $S'$ of $S$ of size $\epsilon|S|$, the set $\Sigma(P(S'))$ contains all integers in the interval $[\frac{e}{9}P(x)|S'|,\frac{8}{9}P(x)|S'|]$. 

\begin{proof}[Proof of Lemma \ref{lem:main-poly-ramsey}]
As $P$ is a complete polynomial, its leading coefficient is positive. Hence, for $x$ sufficiently large, $P$ will be positive and strictly increasing on the interval $[x,(1+1/k)x]$. We may therefore assume that $P$ is injective on the interval $[x,(1+1/k)x)$ and, for any $y$ in this interval, $P(y) \in [P(x),P((1+1/k)x)) \subset [P(x),eP(x))$.  

By the birthday paradox, as $|S|=o(\sqrt{|X|})$, $P(S)$ has distinct terms with high probability. 
Fix a choice of subset $I'$ of $[C(k)\epsilon^{-1}\log x]$ of size $C(k) \log x$ and consider the subsequence $S'$ of $S$ given by $(s_i)_{i\in I'}$. Let $I''$ be the smallest $|S'|/9$ elements in $I'$ and let $S''$ be given by $(s_i)_{i\in I''}$. Let $\ell = 64$. Arrange $I''$ in increasing order and partition $I''$ into $\ell$ sets $I''_1,\dots,I''_\ell$ of consecutive terms so that each set $I''_j$ for $j\in [\ell]$ has size $|I''|/\ell$. This gives a partition of $S''$ into $\ell$ subsequences $S''_{1},\dots,S''_{\ell}$,  where $S''_j = (s_i)_{i\in I''_j}$. Note that each element of $\Sigma(P(S''_{j}))$ is nonnegative and at most $P((1+1/k)x)|S''_j|<eP(x)|S''_j|$. We shall prove below that, with high probability, the sequence $S$ 
has the property that, for all possible choices of $I'$ and 
$j$, $|\Sigma(P(S''_{j}))|\ge P(x)|S''_j|/8$. Assuming
this, we can show that $\Sigma(P(S_{j}''))$ is not contained in an arithmetic
progression with common difference larger than $1$. Indeed, if $\Sigma(P(S_{j}''))$ is contained in an arithmetic progression
with common difference $d>1$, then $d\le\frac{P((1+1/k)x)|S''_j|}{P(x)|S''_j|/8 -1} \leq 9e$.
Moreover, if $\Sigma(P(S_{j}''))$ is contained in an arithmetic progression
with common difference $d>1$, then all elements in $\Sigma(P(S_{j}''))$
are congruent modulo $d$, from which it follows that all elements
of $P(S_{j}'')$ are divisible by $d$. This contradicts the fact that
no element of $P(S_{j}'')$ has a prime factor at most $(\log x)^{1/2}$.
Hence, for each $j$, $\Sigma(P(S_{j}''))$ is not contained in an arithmetic
progression with common difference larger than $1$. Therefore, by Lemma~\ref{lem:Lev}, as $\Sigma(S'')=\Sigma(S''_{1})+\cdots+\Sigma(S''_{\ell})$, the set $\Sigma(S'')$ contains the integers in an interval of length at least $\ell \left(P(x)|S''_j|/8 - 1\right)+1>P((1+1/k)x)$. Finally, by Lemma \ref{lem:verysimple}, $\Sigma(P(S'')\cup P(S'\setminus S''))$ contains all integers in the interval $[eP(x)|S'|/9,8P(x)|S'|/9]$, where we used that all elements of $S'$ are at most $P((1+1/k)x)$, the elements of $\Sigma(S'')$ are at most $P((1+1/k)x)|S''|=P((1+1/k)x)|S'|/9<eP(x)|S'|/9$ and the sum of the elements in $S'$ is at least $P(x)|S' \setminus S''| = 8P(x)|S'|/9$.

It remains to show that, with high probability, the sequence
$S$ has the property that, for all possible choices of $I'$ and 
$j$, $|\Sigma(P(S''_{j}))| \ge P(x) |S''_j|/8$. Fix
an index $1\le j\le \ell$ and partition the index set $I''_{j}$ of $S''_{j}$ into two consecutive blocks $J_1$ and $J_2$ of equal size. Let $Q_1 = (s_i)_{i\in J_1}$ and $Q_2 = (s_i)_{i\in J_2}$, so $|Q_{1}|=|Q_{2}|=\frac{|S_{j}''|}{2}=\frac{C(k)\log x}{18\ell}>k2^{k+4}\log x$. Recall that $X$ is the set of integers in $[x,(1+1/k)x)$ such that $P(x)$
has no prime divisor at most $(\log x)^{1/2}$. Consider $m\in X$.
We note that when we fix the subset of indices $I'$ of $[C(k)\epsilon^{-1}\log x]$ of size $C(k)\log x$ 
and the index $j$, then $J_{1}$ is determined as a particular subsequence of $I'$.
Moreover, each element in $Q_{1}$ is uniformly
and independently distributed in $X$. Taking a union bound over all 
$x \ell \binom{C(k) \epsilon^{-1} \log x}{C(k)\log x}$ choices of $I'$, $j \in [\ell]$ and $m \in X$, Lemma~\ref{lem:cover-poly} implies that 
the probability $|\Sigma_{P(m)}(P(Q_{1}))|<\frac{P(m)}{4}$
for some choice of $I'$, $j \in [\ell]$ and $m\in X$ is at most
\begin{align*}
x\ell\binom{C(k)\epsilon^{-1} \log x}{C(k)\log x}\cdot \left(\log x\right)^{-k2^{k}C_k^{-1}\log x}&\le x \ell (e/\epsilon)^{C(k)\log x}\cdot\left(\log x\right)^{-k2^{k}C_k^{-1}\log x}=o_x(1),
\end{align*}
where $o_x(1)$ tends to $0$ as $x$ tends to infinity. Thus, with high probability, the sequence $S$ is such that $|\Sigma_{P(m)}(P(Q_{1}))|\ge\frac{P(m)}{4}$ for all possible choices of $I'$, $j\in[\ell]$ and $m\in X$. In this case, by repeated application of Lemma \ref{lem:modp}, for all $j \in [\ell]$,
\[
|\Sigma(P(S_{j}''))|\ge\sum_{m\in Q_{2}}P(m)/4 = |Q_2|P(m)/4 = |S''_j|P(m)/8 \geq P(x)|S''_j|/8.
\]
Therefore, with high probability, the sequence $S$ is such that $|\Sigma(P(S''_{j}))|\ge P(x)|S''_j|/8$ for all possible choices of $I'$ and $j$, as required. 
\end{proof}

\subsection{Proof of the lower bound in Theorem \ref{thm:Ramsey-complete}}

We first prove a useful lemma. 

\begin{lem}\label{cool}
Let $S$ be a sequence of positive integers and $m$ and $q$ be positive integers. Then 
$$\left |\Sigma(S) \cap [m]\right | \leq 2^{m/q}\prod_{a \in S \cap [m]}\left(1+2^{-a/q}\right) \leq 2^{m/q}\exp\left(\sum_{a \in S \cap [m]}2^{-a/q}\right).$$
\end{lem}

\begin{proof}
Let $r_S(s)$ denote the number of ways of representing $s$ as a sum of distinct elements from $S$. So if $s \in \Sigma(S)$, then $r_S(s) \geq 1$, while $r_S(s)=0$ otherwise. For each $s \in \Sigma(S)\cap [m]$, we get a contribution of one to the leftmost expression. For the middle expression, by expanding the product, for each $s \in \Sigma(S) \cap [m]$ we get a contribution of $r_S(s) \cdot 2^{m/q} \cdot 2^{-s/q} \geq 1$, proving the desired inequality. We then get the last inequality by using $1+z \leq e^z$ for $z \geq 0$. 
\end{proof}

Using the above lemma, we prove the following theorem, giving the lower bound in Theorem \ref{thm:Ramsey-complete}. 

\begin{thm}\label{lowerboundresult}
Let $r \geq 2$ be an integer. If a sequence of positive integers $A$ satisfies $A(n) \leq \frac{r-1}{140}(\log_2 n)^2$ for all sufficently large $n$, then $A$ is not $r$-Ramsey complete. 
\end{thm}

\begin{proof}
By replacing $r$ by $r-1$ if $r$ is odd, it suffices to prove that, for $r \geq 2$ even, a sequence of positive integers $A$ with $A(n) \leq \frac{r}{70}(\log_2 n)^2$ for all sufficiently large $n$ is not $r$-Ramsey complete. 

By reordering, we may suppose that $A=(a_i)_{i=1}^\infty$ is in increasing order $a_1 \leq a_2 \leq \cdots$.  Define an $r/2$-coloring of $A$, which we call the {\it hue coloring}, by assigning $a_{\ell}$ hue $h \in \mathbb{Z}_{r/2}$ if $\ell \equiv h \pmod{r/2}$. For a positive integer $j$, define a red/blue-coloring $C_j$ of $A$ where $C_j(a)$ is red if $a \leq 2^j$ and blue otherwise. Let $c_j$ be the product coloring formed from the hue coloring and the red/blue-coloring $C_j$. That is, $c_j$ is an $r$-coloring of $A$ given by the hue and whether or not the term is at most $2^j$. 

The largest positive integer that can be written as a sum of red elements of the same hue in coloring $c_j$ is at most 
\begin{equation}\label{123} 2^j+\frac{2}{r}\sum_{a \in A \cap [2^j]} a.
\end{equation} 
This follows since, for any two hues $h$ and $h'$, the elements of $A\cap [2^j]$ with hue $h$ and those with hue $h'$ interlace and are bounded by $2^j$, so the sum of elements of hue $h$ is at most $2^j$ more than the sum of elements of hue $h'$ and, therefore, at most $2^j$ more than the average sum of elements taken over all hues. 

Let the \emph{cost} of $a \in A \cap [2^j]$ for the coloring $c_j$ be $a/2^j j$. Over all colorings $c_j$ with $j \geq 1$, the total cost of $a > 2$ is $\sum_{j \geq \log_2 a} a/2^j j \leq 2/\log_2 a$, while the cost of each $a \in \{1,2\}$ over all such $c_j$ is at most $2$. If any number larger than $2^{j-1}(j+2)$ can be written as a sum of monochromatic red elements in coloring $c_j$ (so they are also of the same hue), then, by (\ref{123}), we have 
$$2^j+\frac{2}{r}\sum_{a \in A \cap [2^j]} a> 2^{j-1}(j+2),$$
or, equivalently, $\sum_{a \in A \cap [2^j]} a > r2^{j-2}j$, so the total cost of all elements in $A \cap [2^j]$ for the coloring $c_j$ is at least $r/4$. 

Let $i$ be a sufficiently large positive integer. The total cost of the elements $a \in A \cap [2^i]$ for the colorings $c_1,\ldots,c_i$ is at most 
\begin{equation}\label{ineq42} O(1) + \sum_{a \in A \cap (2, 2^i]} \frac{2}{\log_2 a} < \frac{ri}{32},\end{equation}
where the $O(1)$ term comes from considering the cost of the terms $a \in \{1,2\}$.  To prove inequality (\ref{ineq42}), we use Abel's summation formula $$\sum_{x_0 < n\le x}t_nf(n) = T(x)f(x) - T(x_0)f(x_0) - \int_{x_0}^{x} T(y)f'(y)dy,$$ where $f$ is a continuously differentiable function on $[x_0,x]$ and $T(y)=\sum_{n\le y}t_n$. Using Abel's summation formula with $t_n = 1$ if $n\in A\cap (2^{i_0},2^i]$ and $t_n = 0$ otherwise, where $i_0$ is chosen so that $A(n) \le \frac{r}{140}(\log_2 n)^2$ for all $n\ge 2^{i_0}$, and $f(x) = \frac{1}{\log_2 x}$, we obtain 
\begin{align*}
\sum_{a \in A \cap (2, 2^i]} \frac{1}{\log_2 a} &= O(1) + \frac{A(2^i)}{i} - \frac{A(2^{i_0})}{i_0}+ \int_{2^{i_0}}^{2^i} \frac{A(x) \log 2}{x (\log x)^2} dx \\ 
&\le O(1) + \frac{ri^2}{140 i} + \int_{2^{i_0}}^{2^i} \frac{r(\log_2 x)^2 \log 2}{140 x (\log x)^2} dx\\
&\le O(1) + \frac{2ri}{140} < \frac{ri}{65},
\end{align*} where we assume in the last inequality that $i$ is sufficiently large. 

Thus, fewer than $(ri/32)/(r/4)=i/8$ of the $i$ colorings $c_j$ with $1 \leq j \leq i$ have the property that there is a number greater than $2^{j-1}(j+2)$ that can be expressed as a sum of elements which are red of the same hue. We call $j$ {\it red-strong} if there is a number greater than $2^{j-1} (j+2)$ that can be expressed as a sum of elements which are red of the same hue in the coloring $c_j$. 

For a non-negative integer $j$, define
$$g(j)=\sum_{a \in A, 2^j < a \leq 2^j j}2^{-a/2^{j+4}}.$$ 
For each $a \in A$, the contribution of $a$ to the various $g(j)$ is $$\sum_{j : 2^j < a \leq 2^j j}2^{-a/2^{j+4}} \leq \sum_{h \geq 0} 2^{-2^{h-4}} < 5,$$
where we used the change of variables $h=\lfloor \log_2 a \rfloor -j$. 
 Hence, $$\sum_{j=1}^i g(j) < \sum_{a \in A \cap [2^i i]}5 = 5A(i2^i).$$

For $s \in \mathbb{Z}_{r/2}$, let $A_s$ be the subset of $A$ consisting of elements of hue $s$. Let $A_{s,>t}=\{a \in A_s: a > t\}$. Let $b(j)$ denote the number of elements of $[2^j j]$ which can be written as a sum of blue elements in coloring $c_j$ of the same hue, so $b(j)=|\bigcup_{s=1}^{r/2} \Sigma(A_{s,> 2^j}) \cap [2^j j]| \leq \sum_{s=1}^{r/2}|\Sigma(A_{s,> 2^j}) \cap [2^j j]|$. Applying Lemma \ref{cool} with $S=A_{s,>2^j}$, $m = 2^j j$ and $q = 2^{j+4}$, we have 
$$\left|\Sigma(A_{s,>2^j}) \cap [2^j j]\right| \le 2^{j/16} \exp\left(\sum_{a\in A_{s,>2^j} \cap [2^j j]} 2^{-a/2^{j+4}}\right) \le 2^{j/16}\exp\left(1+\frac{2}{r}g(j)\right),$$ where we have again used the fact that, for any two hues $h$ and $h'$, the blue elements with hue $h$ interlace the blue elements of hue $h'$ together with the observation that the function $x\mapsto 2^{-2^{x-4}}$ is monotone and bounded above by $1$.   
We thus have $b(j) \leq \frac{r}{2}2^{j/16}\exp\left(1+\frac{2}{r}g(j)\right)$. 
Hence, for $i$ sufficiently large, \begin{equation}\label{equationone} \prod_{j=1}^i b(j) \leq \prod_{j=1}^i \frac{r}{2}2^{j/16}\textrm{exp}\left(1+\frac{2}{r}g(j)\right) \leq r^i2^{i^2/31}\exp\left(\frac{2}{r}\sum_{j=1}^i g(j)\right) \leq r^i2^{i^2/31}\exp\left(\frac{10}{r}A(i2^i) \right)< 2^{i^2/4}.\end{equation}

If at least $3i/4$ of the $i$ colorings $c_j$ for $j=1,\ldots,i$ have the property that at least $2^{j}$ positive integers at most $2^j j$ can be written as a sum of blue elements of the same hue, then the left hand side of (\ref{equationone}) is at least $\prod_{j=1}^{3i/4}2^{j} = 2^{\binom{3i/4 + 1}{2}} > 2^{i^2/4}$, contradicting (\ref{equationone}). Hence, for at least $i/4$ of the colorings $c_j$ with $j=1,\ldots,i$, we have that there are at most $2^{j}$ positive integers at most $2^j j$ which can be written as a sum of blue elements of the same hue in $A$. Call $j$ {\it blue-strong} if in coloring $c_j$ at least $2^j$ positive integers at most $2^j j$ can be written as a sum of blue elements of the same hue in $A$. Call $j$ {\it weak} if it is neither blue-strong nor red-strong. Hence, for $i$ sufficiently large, there are at least $i-i/8-3i/4=i/8$ elements $j \in [i]$ which are weak. Thus, there are infinitely many weak $j$ and we let $J=\{j_h\}_{h \geq 1}$ be an infinite sequence of weak $j$ so that $j_h \geq 2^{rj_{h-1}^2}$.

We next define an $r$-coloring $c$ of $A$ for which there are infinitely many integers which cannot be written as a sum of monochromatic elements from $A$. The coloring $c$ is a product coloring of the hue coloring (which uses $r/2$ colors) and a red/blue-coloring of $A$. We color an integer in $A$ blue if it is in one of the intervals $(2^j,2^j j]$ with $j \in J$ and red otherwise. We will prove that at least half the elements in $(2^{j-1}(j+2),2^j j]$, where $j \in J$ is sufficiently large, cannot be written as a monochromatic sum in the coloring $c$. 

Suppose now that $N \in (2^{j-1}(j+2),2^j j]$ is a sum of red elements of the same hue. Since there are no red elements in $(2^j,2^j j]$ in the coloring $c$, $N$ can also be written as a sum of red elements of the same hue in $c_j$, contradicting the assumption that $j$ is weak. Hence, no element in $(2^{j-1}(j+2),2^j j]$ is a sum of red elements of the same hue in the coloring $c$. 

As $j=j_h$ is weak, there are at most $2^j$ elements at most $2^j j$ that can be written as a monochromatic sum of blue elements of the same hue in $(2^j,2^j j]$. The number of remaining blue elements in $[2^j j]$ is at most 
$$A(2^{j_{h-1}}j_{h-1})\leq \frac{r}{140}(\log_2 (2^{j_{h-1}}j_{h-1}))^2 \leq \frac{r j_{h-1}^2}{2} \leq \frac{\log_2 j}{2}.$$ 
Thus, the number of positive integers at most $2^j j$ which can be written as a monochromatic sum of blue elements in the coloring $c$ is at most $2^{(\log_2 j)/2}2^j <2^j j/8$. Hence, as $2^j j/8 \leq \frac{1}{2}\left( 2^j j-2^{j-1}(j+2)\right)$, 
at least half the elements in $(2^{j-1}(j+2),2^j j]$ cannot be written as a sum of blue elements of the same hue in the coloring $c$. As there are infinitely many such $j$, there are infinitely many positive integers which are not the monochromatic sum of elements in the coloring $c$. This completes the proof. \end{proof}

\noindent {\it Remark.} In the proof above, for $j=1,\ldots,i$, we made use of colorings $C_j$ which color the positive integers up to $2^j$ red and all larger integers blue. Alternatively, we could have picked a random coloring $\phi_x$ which colors all positive integers up to $x$ red and all larger integers blue, where $x \in [N]$ is chosen with probability $\frac{1}{xH(N)}$ with $H(N)=\sum_{x=1}^N \frac{1}{x}$. One can then do a similar analysis using elementary probability to get a better constant factor in Theorem \ref{lowerboundresult}.

\section{\label{sec:density-ramsey}Density completeness}

In this section, we discuss Theorems \ref{epsilonclean} and \ref{thm:poly-density-ramsey}, our results on density completeness. Since reordering a sequence does not change whether or not it is $\epsilon$-complete, it will suffice to consider monotonically increasing sequences. We will begin with the following simple result, from which the first part of Theorem~\ref{epsilonclean} follows.

\begin{thm}
\label{epsilonthm1} Let $\epsilon>0$. If $A=(a_{n})_{n\geq1}$ is
a monotonically increasing sequence of positive integers which is $\epsilon$-complete,
then there is $C$ such that 
\[
a_{n}\leq\sum_{i\leq\epsilon n+C}a_{i}
\]
holds for all positive integers $n$. 
\end{thm}

\begin{proof}
Suppose that there is no such $C$. Then there is a function $g:\mathbb{N}\rightarrow\mathbb{N}$
with $\lim_{n\to\infty}g(n)=\infty$ such that $a_{n}>\sum_{i\leq\epsilon n+g(n)}a_{i}$
holds for infinitely many $n$. Thus, we can pick an infinite sequence
of positive integers $n_{1},n_2,\dots$ such that, for all $j$, we have
$a_{n_{j}}>\sum_{i\leq\epsilon n_{j}+g(n_{j})}a_{i}$ and $g(n_{j})> n_{j-1}$. Pick a subsequence $A'$ of $A$ by deleting all elements $a_{i}$
of $A$ whose subscript $i$ satisfies $\epsilon n_{j}+g(n_{j})<i\leq n_{j}$
for some positive integer $j$.

We first show that $A'(x)\geq\epsilon A(x)$ holds for all $x$. It
suffices to check this when $x=a_{n_{j}}$ for
some positive integer $j$. However, we have $A'(a_{n_{j}})\geq\epsilon n_{j}+g(n_{j})-n_{j-1}>\epsilon n_{j}=\epsilon A(a_{n_{j}})$,
as required.

To see that $A'$ is not complete, we show that
each integer $a_{n_{j}}$ is not the sum of elements from $A'$. Indeed,
such elements must be at most $a_{n_{j}}$ and hence at most $a_{\epsilon n_{j}+g(n_{j})}$.
However, $\sum_{i\leq\epsilon n_{j}+g(n_{j})}a_{i}<a_{n_{j}}$, so
$a_{n_{j}}$ is not in $\Sigma(A')$ and $A'$ is not complete.
\end{proof}

This gives a necessary growth condition for
a sequence to be $\epsilon$-complete. Recall that it is also necessary
for an $\epsilon$-complete sequence to satisfy the divisibility
condition that no prime is a factor of more than an $\epsilon$-proportion
of the elements in the sequence. In the proof of Theorem \ref{epsilonthm2}
below, we show that, apart from some mild additional assumptions, a
random sequence satisfying both the growth condition from Theorem~\ref{epsilonthm1} and a suitable variant of this divisibility condition
is likely to be $\epsilon$-complete. 

Recall that, for a sequence $B=(b_{n})_{n\geq1}$, the discrete
derivative is defined by $\Delta b_{n}:=b_{n+1}-b_{n}$. Fix $0<\epsilon<1$. A sequence $B=(b_{n})_{n\geq1}$ is called \emph{$\epsilon$-friendly} (or \emph{friendly}) if it satisfies the following five growth conditions:

\begin{enumerate}
\item For some constant $C$ and all $n$,  
\[
b_{n}\leq\sum_{i\leq\epsilon n+C}b_{i}.
\]
\item $\lim_{n\to\infty}\Delta b_{n}=\infty$. 
\item $\lim_{i\to\infty}\frac{B(2^{i+1})-B(2^{i})}{i}=\infty$. 
\item There exists $0<c<1$ such that
$c\Delta b_{i}\leq\Delta b_{j}$ for all $i<j$. Moreover, if $b_{j}<2b_{i}$,
then $\Delta b_{j}\leq\frac{1}{c}\Delta b_{i}$. 
\item $B$ is strictly increasing. 
\end{enumerate}
By Theorem \ref{epsilonthm1}, condition (i) is necessary for
an increasing sequence of positive integers to be $\epsilon$-complete. The other growth conditions
are mild assumptions that will be helpful in proving the existence
of an $\epsilon$-complete sequence $A=(a_{n})_{n\geq1}$ which interlaces
$B$, that is, for which $b_{n}\leq a_{n}\leq b_{n+1}$ for all $n$.

Let $b_{1},b_{2},\ldots,b_{t}$ be any finite strictly increasing
sequence of positive integers. Let $\{x\}=x-\lfloor x\rfloor$ denote
the fractional part of $x$. If we define $B=(b_{n})_{n\geq1}$ recursively
by 
\[
b_{n}=\lfloor\{\epsilon n\}b_{\lceil\epsilon n\rceil}\rfloor+\sum_{i\leq\epsilon n}b_{i}
\]
for $n>t$, then it is easy to check that such a sequence is friendly
and satisfies $b_{n}=\Theta(f_{n})$, where, following the introduction, $F=(f_{n})_{n\geq1}$ is any sequence
of positive integers for which $f_{n}=\sum_{i\leq\epsilon n}f_{i}$
for all sufficiently large $n$. We note that the term $\lfloor\{\epsilon n\}b_{\lceil\epsilon n\rceil}\rfloor$ is added as a ``discrete interpolation'' factor to guarantee conditions (ii) and (iv) of friendly sequences. 

\begin{claim}
\label{doubleclaim} If $B=(b_{n})_{n\geq1}$ is a friendly sequence
and $n$ is a sufficiently large positive integer, then $b_{2n/c+1}\geq2b_{n+1}$,
where $c$ is the constant in condition (iv) of friendly sequences. 
\end{claim}

\begin{proof}
We have
\begin{equation}
b_{n+1}=b_{1}+\sum_{i=1}^{n}\Delta b_{i}\leq b_{1}+n\max_{i\leq n}\Delta b_{i}\label{continuation}
\end{equation}
and 
\begin{eqnarray*}
b_{2n/c+1} &=& b_{n+1}+\sum_{i=n+1}^{2n/c}\Delta b_{i} \\ & \geq & b_{n+1}+(2n/c-n)\min_{j\geq n}\Delta b_{j} \\ & \geq & b_{n+1}+(n/c)\cdot c\max_{i\leq n}\Delta b_{i}+(n/c-n)\min_{j\geq n}\Delta b_{j} \\ & \geq  & b_{n+1}+(n/c)\cdot c\max_{i\leq n}\Delta b_{i}+b_1 \\ & \ge & 2b_{n+1}.
\end{eqnarray*}
Here we used condition (iv) of friendly sequences to deduce the second inequality, the third inequality 
follows from $(n/c-n)\min_{j\geq n}\Delta b_{j}> b_1$ for $n$ sufficiently large and the last inequality
is by (\ref{continuation}). 
\end{proof}

The next theorem is our main result on $\epsilon$-complete sequences
and completes the proof of Theorem \ref{epsilonclean}. We remark that since condition (i) of friendly sequences only
gives an upper bound on $b_{n}$, this result also allows us to find sequences that are considerably denser than
$(f_{n})_{n\geq1}$ that are $\epsilon$-complete.

\begin{thm}
\label{epsilonthm2} Let $0 < \epsilon < 1$ and $B=(b_{n})_{n\geq1}$ be a friendly sequence
as defined above. Then there is a sequence $A=(a_{n})_{n\geq1}$ of positive
integers that interlaces $B$, i.e., $b_{n}\leq a_{n}<b_{n+1}$ for
all $n$, which is $\epsilon$-complete. 
\end{thm}

\begin{proof}
Let $\epsilon_0>0$ be sufficiently small. We pick the sequence $A$ by taking, for $j$ sufficiently large,
$a_{j}$ to be a uniform random integer in $[b_{j},b_{j+1})$ which has
no prime factor at most $(\max(1/\epsilon,1/\epsilon_0))^{4000}$. For small $j$ this might not be possible,
as the interval $[b_{j},b_{j+1})$ might not contain any integer with no prime factor
at most $(\max(1/\epsilon,1/\epsilon_0))^{4000}$, so we let $a_{j}$ be any integer in $[b_{j},b_{j+1})$
in this case. This guarantees that $A$ interlaces $B$.

For a positive integer $i$, let $h(i)$ be the smallest integer for
which $b_{h(i)}\geq2^{i}$. Note that $b_{h(i+1)-1}$ is the largest
element of $B$ which is less than $2^{i+1}$. Let $A_{i}:=A\cap[b_{h(i)},b_{h(i+1)-1})$,
so $A_{i}$ consists of all but at most two elements of $A\cap[2^{i},2^{i+1})$. By condition (iii) of friendly sequences, for any $C$ and for $i$ sufficiently large depending on $C$, $|A_{i}| \ge C\max(1/\epsilon,1/\epsilon_0)i$. 
The following lemma is a close relative of Lemma \ref{lem:main-ramsey}. The proof of the lemma, which is an appropriate modification of the proof of Lemma \ref{lem:main-ramsey}, is deferred to Appendix \ref{appendix:ramsey-density}. 

\begin{lem}\label{lem:ramsey-density}
There exist positive constants $\epsilon_0$, $C_1$ and $C_2$ such that the following holds. For $i$ sufficiently large, with positive probability,
the set $A_{i}$ has the property that, for any subset $A'_{i}\subset A_{i}$
with $|A'_{i}|\geq(\min(\epsilon,\epsilon_0)/4)|A_{i}|$, $A'_{i}$ contains
a subset $A''_{i}$ with $|A''_{i}|\leq C_{1}i$ such that $\Sigma(A''_{i})$
contains every integer in the interval $[y,2y]$, where $y=C_{2}2^{i}i$. 
\end{lem}

Since the choices of $A_{i}$ for different $i$ are
mutually independent, we can guarantee and will assume that $A_i$ satisfies the conclusion of Lemma \ref{lem:ramsey-density} for \emph{each} sufficiently large $i$.

Our goal now is to show that if $A'$ is any subsequence of $A$
with $|A'(n)|\geq\epsilon|A(n)|$ for all sufficiently large $n$,
then $A'$ is complete. We first show that for each $n_{0}$ there
is $n\geq n_{0}$ such that $A'$ is complete or $A'$ contains only
roughly $\epsilon n$ elements among the first $n$ elements of $A$. 
We then go through a very similar argument using this additional structure
to conclude that $A'$ is complete.

Let $i_{0}$ be a sufficiently large positive integer and $m=A(2^{i_{0}})$.
The number of elements in $A$ which are at most $2^{i_{0}}$ and not
in any $A_{i}$ is at most $2i_{0}\leq\frac{\epsilon}{4}m$. Let $A'_i= A_i \cap A'$. So the
set $\bigcup_{i\leq i_{0}}A'_{i}$ of elements in $A'$ which are
at most $2^{i_{0}}$ and in some $A_{i}$ has size at least $\epsilon A(2^{i_{0}})-\frac{\epsilon}{4}m\geq\frac{3\epsilon}{4}m$.
Let $i_{1}\leq i_{0}$ be the largest positive integer for which $|A'_{i_{1}}|\geq\frac{\epsilon}{4}|A_{i_{1}}|$, which exists by the observation just made.
The set $\bigcup_{i_{1}<i\leq i_{0}}A'_{i}$ has cardinality at most
$\frac{\epsilon}{4}|\bigcup_{i_{1}<i\leq i_{0}}A_{i}|\leq\frac{\epsilon}{4}m$,
so there are at least $\frac{3\epsilon}{4}m-\frac{\epsilon}{4}m=\frac{\epsilon}{2}m$
elements in $\bigcup_{i\leq i_{1}}A'_{i}$. In particular, $A(2^{i_{1}+1})\geq A'(2^{i_{1}+1})\geq\frac{\epsilon}{2}m$.

Since $A_{i_1}$ satisfies the conclusion of Lemma \ref{lem:ramsey-density} and $|A'_{i_1}| \ge (\epsilon/4)|A_{i_1}|\ge (\min(\epsilon,\epsilon_0)/4)|A_{i_1}|$, there is $A''_{i_{1}}\subset A'_{i_{1}}$ with
$|A''_{i_{1}}|\leq C_{1}i_{1}$ such that $\Sigma(A''_{i_{1}})$ contains
every integer in the interval $[y,2y]$ where $y=C_{2}2^{i_{1}}i_{1}$.
Label the elements in $A'\setminus A''_{i_{1}}$ in increasing order
as $a'_{1},a'_{2},\ldots$.

By Lemma \ref{lem:verysimple}, if, for each $j$, we have $a'_{j}\leq2y-y+a'_{1}+\cdots+a'_{j-1}$,
then $\Sigma(A'\setminus A''_{i_{1}})+[y,2y)$ contains all integers
at least $y$ and, as $\Sigma(A')$ is a superset of $\Sigma(A'\setminus A''_{i_{1}})+[y,2y)$,
$A'$ would be complete. So we may assume that there is some $j$ for
which 
\begin{equation}
a'_{j}>y+a'_{1}+\cdots+a'_{j-1}.\label{notenuf}
\end{equation}
In particular, $a'_{j}\geq2^{i_{1}+1}\geq a_{\frac{\epsilon}{2}m}\geq\frac{\epsilon}{2}m$, so that $j$ can be made sufficiently large by taking $i_0$ and, hence, $m$ sufficiently large.

As $a'_{j}\in A'\setminus A''_{i_1}  \subset A$, there is a positive integer $n$ for
which $a'_{j}=a_{n}$, where, again, $n$ can be made sufficiently large by taking $i_0$ sufficiently large. We have 
\[
\sum_{i\leq j-1}a'_{i}<a_{n}<b_{n+1}\leq\sum_{i\leq\epsilon(n+1)+C}b_{i}\leq\sum_{i\leq\epsilon(n+1)+C}a_{i},
\]
where the first inequality follows from (\ref{notenuf}), the second
and fourth inequalities are by the fact that $A$ interlaces $B$ and the third
inequality follows from condition (i) of friendly sequences.
This implies that $j-1\le \epsilon(n+1)+C$, so  
\begin{equation}
A'(a_{n}-1)\leq|A''_{i_1}|+j-1\leq C_{1}i_{1}+\epsilon(n+1)+C<C_{1}i_{0}+\epsilon n+C+1.\label{steponepart}
\end{equation}
That is, the number of elements of $A'$ amongst the first $n$ elements of $A$ is roughly $\epsilon n$.

We next give a similar argument, but using the extra information that
there are many elements $a_{s}\in A\setminus A'$ with $s\leq n$
in order to conclude that $A'$ is complete. Let $N=8(n+1)/(\epsilon^{2}c)$.
Let $i_{2}$ be the least positive integer such that $2^{i_{2}}\geq a_{N}$ and
let $m'=A(2^{i_{2}})$, so $m'\geq N$. 
As $a_{2N/c+1}\geq b_{2N/c+1}\geq2b_{N+1}>2a_{N}$, where the middle
inequality follows from Claim \ref{doubleclaim}, there is a perfect
power of two which is at least $a_{N}$ and less than $a_{2N/c+1}$,
so $m'=A(2^{i_{2}})\leq2N/c$. Furthermore, since $a_n = a_j' \ge a_{\frac{\epsilon}{2}m}$, we have that $n\ge \frac{\epsilon}{2}m$. Thus, $N = 8(n+1)/(\epsilon^{2}c) > m$, so $A(2^{i_2}) \ge N > m =A(2^{i_0})$. In particular, we obtain that $i_2 \ge i_0$. Hence, $N$, $i_2$ and $m'$ may be made sufficiently large by taking $i_0$ and, hence, $n$ sufficiently large.

We also have $A'(2^{i_{2}})\geq\epsilon A(2^{i_{2}})=\epsilon m'$,
so $A'$ contains at least $\epsilon m'$ elements $a_{s}\leq2^{i_{2}}$.
Let $i_{3} \leq i_2$ be the largest positive integer such 
that $|A'_{i_{3}}|\geq\frac{\epsilon}{4}|A_{i_{3}}|$. Recall that,
for each positive integer $i$, the number of elements of $A$ in
$[2^{i},2^{i+1})\setminus[b_{h(i)},b_{h(i+1)-1})$ is at most two.
It follows that $A'\cap[2^{i_{2}}]\setminus\bigcup_{i<i_{2}}[b_{h(i)},b_{h(i+1)-1})$
has cardinality at most $2i_{2}\leq\frac{\epsilon m'}{4}$, where
the last inequality follows from condition (iii) of friendly sequences and the fact that $i_{2}$ is sufficiently large. Hence, at
least a fraction $\epsilon-\frac{\epsilon}{4}-\frac{\epsilon}{4}=\frac{\epsilon}{2}$
of the elements of $A$ up to $2^{i_{2}}$ are greater than $a_{\epsilon N/4}$
and in $A'\cap\bigcup_{i<i_{2}}[b_{h(i)},b_{h(i+1)-1})$ and, therefore, 
$i_{3}$ satisfies $2^{i_{3}+1}>a_{\epsilon N/4}$.

Since $A_{i_3}$ satisfies the conclusion of Lemma \ref{lem:ramsey-density} and $|A'_{i_3}| \ge (\min(\epsilon,\epsilon_0)/4)|A_{i_3}|$, there is $A_{i_{3}}''\subset A'_{i_{3}}$ with
$|A_{i_{3}}''|\leq C_{1}i_{3}$ such that $\Sigma(A_{i_{3}}'')$ contains
every integer in the interval $[y',2y']$ where $y'=C_{2}2^{i_{3}}i_{3}$.
Label the elements in $A'\setminus A_{i_{3}}''$ in increasing order
as $a'_{1},a'_{2},\ldots$, noting that we have relabeled most of the elements
in $A'$.

Again, by Lemma \ref{lem:verysimple}, if, for each $j$,
we have $a'_{j}\leq2y'-y'+a'_{1}+\cdots+a'_{j-1}$, then $\Sigma(A'\setminus A_{i_{3}}'')+[y',2y')$
contains all integers at least $y'$ and, as $\Sigma(A')$ is a superset
of $\Sigma(A'\setminus A_{i_{3}}'')+[y',2y')$, $A'$ would be
complete. So we may assume that there is some $j'$ for which 
\[
a'_{j'}>y'+a'_{1}+\cdots+a'_{j'-1}.
\]
Note in particular that $a'_{j'}\geq y'\geq2^{i_{3}+1}\geq a_{\epsilon N/4}$.
Let $n'$ be such that $a_{n'}=a'_{j'}$, so 
\begin{equation}
n'\geq\epsilon N/4.\label{eqnow}
\end{equation}

By condition (i) of friendly sequences, we have 
\[
a'_{j'}=a_{n'}<b_{n'+1}\leq\sum_{i\leq\epsilon(n'+1)+C}b_{i}\leq\sum_{i\leq\epsilon(n'+1)+C}a_{i}.
\]
Note also that, for $i\geq\epsilon n'$, we have 
\begin{equation}
a_{i}\geq a_{\epsilon n'}\geq b_{\epsilon n'}\geq2b_{c\epsilon n'/2}>2a_{c\epsilon n'/2-1}\geq2a_{c\epsilon^{2}N/8-1}\geq2a_{n}.\label{eqnext1}
\end{equation}
Here the first inequality follows from $A$ being increasing, the second
and fourth inequalities follow from the fact that $A$ interlaces $B$, the third
inequality follows from Claim \ref{doubleclaim}, the fifth inequality
follows from (\ref{eqnow}) and the last inequality follows from
the choice of $N$ and the fact that $A$ is increasing.

Since $A'$ has at least $\epsilon n'$ elements up to $a_{n'}$,
we have $j'\geq\epsilon n'-|A''_{i_{3}}|$. It follows that 
\begin{eqnarray}
\sum_{k<j'}a'_{k} & \geq & -\sum_{a\in[a_{n}]\cap(A\setminus A')}a+\sum_{k\leq\epsilon n'-|A_{i_{3}}''|-1+n-A'(a_{n})}a_{k}\nonumber\\
 & \geq & -a_{n}(n-A'(a_{n}))+\sum_{k\leq\epsilon n'-|A_{i_{3}}''|-1+n-A'(a_{n})}a_{k}\nonumber \\
 & \geq & -\frac{1}{2}a_{\epsilon n'}(n-A'(a_{n}))+\sum_{k\leq\epsilon n'-|A_{i_{3}}''|-1+n-A'(a_{n})}a_{k}\nonumber \\
 & \geq & \sum_{k\leq\epsilon n'-|A_{i_{3}}''|-1+\frac{1}{2}\left(n-A'(a_{n})\right)}a_{k}, \nonumber 
\end{eqnarray}
where the first inequality uses that $A$ is an increasing sequence and $\sum_{k<j'}a'_k+\sum_{a\in [a_n]\cap(A\setminus A')}a$ is a sum of at least $\epsilon n'-|A''_{i_3}|-1+n-A'(a_n)$ distinct terms of $A$, which is at least the sum of the first $\epsilon n'-|A''_{i_3}|-1+n-A'(a_n)$ terms in $A$.  
The second inequality follows from $A$ being increasing,
the third inequality follows from using (\ref{eqnext1}) and the
last inequality follows from $A$ being increasing and the following estimate showing that $\frac{1}{2}(n-A'(a_n))-|A''_{i_3}|-1>0$. We have 
\[
-|A_{i_{3}}''|-1+\frac{1}{2}\left(n-A'(a_{n})\right)\geq-C_{1}i_{3}-1+\frac{1}{2}\left(n-\left(C_{1}i_{0}+\epsilon n+C+2\right)\right)\geq-2C_{1}i_{2}+\frac{1}{2}(1-\epsilon)n-\frac{C+2}{2}>C+1,
\]
where the first inequality is by (\ref{steponepart}), the second
inequality uses $i_{0},i_{3}\leq i_{2}$ and $i_{2}$ is sufficiently 
large, while the last inequality uses $n\geq\epsilon^{2}cN/10$, 
$N\geq cm'/2$, condition (i) of friendly sequences, the fact that $A$ 
interlaces $B$, $m'=A(2^{i_{2}})$, $i_{2}$ is sufficiently
large and $m' \geq 2^{\sqrt{\left(2\log_{2}(1/\epsilon)+o(1)\right)i_{2}}}$ from Appendix \ref{appendix:est-density-ramsey}, from all of which it follows that $i_{2}\ll n$.
However, this implies that 
\[
b_{n'+1}>a_{n'}=a'_{j'}>\sum_{k<j'}a'_{k}>\sum_{k\leq\epsilon n'+C+1}a_{k}\geq\sum_{k\leq\epsilon n'+C+1}b_{k}\geq\sum_{k\leq\epsilon(n'+1)+C}b_{k},
\]
contradicting condition (i) of friendly sequences.
\end{proof}

Theorem \ref{thm:poly-density-ramsey} is obtained similarly, by replacing Lemma \ref{lem:ramsey-density} in the above proof by an appropriate analogue of Lemma \ref{lem:main-poly-ramsey}. As indicated in the introduction, we omit the details.

\section{\label{sec:Monochromatic-subset-sums}Monochromatic subset sums}

\subsection{Proof of the lower bound in Theorem \ref{thm:monochromatic-subset-sums}\label{subsec:monochromatic-subset-sums}}

Throughout this section, we use the convention that products and sums indexed by $p$ run over primes. Recall that $p_i$ is the $i^{\textrm{th}}$ prime, $W(\rho)=\prod_{i=1}^{\rho}p_i$ and $\tau(\rho,m)=\phi(W(\rho)m)/(W(\rho)m)=\prod_{p|W(\rho)m}(1-1/p)$. We recall from the introduction that, for positive integers $n$ and $m$ with $m\in [n,\binom{n}{2}]$, we define $\rho(n,m)$ to be the smallest positive integer $\rho$ such that $\rho/\tau(\rho,m) \ge n^2/\phi(m)$. Let $\psi(n,m) = \frac{m^{1/3}(m/\phi(m))}{(\log n)^{1/3}(\log\log n)^{2/3}}$ and ${\cal R}(n,m)=\min\left(\psi(n,m),\rho(n,m)\right)$. By Claim \ref{claim:orderrho} in Appendix \ref{appendix:monochromatic}, we note that ${\cal R}(n,m) = \Theta\left(\psi(n,m)\right)$ when $m = O\left( \frac{n^{3/2}(\log\log n)^{1/2}}{(\log n)^{1/2}}\right)$ and ${\cal R}(m,n) = \Theta(\rho(n,m))$ otherwise. 

We aim to prove that $f(n,m)$, the minimum $r$ such that there exists an $r$-coloring of $[n-1]$ where $m$ cannot be written as a sum of distinct monochromatic elements, is bounded below by ${\cal R}(n,m)$ up to a constant factor, giving the lower bound in Theorem \ref{thm:monochromatic-subset-sums}. The main result of this subsection is the following lemma, from which the required lower bound easily follows. 

\begin{lem}
\label{lem:lem-choosey}There exist positive constants $c$ and $C$
such that the following holds. Let $n$ be sufficiently large and $m\in[n,\binom{n}{2}]$ be such that $r=c{\cal R}(n,m)$ is at least $C$.
Let $y<n/2$ be such that $$m\in\left[\frac{y^{2}(m/\phi(m))\tau(r,m)}{25r},\frac{y^{2}(m/\phi(m))\tau(r,m)}{15r}\right]$$
and let $Y$ be the set of integers in $[y,2y)$ of the form $qu$, where $u|m$,
$u\le y^{1/16}$ and $q$ is coprime to $W(r)m$.
Then, in any $r$-coloring of $Y$, there exists a monochromatic subset sum which equals $m$. 
\end{lem}

By Claim \ref{claim:cond-for-y} in Appendix \ref{appendix:monochromatic}, for any $m\in [n,\binom{n}{2}]$, there exists a choice of $y \in [\max(r^2,n^{3/5}),n/2)$ satisfying the required condition. We may therefore apply the lemma to conclude that if ${\cal R}(n,m) \ge C/c$, then $f(n,m)\ge c{\cal R}(n,m)$. That is, the lower bound in Theorem~\ref{thm:monochromatic-subset-sums} holds in this case. On the other hand, if ${\cal R}(n,m)<C/c$, we have the trivial bound $f(n,m)\ge 1 \ge C^{-1}c{\cal R}(n,m)$, so the lower bound in Theorem~\ref{thm:monochromatic-subset-sums} also holds in this case. For the same reason, we can and will assume throughout that $r $ is sufficiently large.

We will build towards the proof of Lemma \ref{lem:lem-choosey} through a series of reductions and intermediate results. For convenience, we will often use objects and notation in the lemma statements without repeating their definitions from earlier. We begin with the following number-theoretic estimate, whose proof may be found in Appendix~\ref{sec:number-theoretic}.

\begin{lem}
\label{lem:pseudoprimes-count}
Let $r$, $n$ and $m$ be positive integers such that $m\in [n,\binom{n}{2}]$, $r\le n$ and $r$ is sufficiently large. For any interval $I=[x,2x)$ with $x\ge n^{1/4}$, there are at most $8(m/\phi(m))\tau(r,m)x$ integers in $I$ of the form $qu$, where $u|m$, $u\le x^{1/16}$ and $q$ is coprime to $W(r)m$. If also $x \ge r^2$, then there are at least $\frac{1}{8}(m/\phi(m))\tau(r,m)x$  integers in $I$ of this form. 
\end{lem}

By Lemma \ref{lem:pseudoprimes-count}, the set $Y$ defined in Lemma \ref{lem:lem-choosey} satisfies $$|Y|\ge \frac{1}{8}(m/\phi(m))\tau(r,m)y.$$ By the pigeonhole principle, in any $r$-coloring of $Y$, there is 
one color class whose size is at least $\frac{1}{8}(m/\phi(m))\tau(r,m)\frac{y}{r}$. Let $Q_0$ be the elements of $Y$ in this color class. We will prove that $m\in\Sigma(Q_0)$.

Call a set $X$ of integers \textit{$k$-diverse} if, for each $v\ge2$,
there are at least $k$ elements of $X$ which are not divisible by
$v$. If $Q_0$ is not $y^{1/4}$-diverse, there exists $v_0\ge 2$ such that at most $y^{1/4}$ elements of $Q_0$ are not divisible by $v_0$. We replace $Q_0$ by
$Q_{1}=\{a/v_0:a\in Q_0,v_0|a\}\subseteq[y/v_0,2y/v_0)$. We then iterate this process. For $i\ge 1$, if $Q_i$ is not $y^{1/4}$-diverse, we can remove at most $y^{1/4}$ elements of $Q_i$ so that the remaining elements are divisible by some $v_i\ge 2$. We then let $Q_{i+1} = \{x/v_i:x\in Q_i,v_i|x\}$. We stop the process once we reach a set $Q_s$ which is $y^{1/4}$-diverse. Note that there can be at most $\log_2 n$ iterations, so there must be at least $\frac{1}{8}(m/\phi(m))\tau(r,m)\frac{y}{r}-y^{1/4}\log_2 n$
elements in $Q_{s}$. 

By the process defining $Q_s$, there exists $v$ such that $Q_s=\{x/v:x\in Q_0,v|x\}$.
Let $Q=Q_{s}$. Then $Q$ is a subset of $[y/v,2y/v)$ of
size at least $\frac{1}{8}(m/\phi(m))\tau(r,m)\frac{y}{r}-y^{1/4}\log_2 n$ which is $y^{1/4}$-diverse. 
Let $z=\frac{1}{8}(m/\phi(m))\tau(r,m)\frac{y}{r}-y^{1/4}\log_2 n$. Note that 
\begin{equation} 
\frac{1}{8}(m/\phi(m))\tau(r,m)\frac{y}{r}\ge z\ge\frac{1}{10}(m/\phi(m))\tau(r,m)\frac{y}{r} \ge \frac{1}{10}\tau(r,m)\sqrt{y}>y^{1/3},\label{eq:z}
\end{equation} 
where we used that $y\ge \max(r^2,n^{3/5})$, which is inequality (\ref{eq:ycond}) of Claim \ref{claim:cond-for-y} in Appendix \ref{appendix:monochromatic}, and $\tau(r,m) \ge 1/(8\log n\log \log n)$ by inequality (\ref{eq:ineq-tau}) in Appendix \ref{appendix:monochromatic}. In particular, for $r$ sufficiently large, $$|Q| = z \ge \frac{1}{10}(m/\phi(m))\tau(r,m)\frac{y}{r} > 64(m/\phi(m))\tau(r,m)\frac{y}{r\log r}.$$ The next lemma shows that $v|m$. 

\begin{lem} \label{lem:div-by-divisors}
If there exist at least $64(m/\phi(m))\tau(r,m)\frac{y}{r\log r}$
elements in $Y$ which are divisible by $v$, then $v|m$ and $v\le y^{1/16}$. Furthermore, all elements of $Y$ which are divisible by $v$ have the form $qvu$, where $(vu)|m$, $vu \le y^{1/16}$ and $\gcd(q,W(r)m) = 1$. 
\end{lem}

\begin{proof}
Note that if $\gcd(q,W(r)m)=1$ and $q\ne1$, then any prime factor of $q$ is at least $p_r>r \log r/8$ for sufficiently large $r$. Recall that elements of $Y$ have the form $qu$, where $u|m$, $u\le y^{1/16}$ and $\gcd(q,W(r)m)=1$. Assume that there exists $v$ such that either $v\nmid m$ or $v>y^{1/16}$ and at least $64(m/\phi(m))\tau(r,m)\frac{y}{r\log r}$ elements in $Y$ are divisible by $v$. We claim that $v$ must have a prime factor $p$ which is coprime to $W(r)m$. Indeed, if this were not the case, then $v$ only has prime factors which are divisors of $W(r)m$, so $\gcd(v,q)=1$ for any $q$ coprime to $W(r)m$. Thus, if an element of the form $qu$ with $u|m$, $u\le y^{1/16}$ and $\gcd(q,W(r)m)=1$ is divisible by $v$, then $v|u$, so $v|m$ and $v\le y^{1/16}$, contradicting our assumption. Thus, $v$ has a prime factor $p$ which is coprime to $W(r)m$. In particular, $p \ge p_r>r(\log r)/8$. 

We have that at least $64(m/\phi(m))\tau(r,m)\frac{y}{r\log r}$ elements
of $Y$ are divisible by $p$. For each element $qu$ of $Y$ which is divisible by $p$, since $p$ is coprime to $W(r)m$, we must have $p|q$, so $q=q'p$ for $q'$ coprime to $W(r)m$. Hence, elements of $Y$ which are divisible by $p$ have the form $pq'u$ where $u|m$, $u\le y^{1/16}$ and $\gcd(q',W(r)m)=1$. If $y/p \geq n^{1/4}$, Lemma \ref{lem:pseudoprimes-count} implies that the number of such elements is at most $8(m/\phi(m))\tau(r,m)\frac{y}{p}<64(m/\phi(m))\tau(r,m)\frac{y}{r\log r}$. If $y/p < n^{1/4}$, then the number of such elements is at most $y/p < n^{1/4} < 64(m/\phi(m))\tau(r,m)\frac{y}{r\log r}$, where the second inequality is verified as inequality (\ref{eq:ycond2}) of Claim \ref{claim:cond-for-y} in Appendix \ref{appendix:monochromatic}. In either case, we have a contradiction, so we must have that $v|m$ and $v\le y^{1/16}$. 

Since $v|m$, we have $\gcd(v,W(r)m)=v$, so each element of the form $qu$ where $u|m$, $u\le y^{1/16}$ and $\gcd(q,W(r)m)=1$ which is divisible by $v$ must have $v|u$. Hence, $qu=qvu'$ where $(vu')|m$, $vu'\le y^{1/16}$ and $\gcd(q,W(r)m)=1$, establishing the second claim in the lemma. 
\end{proof}

Since $|Q| > 64(m/\phi(m))\tau(r,m)\frac{y}{r\log r}$ and $\{vx:x\in Q\}$ is a subset of $Y$, Lemma \ref{lem:div-by-divisors} implies that each element of $Q$ can be written in the form $qu$, where $(vu)|m$, $vu\le y^{1/16}$ and $\gcd(q,W(r)m)=1$. Let 
\[
Y_{v}=\{t\in[y/v,2y/v):t=qu,u|(m/v),u\le y^{1/16}/v,\gcd(q,W(r)m)=1\}.
\]
We have that $vt \in Y$ for all $t\in Y_{v}$ and $Q\subseteq Y_{v}$.

Let $V$ be a random subset of $Q$ of size $z/8$. The next lemma implies that $V$ is $y^{1/4}/16$-diverse with probability at least $1/2$. From now, we fix $V$ to be a subset of $Q$ of size $z/8$ which is $y^{1/4}/16$-diverse.

\begin{lem}
\label{lem:diverse}Let $k$ and $h$ be positive integers with $h\neq 1$ and $N=\exp(k/16h)$.
Let $A$ be a set of $t$ integers in $[N]$ which is $k$-diverse. Let $B$ be a uniformly random subset of $A$ of
size $t/h$. Then $B$ is $k/(2h)$-diverse with probability at least $1-1/N$.
\end{lem}

\begin{proof}
For each $d \in [N]$ with $d>1$, let $X_d$ be the set of elements in $A$ which are not divisible by $d$. By our assumption, $|X_d|\ge k$ for each $d$. The number of elements in $B\cap X_d$ follows a hypergeometric distribution. As the hypergeometric distribution is at least as concentrated as the corresponding binomial distribution (for a proof, see Section 6 of \cite{Hoeff}), we can apply the Chernoff bound to obtain that the probability that $|B\cap X_d|<|X_d|/(2h)$ is at most $\exp(-|X_d|/8h) \le \exp(-k/8h)=N^{-2}$. By taking a union bound over all $d \in [N]$ with $d>1$, we conclude that the probability $B$ is not $k/(2h)$-diverse is at most $N \cdot N^{-2} = N^{-1}$. 
\end{proof}

The following lemma is the key to proving Lemma \ref{lem:lem-choosey}. 

\begin{lem}
\label{lem:build-prime-class}Let $\ell =\lceil32/\xi\rceil$, where $\xi$ is the constant in Lemma \ref{lem:Freiman-2.04}. Let $A$ be a subset of $Y_{v}$ of
size $z/(8\ell)$ which is $y^{1/4}/(32\ell)$-diverse. Then $|\Sigma(A)|\ge yz/(\ell^{2}v)$
and $\Sigma(A)$ is not a subset of an arithmetic progression with
common difference greater than $1$. 
\end{lem}

Before moving on to the proof of Lemma \ref{lem:build-prime-class}, we show how Lemma \ref{lem:lem-choosey} follows from it. 

\begin{proof}[Proof of Lemma \ref{lem:lem-choosey} assuming Lemma \ref{lem:build-prime-class}]
Recall that we have fixed a subset $V$ of $Q$ of size $z/8$ which is $y^{1/4}/16$-diverse. We will prove that $\Sigma(V)$ contains an interval $J=[a,b]$ of length
at least $2y/v$. To see why this suffices, first note that
\[
vb\le v\cdot \max \Sigma(V) <2y\cdot\frac{z}{8}\le2y\cdot\frac{(m/\phi(m))\tau(r,m)y}{8\cdot 8r}<m,
\] 
where the third inequality follows from (\ref{eq:z}) and the last inequality follows since $y^2 \le \frac{25rm}{(m/\phi(m))\tau(r,m)}$ by the choice of $y$. 
If now we can find the required interval $J$, Lemma \ref{lem:lem-choosey} follows
since each element of $Q$ is at most $2y/v$ and, hence, by Lemma \ref{lem:verysimple}, $\Sigma(Q)=\Sigma(V\cup(Q\setminus V))$
contains an interval whose smallest element is $a<m/v$ by the inequality above and whose largest element is
\[
b+\sum_{i\in Q\setminus V}i>\frac{y}{v}\cdot z(1-1/8)\ge\frac{y}{v}\cdot\frac{7(m/\phi(m))\tau(r,m)y}{8\cdot10r}>\frac{m}{v},
\] where the last inequality follows since $y^2 \ge \frac{15rm}{(m/\phi(m))\tau(r,m)}$ by the choice of $y$. 
Hence, $\Sigma(Q_0)$ contains the progression $\{va,v(a+1),\dots,v(m/v)\}$, which contains $m$. 

We partition $V$ randomly into $\ell$ sets $V_{1},V_{2},\dots,V_{\ell}$ of size $z/(8\ell)$. By Lemma \ref{lem:diverse} and the union
bound, the probability that $V_{i}$ is $y^{1/4}/(32\ell)$-diverse for all $i\in[\ell]$ is at least $1/2$. Hence, we can fix
a partition of $V$ into $\ell$ sets $V_{1},V_{2},\dots,V_{\ell}$ of size $z/(8\ell)$, where $V_{i}$ is $y^{1/4}/(32\ell)$-diverse for each $i\in[\ell]$. 

For each $i\in[\ell]$, $\Sigma(V_{i})$ is a subset of the interval
$[0,z/(8\ell)\cdot2y/v]=[0,yz/(4\ell v)]$. By Lemma \ref{lem:build-prime-class},
$|\Sigma(V_{i})| \ge yz/(\ell^{2}v)$
and $\Sigma(V_{i})$ is not a subset of an arithmetic progression
with common difference greater than $1$. Therefore, by Lemma \ref{lem:Lev},
$\Sigma(V_{1})+\cdots+\Sigma(V_{\ell})$ contains an interval of length
at least $yz/(2\ell v)>2y/v$ for $n$ sufficiently large, as required. 
\end{proof}

We have therefore reduced the task of proving the lower bound in Theorem \ref{thm:monochromatic-subset-sums} to Lemma \ref{lem:build-prime-class}. The strategy for proving Lemma~\ref{lem:build-prime-class}
is now as follows. We partition $A$ into two subsets $A_{1}$ and $A_{2}$
of size $z/(16\ell)$, observing that we can choose $A_{1}$ and $A_{2}$
to be $y^{1/4}/(128\ell)$-diverse by Lemma \ref{lem:diverse}. 
We then show that $\Sigma(A_{1})$ contains elements
in many different congruence classes modulo $t$ for all $t$ in $A_{2}$, allowing
us to apply Lemma \ref{lem:modp} repeatedly (as in the proofs of our results on completeness) 
to conclude that each element of $A_{2}$ introduces many new elements to the set of subset sums. 

The next lemma is the main step in the proof of
Lemma \ref{lem:build-prime-class}. Recall that 
\[
\Sigma_{t}(A)=\left\{ \sum_{x\in S}x\bmod{t}: S\subseteq A\right\}
\]
and $\xi$ is the absolute constant defined in Lemma \ref{lem:Freiman-2.04}.

\begin{lem}
\label{lem:growth-}Let $t\in[y/v,2y/v)$. Let $A$ be a subset of
$Y_{v}$ of size $z/(16\ell)$ which is $y^{1/4}/(128\ell)$-diverse. Then
$|\Sigma_{t}(A)|\ge\min(\xi,32/\ell)t$. 
\end{lem}

To show that the set of mod $t$
subset sums is large, we prove the following structural lemma, stating that the set of
elements whose inclusion does not expand the set of mod $t$ subset
sums must either be small or additively structured. 
We will then use this additive structure
to show that the corresponding set in $\mathbb{Z}$ must contain a
small number of integers of the form $qu$, which we will see is impossible. 

\begin{lem}
\label{lem:structure}Let $t$ be an integer. Let $A\subseteq\mathbb{Z}_{t}$
be such that $8d<|A|<\xi t$. Let $G_{d}\subseteq\mathbb{Z}_{t}$
be the set of $x$ such that $|(A+x)\cup A|\le|A|+d$. Then either
$G_{d}$ is contained in a proper subgroup of $\mathbb{Z}_{t}$, $|G_{d}|\le\frac{20|A|}{(|A|/2d)^{1.02}}$
or there is a subgroup $H$ of $\mathbb{Z}_{t}$ such that $G_{d}$
is contained in a set of size at most $128d$ which is an arithmetic progression of $H$-cosets.
\end{lem}

\begin{proof}
By Lemma \ref{lem:stable-period}, $kG_{d}\subseteq G_{kd}$, where $kG_{d}=\{x_{1}+x_{2}+\cdots+x_{k}:x_{1},x_{2},\dots,x_{k}\in G_{d}\}$. Let $i=\lfloor\log_{2}(|A|/2d)\rfloor$ and let $k=2^{i}$, noting that $kd \leq |A|/2$. Therefore, applying
Lemma \ref{lem:double-count} to $G_{kd}$, we get 
\begin{equation}
|kG_{d}|\le|G_{kd}|\le\frac{|A|^{2}}{|A|-kd}\le2|A|.\label{eq:boundkG}
\end{equation}

Assume that $G_{d}$ is not contained in a proper subgroup of $\mathbb{Z}_{t}$. Let $j$ be such that $0\le j<i$. Since $0\in G_{d}$ by definition, we have $G_d \subset 2^jG_d$, so $0\in 2^j G_d$ and $2^j G_d$ is not contained in a proper subgroup of $\mathbb{Z}_t$. Thus, $2^{j}G_{d}$ is not contained in a coset of a proper subgroup of $\mathbb{Z}_{t}$. By Lemma \ref{lem:Cauchy-Davenport},
$|2^{i}G_{d}|\ge\min\{t,2^{i-j-1}|2^{j}G_{d}|\}=2^{i-j-1}|2^{j}G_{d}|$,
where we used that $|2^{i}G_{d}|\le2|A|<2\xi t<t$ from (\ref{eq:boundkG}).
Thus, 
\begin{equation}
|2^{j}G_{d}|\le2^{1-i+j}|2^i G_d| < 2\xi t. \label{eq:2j} 
\end{equation}

Assume now that $|2^{j+1}G_{d}|=|2^{j}G_{d}+2^{j}G_{d}|\le2.04|2^{j}G_{d}|$
for some $2\le j<i$. By Lemma \ref{lem:Freiman-2.04}, there exists a proper subgroup $H$ of $\mathbb{Z}_{t}$ such that one of the following holds:
\begin{enumerate}
\item $2^{j}G_{d}$ is contained in a set of size at most $\frac{\ell}{\ell-1}\cdot 1.04|2^{j}G_{d}|$ which is an arithmetic progression of $H$-cosets of length $\ell \ge 2$,
\item $2^{j}G_{d}$ meets exactly three $H$-cosets which are terms of an arithmetic progression of $H$-cosets of length $\ell$ and $(\min(\ell,4)-1)|H|\le 1.04|2^{j}G_{d}|$ or 
\item $2^{j}G_{d}$ is contained in one $H$-coset. 
\end{enumerate}

We have already seen that the third case cannot happen, that is, that $2^{j}G_{d}$ is not contained in a coset of a proper subgroup of $\mathbb{Z}_{t}$. 

Suppose that we are in the second case. Then $2^{j}G_{d}$ is contained in a union of three $H$-cosets, so $4G_d$ is contained in a union of three $H$-cosets. Since $0\in G_d$ and $G_d$ is not contained in an $H$-coset, the image of $G_d$ in $\mathbb{Z}_{t}/H$ is a subset $S$ of $\mathbb{Z}_{t}/H$ of size at least $2$ such that $0\in S$ and $4S$ has size at most $3$. This can only happen if $S$ is contained in a subgroup of $\mathbb{Z}_{t}/H$ of size at most $3$. In this case, $G_{d}$ is contained in a subgroup of $\mathbb{Z}_{t}$ of size at most $3|H|$. Since $|H|\le (\min(\ell,4)-1)|H|\le 1.04|2^{j}G_{d}|$, we have $3|H|\le 3.12|2^{j}G_{d}| < 6.24\xi t < t$, so $G_{d}$ is contained in a proper subgroup of $\mathbb{Z}_{t}$, a contradiction. Thus, the second case cannot happen. 

We now consider the first case, where $2^{j}G_{d}$
is contained in a set of size at most $\frac{\ell}{\ell-1}\cdot 1.04|2^{j}G_{d}|\le 2.08|2^{j}G_{d}|$ which is
an arithmetic progression of $H$-cosets of length $\ell$. As $0\in G_{d}$,
this progression of $H$-cosets contains $0$. Let $m$ be such that
$H=\{x\in\mathbb{Z}_{t}:m|x\}$. Then the $H$-cosets can be identified
with elements of $\mathbb{Z}_{m}$. The common difference of the progression
of $H$-cosets must be coprime to $m$, as otherwise $2^{j}G_{d}$
would be contained in a proper subgroup of $\mathbb{Z}_{t}$. Thus, by rescaling if necessary,
we may assume that the common difference
of the progression of $H$-cosets is $1$. Let $P_{j}$ be the interval in $\mathbb{Z}_{m}$ which corresponds
to the $H$-cosets in the progression containing $2^{j}G_{d}$. Note
that $0\in G_{d}$, so that $2^{h-1}G_{d}\subseteq2^{h}G_{d}$ for all $h$.
Hence, for each $h\le j$, we can choose intervals $P_{h}$ around $0$ in $\mathbb{Z}_{m}$
such that $2^{h}(G_{d}/H) \subseteq P_{h}$, $P_{h}\subseteq P_{j}$ and $2^{j-h}P_h\subseteq P_j$.
The length $\ell$ of $P_{j}$ is at most $m/2$, since otherwise $|2^jG_d| \ge \frac{t/2}{2.08}$, contradicting (\ref{eq:2j}).
We can thus deduce that, for all $h\le j$, $P_{h}$ is an interval of length at most $1+2^{h-j}(|P_j|-1)$ around $0$ in $\mathbb{Z}_{m}$, since, for two intervals $I$, $I'$ around 0 of length $|I|\le |I'| \le m/2$ with $2I\subseteq I'$, we have $|I| \le (|I'|+1)/2$. 
Hence, $G_{d}/H$ is a subset of
an interval of length at most $1+(|P_{j}|-1)/2^{j}$. Since $G_{d}$ is not contained in a proper subgroup of $\mathbb{Z}_{t}$, $(|P_j|-1)/2^j \ge 1$. Thus, we have $1+(|P_{j}|-1)/2^{j} \le |P_{j}|/2^{j-1}$. Therefore, $G_{d}$ is contained in a union of $H$-cosets of size at most
$$|H||P_{j}|/2^{j-1}\le\frac{\ell}{\ell -1} \cdot 1.04|2^jG_d|\cdot 2^{1-j} \le \frac{\ell}{\ell-1} \cdot 1.04 \cdot 2^{3-i} |A| \le 128d,$$ 
where, in the second inequality, we used (\ref{eq:boundkG}) and (\ref{eq:2j}) and, in the final inequality, we used that $i=\lfloor\log_{2}(|A|/2d)\rfloor$ and $\ell \geq 2$. 

If there does not exist $j\in [2,i)$ such that $|2^{j+1}G_d| \le 2.04|2^jG_d|$, then 
\[
|2^{i}G_{d}|\ge2.04^{i-2}|G_{d}|\ge(|A|/2d)^{1.02}|G_{d}|/10.
\]
Combining this with (\ref{eq:boundkG}), we deduce that $|G_{d}|\le\frac{20|A|}{(|A|/2d)^{1.02}}$. 
\end{proof}

Besides Lemma \ref{lem:structure}, we need several other ingredients for the proof of Lemma \ref{lem:growth-}. We begin with the following result, which will also be useful to us in subsequent sections. For this section, the key corollary is that if $A$ is $k$-diverse for $k\ge d-1$, then $\Sigma_d(A)=\mathbb{Z}_d$.

\begin{lem}
\label{lem:full-mod-d}Let $d$ be a positive integer. Let $A$ be a set of integers such that, for each $d'|d$, at least $d'-1$ elements of $A$ are not divisible by $d'$. Then $\Sigma_{d}(A)=\mathbb{Z}_{d}$. Furthermore, if $A$ contains at least $d-1$ elements which are not divisible by $d$, then $\Sigma_d(A)$ contains a non-zero subgroup of $\mathbb{Z}_{d}$. 
\end{lem}
\begin{proof}
We will use the following simple claim.

\vspace{2mm}

\noindent {\bf Claim.} If $S$ is a subset of $\mathbb{Z}_t$ and $X \subseteq \mathbb{Z}_t$ is such that $|(S+x)\cup S| = |S|$ for all $x\in X$, then $S$ is a union of cosets of the subgroup of $\mathbb{Z}_t$ spanned by $X$. 

\vspace{2mm}

\noindent {\it Proof.} If $|(S+x)\cup S| = |S|$, then $S+x = S$. Thus, by induction, we have that $S+x_1+\dots+x_k = S$ for all $k\in \mathbb{N}$ and $x_1,\dots,x_k\in X$. In particular, we have $S + \langle X\rangle = S$, where $\langle X\rangle$ is the subgroup of $\mathbb{Z}_t$ spanned by $X$. Since $S + \langle X\rangle$ is a union of cosets of $\langle X\rangle$, we obtain the desired conclusion. \qed

\vspace{2mm}

Note that $\Sigma_t(S\cup \{x\}) = \Sigma_t(S) \cup (\Sigma_t(S)+x)$. From the claim, if $S$ is a multiset in $\mathbb{Z}_t$ and $x\in \mathbb{Z}_t$ is coprime to $t$, then we have $|\Sigma_t(S\cup \{x\})| \ge \min(|\Sigma_t(S)|+1,t)$, as either $|\Sigma_t(S\cup \{x\})| \ge |\Sigma_t(S)|+1$ or $\Sigma_t(S)$ is a union of cosets of $x\mathbb{Z}_t = \mathbb{Z}_t$, so $\Sigma_t(S) = \mathbb{Z}_t$. Thus, if $B=\{b_1,\dots,b_{t-1}\}$ is a multiset of size $t-1$ consisting
of elements in $\mathbb{Z}_{t}$ coprime to $t$, then $\Sigma_{t}(B)=\mathbb{Z}_{t}$.
Indeed, this follows easily from the fact that $\Sigma_t(\emptyset) = \{0\}$ and, for each $i\ge 1$, $|\Sigma_t(\{b_1,\dots,b_i\})| \ge \min(|\Sigma_t(\{b_1,\dots,b_{i-1}\})|+1,t)$. 

Suppose now that $A$ is a set of integers such that, for each $d'|d$, at least $d'-1$ elements of $A$ are not divisible by $d'$.
We will prove that $\Sigma_{d'}(A)=\mathbb{Z}_{d'}$ for all $d'|d$ by induction on the number of prime factors (counted
with repeats) of $d'$. When $d'$ is a prime, the conclusion follows
from the observation above. Assume now that the conclusion holds whenever $d'$ has at most $j$ prime factors, for some $j\ge 1$.  

Let $d'$ be a divisor of $d$ with $j+1$ prime factors. Let $A_0$ be the multiset of elements in $A$ not divisible by $d'$, considered modulo $d'$. By our assumption, $A_0$ has size at least $d'-1$. Observe that $\Sigma_{d'}(A)=\Sigma_{d'}(A_0)$. Assume that $|\Sigma_{d'}(A_0)|<d'$. Let $\Sigma_{d'}(0)=\{0\}$. We consider the following iterative process. At step $i\ge 1$, we choose $a_{i}\in A_{i-1}$
so that $|(\Sigma_{d'}(i-1)+a_{i})\setminus\Sigma_{d'}(i-1)|$ is maximized and
let $\Sigma_{d'}(i)=\Sigma_{d'}(i-1)\cup(\Sigma_{d'}(i-1)+a_{i})$ and $A_{i}=A_{i-1}\setminus\{a_{i}\}$. Note that we consider the $A_i$ as a multiset of elements of $\mathbb{Z}_{d'}$ and the $\Sigma_{d'}(i)$ as subsets of $\mathbb{Z}_{d'}$. 

Let $i\le |A_0|$ be the first
step where $|\Sigma_{d'}(i)|\le |\Sigma_{d'}(i-1)|$. Note that $i$ must exist since, otherwise, $|\Sigma_{d'}(|A_0|)| \ge |\Sigma_{d'}(0)| + |A_0| = 1+|A_0|\ge d'$, contradicting our assumption that $|\Sigma_{d'}(|A_0|)| = |\Sigma_{d'}(A_0)| < d'$. 
Since $i$ is the first step with $|\Sigma_{d'}(i)|\le |\Sigma_{d'}(i-1)|$, we must have that $|\Sigma_{d'}(j)| \ge |\Sigma_{d'}(j-1)|+1$ for all $j<i$. Thus, $|\Sigma_{d'}(i)|\ge 1+i-1=i$. In step $i$, we have $|(\Sigma_{d'}(i-1) + a) \cup \Sigma_{d'}(i-1)|=|\Sigma_{d'}(i-1)|$ for all $a\in A_{i-1}$, so, by the claim, $\Sigma_{d'}(i-1)$ is a union of cosets of the subgroup of $\mathbb{Z}_{d'}$ spanned by $A_{i-1}$. Let $d''$ be the largest divisor of $d'$ which divides all elements in $A_{i-1}$. Then the subgroup of $\mathbb{Z}_{d'}$ spanned by $A_{i-1}$ is $d''\mathbb{Z}_{d'}$ and we have that $\Sigma_{d'}(i-1)$ is a union of $d''\mathbb{Z}_{d'}$-cosets.  
Note that $d''\ne d'$, since the elements of $A_0$ are not divisible by $d'$ and $A_{i-1}$ contains at least one element in $A_0$. Thus, $d''<d'$ and, hence, $d''$ has at most $j$ prime factors. By the induction hypothesis, $\Sigma_{d''}(A_0)=\mathbb{Z}_{d''}$.
Note that $\Sigma_{d''}(A_0)=\Sigma_{d''}(\{a_{1},\dots,a_{i-1}\})$, since
all remaining elements of $A_0$ are divisible by $d''$. Thus, $\Sigma_{d'}(i-1)$ contains an element in each $d''\mathbb{Z}_{d'}$-coset of $\mathbb{Z}_{d'}$. Since $\Sigma_{d'}(i-1)$ is a union of $d''\mathbb{Z}_{d'}$-cosets and
contains an element in each $d''\mathbb{Z}_{d'}$-coset of $\mathbb{Z}_{d'}$,
$\Sigma_{d'}(i-1)$ contains all elements of $\mathbb{Z}_{d'}$. Thus, $\Sigma_{d'}(A)=\Sigma_{d'}(i-1)=\mathbb{Z}_{d'}$, completing the induction.

For the second statement, observe, by the claim, that if $|\Sigma_{d}(S\cup\{x\})|=|\Sigma_{d}(S)|$, then $\Sigma_{d}(S)$ is a union of cosets of $x\mathbb{Z}_d$ and, as $0\in \Sigma_d(S)$, we have that $\Sigma_d(S)$ contains the subgroup $x\mathbb{Z}_d$ of $\mathbb{Z}_d$. Thus, if $A$ contains at least $d-1$ integers $a_1,\dots,a_{d-1}$ not divisible by $d$ and $\Sigma_d(A)$ does not contain a non-zero subgroup of
$\mathbb{Z}_{d}$, then we must have $|\Sigma_{d}(\{a_1,\dots,a_i\})| \ge |\Sigma_{d}(\{a_1,\dots,a_{i-1}\})|+1$ for all $i\ge 1$. But then $|\Sigma_d(A)|\ge1+(d-1)=d$, which means that $\Sigma_d(A)$ equals $\mathbb{Z}_{d}$. 
\end{proof}

We remark that the condition in the above lemma is tight, since if $d$ is prime and $A$ consists of $d-2$ elements congruent to $1$ modulo $d$, then $\Sigma_d(A)$ does not contain any non-zero subgroup of $\mathbb{Z}_d$.  

The next lemma gives an upper bound on the number of integers coprime to $W(r)/\gcd(W(r),m)$ in an arithmetic progression. Note that all integers of the form $qu$ where $u|m$ and $\gcd(q,W(r)m)=1$ are coprime to $W(r)/\gcd(W(r),m)$. The proof of this lemma uses the Selberg sieve and may be found in Appendix \ref{sec:number-theoretic}. 

\begin{lem}\label{lem:selberg} Let $r$ and $n$ be sufficiently large positive integers and $m\in [n,\binom{n}{2}]$. Let $X$ be an arithmetic progression of size
$|X|\ge r^{1/16}$ with common difference $b\le n$. Then the number of elements
of $X$ which are coprime to $W(r)/\gcd(W(r),m)$ is at most 
\[
\frac{256|X|\log\log n}{\log r}.
\]
Furthermore, when $b=1$, the number of elements
of $X$ which are coprime to $W(r)/\gcd(W(r),m)$ is at most 
\[
256|X|\prod_{p|W(r),p\nmid m}(1-1/p). 
\]
\end{lem}

Given a cyclic group $\mathbb{Z}_t$ and an interval of integers $[x,x+t)$, we have a natural identification $\psi_t:\mathbb{Z}_t \to [x,x+t)$, where $\psi_t(u)$ is the unique integer in $[x,x+t)$ which is congruent to $u$ modulo $t$. The next lemma shows that under this identification, for a subgroup $H$ of $\mathbb{Z}_t$, a progression of $H$-cosets is identified with a large subset of a union of long arithmetic progressions of integers. A variant of this lemma goes back at least to the proof of Roth's Theorem \cite{R53}.

\begin{lem} \label{lem:ap-mod-t} 
Let $H$ be a subgroup of $\mathbb{Z}_t$ and let $R$ be an arithmetic progression of $H$-cosets. Consider the image $\psi_t(R)$ of $R$ under the identification $\psi_t:\mathbb{Z}_t\to [x,x+t)$. Then $\psi_t(R)$ is contained in a set of size at most $3|R|$ which is a union of arithmetic progressions of integers, each of length at least $|R|^{1/3}$. 
\end{lem}

\begin{proof}
First observe that the image under $\psi_t$ of each $H$-coset is an arithmetic progression. Thus, if $|H| \ge |R|^{1/3}$, then $\psi_t(R)$ is a union of arithmetic progressions, each of length at least $|R|^{1/3}$. 

Assume now that $|H| < |R|^{1/3}$. Let $H = d\mathbb{Z}_t$ for some divisor $d$ of $t$. Let $\ell = |R|/|H|\ge |R|^{2/3}$. By definition, we can write $R = \bigcup_{i\in [\ell]} (x+iy+H)$ for some $x,y\in \mathbb{Z}_t$ and $y\notin H$. For each $H$-coset $x+iy+H$, we can choose a representative $z_i$ for the coset in $[x,x+d)$. Let $P=(z_1,z_2,\dots,z_\ell)$. We have that $P \pmod d$ forms a progression of common difference $u$ in $\mathbb{Z}_d$. We show that $P$ is contained in a set of size at most $3\ell$ which is a union of progressions of integers, each of length at least $\sqrt{\ell}=\sqrt{|R|/|H|} \ge |R|^{1/3}$. From this claim, the conclusion of the lemma easily follows. 

We claim that there exists $s\in [1,\lfloor \sqrt{\ell}\rfloor]$ such that $su$ is congruent to an integer in $[-d/\lceil \sqrt{\ell} \rceil,d/\lceil \sqrt{\ell} \rceil]$ modulo $d$. Partition $\mathbb{Z}_{d}$ into a union of intervals $[kd/\lceil \sqrt{\ell}\rceil,(k+1)d/\lceil \sqrt{\ell} \rceil)$ for $k=0,1,\dots,\lceil \sqrt{\ell} \rceil -1$.  
Suppose that there does not exist $s\in [1,\lfloor \sqrt{\ell}\rfloor]$ such that $su \pmod d \in [-d/\lceil \sqrt{\ell}\rceil,d/\lceil \sqrt{\ell}\rceil]$. Then $u,2u,\dots,\lfloor \sqrt{\ell}\rfloor u \pmod d$ must be contained in the intervals $[kd/\lceil \sqrt{\ell}\rceil,(k+1)d/\lceil \sqrt{\ell} \rceil)$ for $k=1,\dots,\lceil \sqrt{\ell} \rceil - 2$. Since $\lfloor \sqrt{\ell}\rfloor \ge \lceil \sqrt{\ell} \rceil -1$, the pigeonhole principle implies that there are $1\le s'<s'' \le \lfloor \sqrt{\ell}\rfloor$ such that $s'u \pmod d$ and $s''u \pmod d$ are in the same interval $[kd/\lceil \sqrt{\ell}\rceil,(k+1)d/\lceil \sqrt{\ell} \rceil)$. Then $s''-s'\in [1,\lfloor \sqrt{\ell}\rfloor]$ and $(s''-s')u$ is congruent to an integer $v$ in $[-d/\lceil \sqrt{\ell} \rceil,d/\lceil \sqrt{\ell} \rceil]$ modulo $d$, contradicting our assumption. 

Suppose now that $s\in [1,\lfloor \sqrt{\ell}\rfloor]$ is such that $su$ is congruent to an integer in $[-d/\lceil \sqrt{\ell} \rceil,d/\lceil \sqrt{\ell} \rceil]$ modulo $d$. Since $P \pmod d$ forms a progression of common difference $u$ in $\mathbb{Z}_d$, 
we can partition $P$ into $s$ subsets $P_1,P_2,\dots,P_s$ such that $P_i \pmod d$ is a progression with common difference $su$ in $\mathbb{Z}_d$.
Each set $P_i$ can be greedily partitioned into progressions of integers with common difference $v$ such that all of the progressions in the partition, except the first and last ones, have length at least $\lceil \sqrt{\ell} \rceil$. 
By extending arbitrarily the progressions with length less than $\lceil \sqrt{\ell}\rceil$, we obtain that $P$ is contained in a union of arithmetic progressions of integers, each of length at least $\lceil\sqrt{\ell}\rceil$, where the size of the union is at most $(\lceil\sqrt{\ell}\rceil - 1)\cdot 2s + |P| \le 3|P|$. This verifies the desired claim. 
\end{proof}

We will also need the following simple lemma in the proof of Lemma \ref{lem:growth-}.

\begin{lem}\label{lem:Su}
Let $A$ be a multiset of elements of $\mathbb{Z}_t$ and let $d$ be a divisor of $t$. Then, for any $u\in \mathbb{Z}_t/d\mathbb{Z}_t$ such that $\Sigma_t(A)\cap (u+d\mathbb{Z}_t)\ne \emptyset$, 
$$|\Sigma_t(A) \cap (u+d\mathbb{Z}_t)| \ge |\Sigma_t(A \cap d\mathbb{Z}_t)|.$$
\end{lem}

\begin{proof}
Let $S_u = \Sigma_t(A) \cap (u+d\mathbb{Z}_t)$. For all non-zero $u\in \mathbb{Z}_{t}/d\mathbb{Z}_{t}$, if $S_u \ne \emptyset$, then we can find an element $x$ in $S_u$ which is a sum of distinct elements of $A$ which are not in $d\mathbb{Z}_t$. Thus, each element of $x+\Sigma_t(A \cap d\mathbb{Z}_t)$ can be written a sum of distinct elements in $A$, so $x+\Sigma_t(A \cap d\mathbb{Z}_t) \subseteq \Sigma_t(A)$. It is also clear that $x+\Sigma_t(A\cap d\mathbb{Z}_t) \subseteq u+d\mathbb{Z}_t$, so $x+\Sigma_t(A \cap d\mathbb{Z}_t) \subseteq S_u$. If $u=0 \in \mathbb{Z}_t/d\mathbb{Z}_t$, then letting $x=0$, we have $x+\Sigma_t(A \cap d\mathbb{Z}_t) = \Sigma_t(A \cap d\mathbb{Z}_t) \subseteq S_u$. Thus, if $S_u$ is non-empty, then $|S_u| \ge |\Sigma_t(A \cap d\mathbb{Z}_t)|$.
\end{proof}

We can now prove Lemma \ref{lem:growth-}. We recall the statement, that if $t\in[y/v,2y/v)$ and $A$ is a subset of
$Y_{v}$ of size $z/(16\ell)$ which is $y^{1/4}/(128\ell)$-diverse, then $|\Sigma_{t}(A)|\ge\min(\xi,32/\ell)t$. 

\begin{proof}[Proof of Lemma \ref{lem:growth-}]
We consider the following iterative process. Let $\Sigma_t(0)=\{0\}$ and $A_0 = A$. At each step $i\ge1$,
we pick an element $a_{i}$ in $A_{i-1}$ and let $\Sigma_{t}(i)=\Sigma_{t}(i-1)\cup(\Sigma_{t}(i-1)+a_{i})$
and $A_{i}=A_{i-1}\setminus\{a_{i}\}$. In particular, $\Sigma_{t}(i)= \Sigma_{t}(\{a_1,\dots,a_i\}) \subseteq \Sigma_{t}(A)$ for all $i\le |A|$. Let $d_{i}=\gcd(A_{i-1})$. For $i\le |A|/2$, $A_{i-1}$ is a subset of $Y_v$ of size at least $\frac{|A|}{2}\ge \frac{z}{32\ell} > 64(m/\phi(m))\tau(r,m)\frac{y}{r\log r}$, where we used (\ref{eq:z}) and assumed that $r$ is sufficiently large in terms of $\ell$. Thus, $\{vx:x\in A_{i-1}\}$ is a subset of $Y$ of size larger than $64(m/\phi(m))\tau(r,m)\frac{y}{r\log r}$ whose elements are divisible by $vd_{i}$. By Lemma \ref{lem:div-by-divisors}, we obtain $vd_{i}\le y^{1/16}$ and all elements of $A_{i-1}$ have the form $qu$ where $u \mid (m/v)$, $u\le y^{1/16}/v \le y^{1/16}$, $d_i|u$ and $\gcd(q,W(r)m)=1$. We will run the above process for at most $|A|/2$ steps, so we may assume that $i \leq |A|/2$ and these conclusions hold throughout.

For each $i$, we say that step $i$ is either a {\it growth phase}, an {\it unsaturated phase} or a {\it saturated phase}. Note that the cosets of $d_{i}\mathbb{Z}_{t}$ can be indexed by elements of $\mathbb{Z}_{d_{i}}$. For each $u\in\mathbb{Z}_{d_{i}}$, let $S_{u}=\Sigma_t(i-1)\cap(u+d_{i}\mathbb{Z}_{t})$. We say that $i$ is {\it a growth phase} if there exists $u\in \mathbb{Z}_{d_{i}}$ such that $S_u$ is non-empty and has size at most $y^{3/4}$. We say that $i$ is {\it an unsaturated phase} if it is not a growth phase and there exists $u\in \mathbb{Z}_{d_{i}}$ such that $y^{3/4}<|S_u|<\frac{\xi t}{d_i}$. Finally, if step $i$ is neither a growth phase nor an unsaturated phase, then it is a {\it saturated phase}. 

Next we describe how $a_{i}$ is chosen. Let $\Sigma(d,i-1)=\{\sum_{j\in S}a_{j}\pmod t:S\subseteq[i-1]\cap\{j:d|a_{j}\}\}$. Then $\Sigma(d,i-1)$ is a subset of the subgroup $d\mathbb{Z}_t$ of $\mathbb{Z}_t$, which can be identified with $\mathbb{Z}_{t/d}$. We identify $\Sigma(d,i-1)$ with a subset of $\mathbb{Z}_{t/d}$. Similarly, we can identify $A_{i-1}$ with a subset of $\mathbb{Z}_{t/d_i}$. 
If $i$ is a growth phase, we pick $a_{i}$ such that $|\Sigma(d_{i},i)|-|\Sigma(d_{i},i-1)|$
is maximized. Otherwise, if $i$ is not a growth phase, we pick
$a_{i}$ such that $|\Sigma_{t}(i)|-|\Sigma_{t}(i-1)|$ is maximized. 

The following claims capture the key steps in the proof. 

\vspace{2mm} 

\noindent{\bf Claim 1.} The number of growth phases among the first $|A|/2$ steps is at most $(256\ell y^{3/4}/z+\log_{3/2}t+2)(\log_2y^{1/16}+1)$.

\vspace{2mm}

\noindent{\bf Claim 2.} If $i\le |A|/2$ is an unsaturated phase, then $|\Sigma_t(i)|-|\Sigma_t(i-1)| > 2^{12}y/z$. 

\vspace{2mm}

Claim 2 is the most important step in the proof and will take up most of our time. However, before proving these claims, let us see how Lemma \ref{lem:growth-} follows from combining them. 

First, suppose that there exists $i\le |A|/2$ such that $i$ is a saturated phase. By Lemma \ref{lem:full-mod-d}, since $A$ is $y^{1/4}/(128\ell)$-diverse and $d_i < y^{1/4}/(128\ell)$, $\Sigma_{d_i}(\{a_1,\dots,a_{i-1}\})=\Sigma_{d_i}(A)=\mathbb{Z}_{d_i}$. Hence, $S_{u}$ is non-empty for all $u\in\mathbb{Z}_{d_{i}}$. Since $i$ is a saturated phase, we have that $|S_u| \ge \frac{\xi t}{d_i}$ for all $u\in \mathbb{Z}_{d_i}$, so $|\Sigma_t(i-1)| = \sum_{u\in \mathbb{Z}_{d_i}}|S_u| \ge \xi t$. Therefore, $|\Sigma_t(A)| \ge |\Sigma_t(i-1)| \ge \xi t$, as desired. 

Next, suppose that no $i\le |A|/2$ is a saturated phase. In this case, if $i\le |A|/2$ is not a growth phase, it must be an unsaturated phase and, by Claim 2, we have $|\Sigma_t(i)| > |\Sigma_t(i-1)| + 2^{12}y/z$. Since Claim 1 implies that there are at least $|A|/2 - (256\ell y^{3/4}/z+\log_{3/2}t+2)(\log_2y^{1/16}+1)$ unsaturated phases in the first $|A|/2$ steps and $|A| = z/(16\ell)$, we have 
\[
|\Sigma_{t}(A)|\ge|\Sigma_t(|A|/2)|\ge\frac{2^{12}y}{z}\cdot\left(\frac{z}{32\ell}-(256\ell y^{3/4}/z+\log_{3/2} t+2)(\log_2 y^{1/16}+1)\right)\ge \frac{2^{12}y}{64\ell}\ge\frac{32t}{\ell},\]
as required.
\end{proof}

We next give the proofs of Claims 1 and 2, beginning with the simpler of the two.
 
\begin{proof}[Proof of Claim 1]
First, we show that in each step $i \le |A|/2$, if $|\Sigma(d_{i},i-1)|<|A_{i-1}|/2$, then 
\begin{equation}
|\Sigma(d_{i},i)|-|\Sigma(d_{i},i-1)| \ge \max_{a\in A_{i-1}} |(\Sigma(d_{i},i-1)+a)\setminus \Sigma(d_{i},i-1)| \ge |\Sigma(d_{i},i-1)|/2, \label{eq:grow-Sigma-1}
\end{equation}
while if $|A_{i-1}|/2 \le |\Sigma(d_{i},i-1)| \le y^{3/4}$, then 
\begin{equation}
|\Sigma(d_{i},i)|-|\Sigma(d_{i},i-1)| \ge \max_{a\in A_{i-1}} |(\Sigma(d_{i},i-1)+a)\setminus \Sigma(d_{i},i-1)| \ge |A_{i-1}|/8. \label{eq:grow-Sigma-2}
\end{equation}

The first bound (\ref{eq:grow-Sigma-1}) follows directly from Lemma \ref{lem:double-count}, since the set of elements $a\in \mathbb{Z}_{t/d_i}$ for which $$|(\Sigma(d_{i},i-1)+a)\setminus \Sigma(d_{i},i-1)| < |\Sigma(d_{i},i-1)|/2$$has size at most $2|\Sigma(d_{i},i-1)| < |A_{i-1}|$. 

For the second bound (\ref{eq:grow-Sigma-2}), assume, for the sake of contradiction, that for some step $i \le |A|/2$ where $|A_{i-1}|/2\le |\Sigma(d_{i},i-1)| \le y^{3/4}$, $|\Sigma(d_{i},i)|-|\Sigma(d_{i},i-1)| < |A_{i-1}|/8$. Then, for all $a\in A_{i-1}$, we have $ |(\Sigma(d_{i},i-1)+a)\setminus \Sigma(d_{i},i-1)| < |A_{i-1}|/8$. Let $k = \left \lfloor \frac{4|\Sigma(d_i,i-1)|}{|A_{i-1}|}\right\rfloor$ and let
\[
T = \{a\in \mathbb{Z}_{t/d_i} : |(\Sigma(d_{i},i-1)+a)\setminus \Sigma(d_{i},i-1)| < |A_{i-1}|/8\}.
\]
We have $A_{i-1}\subseteq T$ and $0\in T$. By Lemma \ref{lem:stable-period}, for any $a\in kT$, 
$$|(\Sigma(d_{i},i-1)+a)\setminus \Sigma(d_{i},i-1)| < |\Sigma(d_{i},i-1)|/2.$$ 
Hence, by Lemma \ref{lem:double-count}, we have $|k(A_{i-1}\cup \{0\})| \le |kT| \le 2|\Sigma(d_i,i-1)|$. Using that $vd_i \le y^{1/16}$ and $t \ge y/v$, we have $|\Sigma(d_i,i-1)|\le y^{3/4} < t/(2d_i)$, so $|k(A_{i-1}\cup \{0\})| < t/d_i$. Identified as a subset of $\mathbb{Z}_{t/d_i}$, $A_{i-1}$ is not a subset of any proper subgroup of $\mathbb{Z}_{t/d_i}$ by the definition of $d_i$, so $A_{i-1}\cup \{0\}$ is not contained in any coset of a proper subgroup of $\mathbb{Z}_{t/d_i}$. Therefore, by Lemma \ref{lem:Cauchy-Davenport}, we have 
\[
|A_{i-1}\cup \{0\}| \le \frac{2}{k+1} |k(A_{i-1}\cup \{0\})| \le \frac{4|\Sigma(d_i,i-1)|}{k+1} < \frac{4|\Sigma(d_i,i-1)|}{4|\Sigma(d_i,i-1)|/|A_{i-1}|} = |A_{i-1}|,
\]
a contradiction. 

Using (\ref{eq:grow-Sigma-1}) and (\ref{eq:grow-Sigma-2}), we may quickly complete the proof of Claim 1. Note that, by Lemma \ref{lem:Su}, for any $u$ such that $S_u$ is non-empty, $|S_u| \ge |\Sigma(d_{i},i-1)|$. Thus, if $i$ is a growth phase, then $|\Sigma(d_i,i-1)|\le y^{3/4}$. 
Note that $d_{i}|d_{i+1}$, so either $d_{i+1}=d_{i}$ or $d_{i+1}\ge2d_{i}$.
As $d_{i}\le y^{1/16}$ for $i\le |A|/2$, $d_{i}$ can change at most $1+\log_2 y^{1/16}$
times in the first $|A|/2$ steps. 
By (\ref{eq:grow-Sigma-1}), if $|\Sigma(d_{i},i-1)|<|A_{i-1}|/2$, then $|\Sigma(d_{i},i)|-|\Sigma(d_{i},i-1)| \ge |\Sigma(d_i,i-1)|/2$. Thus, for each period among the first
$|A|/2$ steps where $d_{i}$ remains constant, the number of steps
where $|\Sigma(d_{i},i-1)|<|A_{i-1}|/2$ is at most $1+\log_{3/2}t$,
since, in each such step, $|\Sigma(d_{i},i-1)|$ grows by a factor of at least $3/2$. For the remaining steps in this period, where $|A_{i-1}|/2 \le |\Sigma(d_{i},i-1)| \le y^{3/4}$, (\ref{eq:grow-Sigma-2}) implies that $|\Sigma(d_{i},i)|-|\Sigma(d_{i},i-1)| \ge |A_{i-1}|/8 \ge z/(256\ell)$ in each step, so there are at most
$1 + 256\ell y^{3/4}/z$ more growth phases where $d_{i}$
stays constant. Thus, the number of growth phases among the first $|A|/2$
steps in each period where $d_{i}$ stays constant is at most $256 \ell y^{3/4}/z+\log_{3/2}t+2$.
Since $d_{i}$ can change at most $1+\log_2 y^{1/16}$ times in the first $|A|/2$ steps, there are
at most $(256\ell y^{3/4}/z+\log_{3/2}t+2)(\log_2y^{1/16}+1)$ growth phases in the
first $|A|/2$ steps. \qedhere
\end{proof}

Finally, we give the proof of Claim 2, thereby completing the proof of Lemma \ref{lem:growth-}. 
\begin{proof}[Proof of Claim 2]
Let $i$ be an unsaturated phase with $i\le |A|/2$. Assume, for the sake of contradiction, that $|\Sigma_t(i)|-|\Sigma_t(i-1)| \le 2^{12}y/z$. Since $i$ is not a growth phase and not a saturated phase, there exists $u\in \mathbb{Z}_{d_i}$ such that $y^{3/4}<|S_{u}| < \frac{\xi t}{d_{i}}$. 

We now view $S_{u}$ and $A_{i-1}$ as subsets of $\mathbb{Z}_{t/d_{i}}$. Note that, by the definition of $d_i$, $A_{i-1}$ is not a subset of any proper subgroup of $\mathbb{Z}_{t/d_i}$. Let $B$ be the set of elements $a$ of $\mathbb{Z}_{t/d_{i}}$ such that $|(S_u+a)\setminus S_u| \le 2^{12}y/z$. By our choice of $a_i$ and our assumption that $|\Sigma_t(i)|-|\Sigma_t(i-1)| \le 2^{12}y/z$, we have $A_{i-1} \subseteq B$. Since $|S_u|< \frac{\xi t}{d_i}$ and $|S_u|> y^{3/4} > 8 \cdot 2^{12}y/z$ by (\ref{eq:z}), we can apply Lemma \ref{lem:structure} to conclude that either the set $B$ is contained in a proper subgroup of $\mathbb{Z}_{t/d_i}$, $B$ has size at most $\frac{20(2^{13}y/z)^{1.02}}{|S_{u}|^{0.02}}$ or there is a subgroup $H$ of $\mathbb{Z}_{t/d_{i}}$ such that $B$ is contained in a set of size at most $2^{20}y/z$ which is an arithmetic progression of $H$-cosets.
The first possibility cannot hold, since $B$ contains $A_{i-1}$ which is not a subset of any proper subgroup of $\mathbb{Z}_{t/d_i}$. The second possibility also cannot hold, since 
$$\frac{20(2^{13}y/z)^{1.02}}{|S_{u}|^{0.02}}\le\frac{2^{5+13\cdot1.02}y^{1.005}}{z^{1.02}}<|A|/2 \le |A_{i-1}| \le |B|,$$ 
where we used the bound $|S_u| > y^{3/4}$.  
Therefore, there is a subgroup $H$ of $\mathbb{Z}_{t/d_{i}}$ such that
$A_{i-1}$, identified as a subset of $\mathbb{Z}_{t/d_i}$, is contained in a set $R$ of size at most $2^{20}y/z$ which is an arithmetic progression of $H$-cosets. We can identify the elements
of $\mathbb{Z}_{t/d_{i}}$ with elements in $[y/v,y/v+t)\supseteq[y/v,2y/v)$
which are divisible by $d_{i}$. Under this identification, $R$ is identified with a set of integers which contains $A_{i-1}$. 

By Lemma \ref{lem:ap-mod-t}, under the above identification, the image of $R$ is contained in a set of integers of size at most $3|R|$ which is a union of arithmetic progressions $P_s$, $s\in \mathcal{S}$, of integers, each of length at least $|R|^{1/3}$. We have $|R| \ge |A_{i-1}| \ge z/(32\ell)$. Thus, $A_{i-1}$ is contained in a set of size at most $2^{22}y/z$ which is a union of arithmetic progressions $P_s$ of integers, each of length at least $(z/(32\ell))^{1/3}\ge r^{1/16}$, by (\ref{eq:z}) and (\ref{eq:ycond}) from Appendix \ref{appendix:monochromatic}. 
Recall that $A_{i-1}$ is a subset of $[y/v,2y/v)$ consisting of elements of the form $qu$, where $u|m$, $u\le y^{1/16}$,
$d_{i}|u$ and $\gcd(q,W(r)m)=1$. Note that each element of the form $qu$ where $u|m$ and $\gcd(q,W(r)m)=1$ is coprime to $W(r)/\gcd(W(r),m)$. Since each arithmetic progression $P_s$ has length at least $r^{1/16}$ and common difference at most $n$, Lemma \ref{lem:selberg} implies that the number of elements in $P_s$ of the form $qu$, where $u|m$, $u\le y^{1/16}$ and $\gcd(q,W(r)m)=1$, is at most 
\[
256|P_s|\cdot\frac{\log\log n}{\log r}.
\] 
Thus, the number of elements of $R$ (identified with a subset of $[y/v,y/v+t)$) of the form $qu$, where $u|m$, $u\le y^{1/16}$ and $\gcd(q,W(r)m)=1$, is at most 
\[
256 (\sum_{s\in \mathcal{S}}|P_s|)\cdot\frac{\log\log n}{\log r} \le 256\cdot \frac{2^{22}y}{z} \cdot\frac{\log\log n}{\log r}. 
\]
We claim that $$256\cdot\frac{2^{22}y}{z}\cdot\frac{\log\log n}{\log r}<\frac{z}{32l}.$$
This holds if $r \ge 10 \log m/ \log \log m$, since then, from (\ref{eq:tau-log}) in Appendix \ref{appendix:monochromatic}, $\tau(r,m) \ge 1/(4\log r)$ and 
\begin{align*}
\frac{\ell y\log\log n}{z^{2}\log r}&\le\frac{100\ell yr^{2}(\log\log n)}{(\log r)(m/\phi(m))^{2}y^{2}\tau(r,m)^2}
\le \frac{100\ell r^{2}(\log\log n)}{(\log r) \tau(r,m)^2 (m/\phi(m))^2}\cdot \sqrt{\frac{\tau(r,m)(m/\phi(m))}{15mr}}\\
&\le100\ell \sqrt{\frac{64r^3(\log r)(\log \log n)^2}{m(m/\phi(m))^3}}
\le 100\ell \sqrt{64c^3}
\le2^{-40},
\end{align*}
where we used (\ref{eq:z}) in the first inequality, the definition of $y$ in the second inequality, the bound $\tau(r,m)\ge 1/(4\log r)$ in the third inequality, the bound $r\le c\psi(n,m) = c\frac{m^{1/3}(m/\phi(m))}{(\log n)^{1/3}(\log \log n)^{2/3}}$ in the fourth inequality and, in the last inequality, we assumed a sufficiently small choice of $c$ (depending on $\ell$). Next, assume that $r<10\log m/\log \log m$. In this case, we have that $r = c{\cal R}(n,m) = c\rho(n,m)$. Furthermore, $m \ge \frac{n^2}{(\log n)^2}$ as, for $ \frac{n^{3/2}(\log\log n)^{1/2}}{(\log n)^{1/2}}\le m\le \frac{n^2}{(\log n)^2}$, $r = c\rho(n,m) =\Theta\left(\frac{n^2/\phi(m)}{\log(n^2/\phi(m))}\right)$ by Claim \ref{claim:orderrho}, so $r > 10 \log m/\log \log m$ for sufficiently large $n$, contradicting our assumption. We also have that $\tau(r,m) \ge (\phi(m)/m)/(2\log r)$ by (\ref{eq:tau-log2}) in Appendix \ref{appendix:monochromatic}, so 
\begin{align*}
\frac{\ell y\log\log n}{z^{2}\log r}&\le  \frac{100\ell r^{2}(\log\log n)}{(\log r) \tau(r,m)^2 (m/\phi(m))^2}\cdot \sqrt{\frac{\tau(r,m)(m/\phi(m))}{15mr}}\\
&\le 100\ell \sqrt{\frac{8r^3(\log r)(\log \log n)^2}{m}}
\le 100\ell \sqrt{\frac{10^5(\log n)^3}{n^2/(\log n)^2}}
\le 2^{-40},
\end{align*} assuming that $n$ is sufficiently large, where in the third inequality we used $r < 10\log m/\log \log m < 20\log n/\log \log n$. 
Thus, in both cases, $$256\cdot\frac{2^{22}y}{z}\cdot\frac{\log\log n}{\log r}<\frac{z}{32\ell}.$$
This is a contradiction since there are at least $\frac{z}{32\ell}$ integers of the form
$qu$ where $u|m,u\le y^{1/16}$ and $\gcd(q,W(r)m)=1$
contained in $R$. Hence, in each step $i\le|A|/2$ which is an unsaturated phase, $|\Sigma_t(i)|$ grows by at least $2^{12}y/z$. 
\end{proof}

Using Lemma \ref{lem:growth-}, we can now give the proof of Lemma
\ref{lem:build-prime-class}, thus completing our proof of the lower bound in Theorem \ref{thm:monochromatic-subset-sums}. Again we recall the statement, that if $\ell =\lceil32/\xi\rceil$ 
and $A$ is a subset of $Y_{v}$ of size $z/(8\ell)$ which is $y^{1/4}/(32\ell)$-diverse, then $|\Sigma(A)|\ge 3yz/(4\ell^{2}v)$
and $\Sigma(A)$ is not a subset of an arithmetic progression with common difference greater than $1$. 

\begin{proof}[Proof of Lemma \ref{lem:build-prime-class}]
Using Lemma \ref{lem:diverse}, we can partition $A$ into two sets $A_{1},A_{2}$ such that $|A_{1}|=|A_{2}|=z/(16\ell)$
and $A_{1}$ is $y^{1/4}/(128\ell)$-diverse.
By Lemma \ref{lem:growth-}, for each $t$ in $A_{2}$, $|\Sigma_{t}(A_{1})|\ge\min(\xi,32/\ell)t$.
Recall that we chose $\ell=\lceil32/\xi\rceil$, so
$|\Sigma_{t}(A_{1})|\ge32t/\ell$. Therefore, by repeated applications of Lemma \ref{lem:modp},
\[
|\Sigma(A)|=|\Sigma(A_1\cup A_2)|\ge \sum_{t \in A_2} \frac{32 t}{\ell} \ge |A_2| \cdot \frac{32y}{\ell v} > \frac{yz}{\ell^2 v}.  
\]
Finally, we prove that $\Sigma(A)$ is not contained in an arithmetic
progression with common difference $d$ larger than $1$. Indeed,
if $\Sigma(A)$ were contained in an arithmetic progression with common
difference $d$ larger than $1$, then all elements of $\Sigma(A)$
would be in the same congruence class modulo $d$, which in turn implies that all elements of
$A$ would be divisible by $d$. 
But this is impossible since $A$ is $y^{1/4}/(32\ell)$-diverse. 
\end{proof}

\subsection{Proof of the upper bound in Theorem \ref{thm:monochromatic-subset-sums}}

In this section, we show how to improve on the construction of Alon and Erd\H{o}s~\cite{AEr} described in the introduction. We begin with the following simple claim.

\begin{claim}
\label{claim:exists-d}There exists a positive constant $\kappa$
such that the following holds. For each positive integer $m$, let $d_m$ be the product of all the primes at most $(\log m)/64$ which are not prime divisors of $m$, where $d_m = 1$ if the product is over an empty set. Then, for $m$ sufficiently large, $d_{m}\le m^{1/32}$, $\gcd(d_{m},m)=1$ and $\frac{m}{\phi(m)}\cdot\frac{d_{m}}{\phi(d_{m})}\ge\kappa\log\log m$. 
\end{claim}

\begin{proof}
It is easy to see that $d_{m}\le m^{1/32}$. Furthermore, 
\begin{align*}
\frac{m}{\phi(m)}\cdot\frac{d_{m}}{\phi(d_{m})} & \ge\prod_{p\le(\log m)/64}\frac{p}{p-1}\ge\exp\left(\sum_{p\le(\log m)/64}\frac{1}{p}\right)\\
 & \ge\exp(\log\log((\log m)/64)-\kappa')\ge\kappa\log\log m,
\end{align*}
for some absolute constants $\kappa',\kappa$. 
\end{proof}

We are now ready to prove the upper bound in Theorem~\ref{thm:monochromatic-subset-sums}. For the sake of easy reference, we recall the statement, that, for all $n$ sufficiently large and $m\in [n,\binom{n}{2}]$,
\[
f(n,m)=O\left({\cal R}(n,m)\right),
\] where ${\cal R}(n,m)=\min\left(\psi(n,m),\rho(n,m)\right)$. Here $\psi(n,m)=\frac{m^{1/3}(m/\phi(m))}{(\log n)^{1/3}(\log\log n)^{2/3}}$ and $\rho(n,m)$ is the smallest positive integer $\rho$ such that $\rho\prod_{p | W(\rho)m}(1-1/p)^{-1} \ge n^2/\phi(m)$, where $W(\rho)$ is the product of the first $\rho$ primes. We also recall from the previous subsection that when $m = O\left( \frac{n^{3/2}(\log\log n)^{1/2}}{(\log n)^{1/2}}\right)$, we have ${\cal R}(n,m) = \Theta\left(\psi(n,m)\right)$ and when $m=\Omega\left(\frac{n^{3/2}(\log\log n)^{1/2}}{(\log n)^{1/2}}\right)$, we have ${\cal R}(m,n) = \Theta(\rho(n,m))$.

\begin{proof}[Proof of the upper bound in Theorem~\ref{thm:monochromatic-subset-sums}]
We first consider the case where $m\le\frac{n^{3/2}(\log\log n)^{1/2}}{(\log n)^{1/2}}$. 
Let $r=C\psi(n,m)=\frac{Cm^{1/3}(m/\phi(m))}{(\log n)^{1/3}(\log\log n)^{2/3}}$
for a sufficiently large constant $C$. Note that $r\ge\frac{n^{1/3}}{(\log n)^{2/3}}$.
Our aim is to construct an $r$-coloring of $[1,n-1]$ such that
the set of subset sums of each color class does not contain $m$. We will do this in four steps. 

\vspace{2mm}
\noindent
\textbf{Step 1.} For $k=2,3,\dots,\frac{r}{2}$, we form a color
class $S_{1}(k)=\{\lceil\frac{m}{k+1}\rceil,\lceil\frac{m}{k+1}\rceil+1,\dots,\lfloor\frac{m}{k}\rfloor\}\cap[n-1]$, while,
for $k=1$, we take $S_{1}(1)=\{\lceil\frac{m}{2}\rceil,\lceil\frac{m}{2}\rceil+1,\dots,m-1\}\cap[n-1]$. As defined, the color
classes may overlap, but we can safely assign any element in the overlap to any color class that includes it.
Crucially, no subset sum of $S_{1}(k)$ can contain $m$, since the sum
of at most $k$ elements from $S_{1}(k)$ is less than $m$, while any
sum of $k+1$ elements from $S_{1}(k)$ is larger than $m$. Let $X_{1}=\bigcup_{k=1}^{r/2}S_{1}(k)$. 

\vspace{2mm}
\noindent
\textbf{Step 2.} For each of the first $r/4$ primes $p$ which are coprime to $m$, we form a color class
$S_{2}(p)=\{kp:k\in\mathbb{N}\}\cap[1,n-1]$, noting that no subset sum of
$S_{2}(p)$ can contain $m$, since each element of $\Sigma(S_{2}(p))$
is a multiple of $p$. Let $X_{2}=\bigcup_{p\le p_{r/4},p\nmid m}S_{2}(p)$. 

\vspace{2mm}
\noindent
\textbf{Step 3.} Let $R=[1,n-1]\setminus (X_{1}\cup X_2)$. The construction in \cite{AEr} also uses color classes like those defined in Steps 1 and 2. They then arbitrarily partition the remaining elements $R$ so that the sum of the elements in each of the classes is smaller than $m$. For our improvement, we need to be more careful. Note that elements in the remainder set $R$ are natural numbers
$t\le\frac{2m}{r}$ such that $t$ is coprime to $W(r/4)/\gcd(W(r/4),m)$, that is, all prime divisors of $t$ which are at most $p_{r/4}$ are also prime divisors of $m$. In particular, $t$ is coprime to the integer $d_{m}$ given by Claim~\ref{claim:exists-d}, since $d_m$ has only prime factors at most $(\log m)/64<r/4$ which are not prime divisors of $m$. 

With $\kappa$ also as in Claim~\ref{claim:exists-d}, we next show that there exists a multiple $d$ of $d_{m}$ such that 
\begin{equation}
\kappa(\phi(m)/m)r(\log\log n)/64\le d\le\kappa(\phi(m)/m)r(\log\log n)/32, \label{eq:boundd}
\end{equation}
$\gcd(d,m)=1$ and the largest prime factor of $d$ is at most $p_{r/4}$. Let $x = \frac{\kappa(\phi(m)/m)r\log\log n}{50d_m}$. Note that $x\ge \frac{\kappa r}{100 d_m \log \log n} > n^{1/4}$, since $r\ge n^{1/3}/(\log n)^{2/3}$ and $d_m \le m^{1/32}\le n^{1/16}$. Since $m$ has at most $\log m \le 2\log n$ distinct prime factors, 
there exists a prime $p < n^{1/100}$ such that $p$ does not divide $m$. Let $k$ be the smallest positive integer such that $x/p^k < n^{1/100}$. Then $x/p^{k-1}\in [n^{1/100},n^{2/100}]$ and, by the prime number theorem, the interval $[(1-1/100)x/p^{k-1},(1+1/100)x/p^{k-1}]$ has at least $10^{-3}n^{1/100}/(\log n)$ primes for sufficiently large $n$. Thus, there exists a prime $p'$ in this interval which does not divide $m$. Then $p'p^{k-1} \in [(1-1/100)x,(1+1/100)x]$ and $\gcd(p'p^{k-1},m)=1$, since $p'$ and $p$ are primes which do not divide $m$. We can now verify that $d=p'p^{k-1}d_m$ satisfies (\ref{eq:boundd}), $\gcd(d,m)=1$ and the largest prime factor of $d$ is at most $2n^{2/100} < p_{r/4}$ (noting that all prime factors of $d_m$ are at most $(\log m)/64 < 2n^{2/100}$). Since $d_m|d$, we also have $d/\phi(d)\ge d_{m}/\phi(d_{m})$. Furthermore, all elements of $R$ are coprime to $d$, since any element in $R$ is coprime to $W(r/4)/\gcd(W(r/4),m)$, whereas $d$ is coprime to $m$ and all prime factors of $d$ are at most $p_{r/4}$. 

Fix $s \in \mathbb{Z}_d^\times$. Then there exist $k$ integers congruent to $s\pmod d$ that sum to $m$ only
if $sk\equiv m\pmod d$. Let $x_s$ be the positive integer in $[d]$ that is congruent to $s^{-1}m\pmod d$. Consider now the color classes
\[
S_{s,1}=\left\{ t:\,t\in R,t\equiv s \, (\bmod \, d), t\ge\frac{m}{x_s}\right\} ,
\]
\[
S_{s,2}=\left\{ t:\,t\in R,t\equiv s \, (\bmod \, d),\frac{m}{d+x_s}\le t<\frac{m}{x_s}\right\}.
\]
If a sum of $k$ elements in $S_{s,1}$ is equal to $m \pmod d$,
then $k\ge x_s$. But then the sum of the $k$ elements is larger than $m$. Similarly, if a sum of $k$ elements in $S_{s,2}$
is equal to $m \pmod d$, then $k=x_s$ or $k\ge d+x_s$.
But the sum of $k$ elements is less than
$m$ if $k=x_s$ and larger than $m$ if $k\ge d+x_s$. Thus,
$m\notin\Sigma(S_{s,1})\cup\Sigma(S_{s,2})$. Note that in this step we have in total defined $2\phi(d)$ color classes of the form $S_{s,1}$ and $S_{s,2}$ for $s\in \mathbb{Z}_d^\times$. 

\vspace{2mm}
\noindent
\textbf{Step 4.} Let $R'=R\setminus(\bigcup_{s \in \mathbb{Z}_d^\times}(S_{s,1}\cup S_{s,2})).$ Then all elements of $R'$ are less than $m/d$. Thus, if we arbitrarily
partition $R'$ into sets of size at most $d$, then no set contains
a subset sum which is equal to $m$. Hence, we need at most $|R'|/d$
colors to color $R'$ so that no color class contains $m$ as a subset
sum. Recall that any element in $R'$ is coprime to $W(r/4)/\gcd(W(r/4),m)$. By the second part of Lemma~\ref{lem:selberg}, applied to the interval $[1, m/d]$, we have 
\begin{align*}
|R'|&\le 256(m/d) \prod_{p|W(r/4), p\nmid m}(1-1/p) \\
&\le 256(m/d) \prod_{p|W(r/4)}(1-1/p) \prod_{p|m, p \leq r/4}(1-1/p)^{-1}\\
&<\frac{500m(m/\phi(m))}{d\log r},
\end{align*} where we used that $r$ is sufficiently large, so that $\prod_{p|W(r/4)}(1-1/p) \le 1.1/\log r$, and $\prod_{p|m,p\le r/4}(1-1/p)^{-1} \le m/\phi(m)$. 
Therefore, the number of color classes used in Step 4 is at most 
\begin{align*}
\frac{500m(m/\phi(m))}{d^{2}\log r} &\le\frac{64^{2}\cdot500m(m/\phi(m))^{3}}{\kappa^{2}r^{2}(\log r)(\log\log n)^{2}} \\&<\frac{r}{16},
\end{align*} where the first inequality follows from (\ref{eq:boundd}) and in the second inequality we have assumed that the constant $C$ is chosen sufficiently large.

Combining all four steps, the total number of colors we have used is at most 
\begin{align*}
\frac{r}{2}+\frac{r}{4}+2\phi(d)+\frac{r}{16} \le \frac{13r}{16} + \frac{2md}{\kappa \phi(m)\log \log n} \le \frac{7r}{8},
\end{align*}
where we have used Claim~\ref{claim:exists-d}, so that $d/\phi(d) \ge d_m/\phi(d_m) \ge \kappa (\log \log m) \phi(m)/m \ge \kappa (\log \log n)\phi(m)/m$, and the bound (\ref{eq:boundd}). Thus, we
can use at most $r$ colors to color $[1,n-1]$ such that no monochromatic
subset sum is equal to $m$, as required.

\vspace{2mm}
Next we consider the case $m\in\left[\frac{n^{3/2}(\log\log n)^{1/2}}{(\log n)^{1/2}},\binom{n}{2}\right]$.
Let $r=C\rho(n,m)$, where $C$ is a sufficiently large absolute constant. We
construct the coloring as follows. 

\vspace{2mm}
\noindent
\textbf{Step 1.} For each of the first $7r/8$ primes $p$ that do not divide $m$, we construct a color class $S_1(p)=\{kp:k\in \mathbb{N}\}$. Let $X_1=\bigcup_{p\le p_{7r/8},p\nmid m}S_1(p)$.

\vspace{2mm}
\noindent
\textbf{Step 2.} Let $R=[1,n-1]\setminus X_1$. The set $R$ consists of those integers less than $n$ which are coprime to $W(7r/8)/\gcd(W(7r/8),m)$. By Lemma~\ref{lem:selberg}, the number of elements of $R$ is at most 
\begin{align*}
256 n\prod_{p|W(7r/8),p\nmid m}(1-1/p) &\le 256 n\prod_{p|W(\rho(n,m))m}(1-1/p)\cdot \prod_{p|\gcd(m,W(\rho(n,m)))}(1-1/p)^{-1} \\
&\le 256 n(m/\phi(m)) \prod_{p|W(\rho(n,m))m}(1-1/p),
\end{align*} where in the first inequality we used that $7r/8=7C\rho(n,m)/8>\rho(n,m)$, which holds by choosing the constant $C$ to be sufficiently large, and in the second inequality we used $\prod_{p|\gcd(m,W(\rho(n,m)))}(1-1/p)^{-1} \le \prod_{p|m}(1-1/p)^{-1} = m/\phi(m)$. 

Since each element of $R$ is less than $n$, if a color class contains at most
$m/n$ elements, then no sum of elements from the color class can equal
$m$. Thus, we can use at most 
$$\frac{|R|}{m/n}\le 256 n^2/\phi(m)\cdot\prod_{p|W(\rho(n,m))m}(1-1/p)\le 256 \rho(n,m)$$
colors to color the elements of $R$ so that no monochromatic subset
sum equals $m$. The second inequality follows from the definition of $\rho(n,m)$, which is the smallest positive integer such that $\rho(n,m)\prod_{p|W(\rho(n,m))m}(1-1/p)^{-1} \ge n^2/\phi(m)$. Hence, the total number of colors we used is at most 
\[
\frac{7r}{8}+256 \rho(n,m)\le r,
\]
assuming that $C$ is a sufficiently large absolute constant. 
\end{proof}

\section{\label{sec:Sze-Vu-proof}Long homogeneous progressions in subset sums}

In this section, we prove Theorem~\ref{thm:Sze-Vu-2},
strengthening Theorem \ref{SVthm}, Szemer\'edi and Vu's result \cite{SV} on arithmetic progressions in subset sums, by showing that the progression may be taken to be homogeneous. For our application to the Erd\H{o}s--Graham problem in Section~\ref{sec:extremal-sum}, we will need a somewhat technical strengthening of this result, for which it will be useful to have the notation
\[
\Sigma^{[h]}(A)=\left\{ \sum_{s\in S}s:S\subseteq A,|S|\le h\right\}.
\]
The main result of this section, which includes Theorem~\ref{thm:Sze-Vu-2} as a special case, is now as follows. To gain some intuition, we remark that for a typical set $A$ which is not dominated by multiples of an integer at least $2$, we will simply have $d = 1$.

\begin{thm}
\label{thm:Sze-Vu-2-strong}There exists an absolute constant $C>0$
such that the following holds. For any subset $A$ of $[n]$ of size $m\ge C\sqrt{n}$,
there exists $d\ge1$ such that, for $A'=\{x/d\,:\,x\in A,d|x\}$ and $k = 2^{50}n/m$, $\Sigma^{[k]}(A')$ contains an interval of length at least $n$.
Furthermore, 
\[
|A|-|A'|\le 2^{30}(\log n)^{3}+\frac{2^{30}n}{m}.
\]
\end{thm}

Theorem \ref{thm:Sze-Vu-2} immediately follows from Theorem \ref{thm:Sze-Vu-2-strong} by noticing that $\Sigma(A)$ contains the set $\{dy\,:\,y\in \Sigma^{[k]}(A')\}$, which is a homogeneous arithmetic progression with common difference $d$. 

As a crucial step in the proof of Theorem~\ref{thm:Sze-Vu-2-strong}, we first show that subsets of $\mathbb{Z}_b$ satisfying a diversity condition have a large set of subset sums. 
We will need to use the mod $b$ analogue of $\Sigma^{[h]}(A)$, namely,
\[
\Sigma_{b}^{[h]}(A)=\left\{ \sum_{s\in S}s\bmod b:S\subseteq A,|S|\le h\right\}.
\]

\begin{lem} \label{lem:structure-1} 
Let $b$ be a positive integer. Let $A$ be a subset of $\mathbb{Z}_b$ of size  $m \in (80(\log b)^{2},b]$ such that, for each $d|b$ with $d\in [2,4b/m]$, there are at least $64(\log b)^2+8d$ elements in $A$ which are not divisible by $d$. Let $k=1280b/m$.
Then $$|\Sigma_{b}^{[k+1]}(A)|\ge\min(m^{2}/256,b/4).$$ 
\end{lem}

\begin{proof}
Let $A'$ be a uniformly random subset of $A$ of size $3m/4$ and $A'' = A\setminus A'$. Let ${\cal E}_1$ be the event that for some $d|b$ with $d\in [2,4b/m]$, there are at most $8(\log b)^2+d$ elements in $A'$ which are not divisible by $d$. Recall that for each $d|b$ and $d\in [2,4b/m]$, at least $64(\log b)^2+8d$ elements of $A$ are not divisible by $d$. By the Chernoff bound for hypergeometric distributions, the probability that there are at most $8(\log b)^2+d$ elements in $A'$ which are not divisible by $d$ is at most $\exp(-64(\log b)^2/32) \le 1/b^2$. By taking a union bound over all $d|b$, the probability that ${\cal E}_1$ happens is then at most $1/b<1$. We may therefore fix a choice of $A'$ and $A''$ so that ${\cal E}_1$ does not hold. 

We consider the following iterative process. Let $\Sigma_b(0)=A''$ and let $A_0 = A'$. At each step
$i\ge 1$, we pick an element $a_{i}$ in $A_{i-1}$ and let $A_{i}=A_{i-1}\setminus\{a_{i}\}$ and 
$\Sigma_b(i)=\Sigma_b(i-1)\cup(\Sigma_b(i-1)+a_{i})$. Observe that the elements in $\Sigma_b(i)$ can be written as the sum of one element in $A''$ and a subset of $A'$ of size at most $i$, so $\Sigma_b(i)\subseteq \Sigma_b^{[i+1]}(A)$. Let $d_{i}|b$ be such that $\langle A_{i-1}\rangle=d_{i}\mathbb{Z}_{b}\cong\mathbb{Z}_{b/d_{i}}$, where $\langle X\rangle$ denotes the subgroup generated by $X$. Note that, by definition, $d_{i}|d_{j}$ if $i\le j$. Furthermore, $|A_{i-1}|\le b/d_{i}$ and $|A_{i-1}| = 3m/4-i+1 \ge m/4$ for $i\le m/2$, so $d_{i}\le 4b/m$ for $i\le m/2$. We will run the above process for at most $m/2$ steps, so we only consider $i\le m/2$ throughout. 

For each $i$, we say that step $i$ is either a {\it growth phase}, an {\it unsaturated
phase} or a {\it saturated phase}. For each $u\in\mathbb{Z}_{b}/d_{i}\mathbb{Z}_{b}\cong\mathbb{Z}_{d_{i}}$, let $S_{u}=\Sigma_b(i-1)\cap(u+d_{i}\mathbb{Z}_{b})$. 
We then say that step $i$ is a {\it growth phase} if
there exists $u\in\mathbb{Z}_{b}/d_{i}\mathbb{Z}_{b}\cong\mathbb{Z}_{d_{i}}$
such that $S_u$ is
non-empty and has size at most $|A_{i-1}|/4$. We say that step $i$ is
an {\it unsaturated phase} if it is not a growth phase and there
exists $u\in\mathbb{Z}_{b}/d_{i}\mathbb{Z}_{b}$ such that $\frac{|A_{i-1}|}{4}<|S_{u}|<\frac{b}{4d_{i}}$.
Finally, if step $i$ is neither a growth phase nor an unsaturated phase, then
it is a {\it saturated phase}. We remark that if $d_i=d$ for all steps $i$ in an interval $[x,y]$, then the interval can be partitioned into three (possibly empty) intervals such that the steps in the first interval are all growth phases, the steps in the second interval are all unsaturated phases and the steps in the third interval are all saturated phases. 

We next discuss how to pick $a_{i}$, which depends on the type of phase. For $d|b$, let $\Sigma(d,i-1)=\Sigma_{b}(\{a_{1},\dots,a_{i-1}\}\cap d\mathbb{Z}_{b})$.
If step $i$ is a growth phase, we pick $a_{i}$ which maximizes $|\Sigma(d_i,i)|-|\Sigma(d_{i},i-1)|$. If step $i$ is an unsaturated or saturated phase, we pick $a_{i}$ which maximizes $|\Sigma_b(i)|-|\Sigma_b(i-1)|$. 

The following claims record important properties of the process we have defined. 

\vspace{2mm}

\noindent {\bf Claim 1.} The first $(\log b)^2$ steps are not growth phases. 

\vspace{1mm}

\noindent {\it Proof.} Consider $i\le (\log b)^2$. Note that $i\le m/2$. Since, for each $d|b$ with $d\in [2,4b/m]$, there are more than $8(\log b)^2$ elements in $A'$ which are not divisible by $d$, there must be at least $7(\log b)^2$ elements in $A_{i-1}$ which are not divisible by $d$. Hence, $d_i = 1$. Thus, there is only one coset $u$ of $d_i\mathbb{Z}_b$ in $\mathbb{Z}_b$ and $|S_u| \ge |\Sigma_b(0)| = m/4 \ge |A_{i-1}|/4$, so $i$ is not a growth phase. \qed

\vspace{2mm}

\noindent {\bf Claim 2.} There are at most $20(\log b)(\log (4b/m))$ growth phases among the first $m/2$ steps. 

\vspace{1mm}

\noindent {\it Proof.} Suppose that step $i$ is a growth phase. By Lemma \ref{lem:Su}, for $u \in \mathbb{Z}_b/d_i\mathbb{Z}_b$, if $S_u$ is non-empty, then $|S_u| \ge |\Sigma(d_i,i-1)|$. Since we are in a growth phase, there is some $u$ such that $S_u$ is non-empty and $|S_u|\le |A_{i-1}|/4$. This implies that $|\Sigma(d_i,i-1)|\le |A_{i-1}|/4$. By Lemma \ref{lem:double-count},
the set of $a\in d_{i}\mathbb{Z}_{b}$ such that $|(\Sigma(d_{i},i-1)+a)\setminus\Sigma(d_{i},i-1)|\le\frac{1}{2}|\Sigma(d_{i},i-1)|$
has size at most $2|\Sigma(d_{i},i-1)|\le|A_{i-1}|/2$. Thus, there
exists $a_{i}\in A_{i-1}$ such that $|\Sigma(d_{i},i)|\ge\frac{3}{2}|\Sigma(d_{i},i-1)|$.
As $|\Sigma(d_i,i)|\le b$ for all $i$, there can be at most $1+\log_{3/2}b$ successive growth
phases with $d_{i}=d$. Since $d_1\le d_2\le \dots \le d_{m/2} \le 4b/m$ and $d_{i+1}\ge 2d_i$ if $d_{i+1}\ne d_i$, $d_i$ can take at most $1+\log_2 (4b/m)$ distinct values. This shows that there can be at most $(1+\log_{3/2}b)(1+\log_2(4b/m)) < 20(\log b)(\log (4b/m))$ 
growth phases among the first $m/2$ steps. \qed 

\vspace{2mm}

\noindent {\bf Claim 3.} Let $i$ be an unsaturated phase. Then $|\Sigma_b(i)|-|\Sigma_b(i-1)|\ge\frac{|A_{i-1}|}{16}$. 

\vspace{1mm}

\noindent {\it Proof.} Note that if $a\in A_{i-1}$, then $a\in d_i\mathbb{Z}_{b}$, since $A_{i-1}\subseteq d_i\mathbb{Z}_{b}$. Thus, for each $u\in \mathbb{Z}_{b}/d_i\mathbb{Z}_{b}$ and $x\in \mathbb{Z}_{b}$, $x+a\in u+d_i\mathbb{Z}_{b}$ if and only if $x\in u+d_i\mathbb{Z}_{b}$. Hence, $(\Sigma_b(i-1)+a)\cap(u+d_{i}\mathbb{Z}_{b})=(\Sigma_b(i-1)\cap (u+d_{i}\mathbb{Z}_{b}))+a = S_u+a$ and $((\Sigma_b(i-1)+a)\setminus \Sigma_b(i-1))\cap (u+d_{i}\mathbb{Z}_{b})) = (S_u+a)\setminus S_u$. Let $u_{0}$ be such that $|S_{u_{0}}|=\max_{u:|S_{u}|<b/4d_{i}}|S_{u}|$. We have
\begin{align*}
|(\Sigma_b(i-1)+a)\setminus\Sigma_b(i-1)|&=\sum_{u \in \mathbb{Z}_{b}/d_i\mathbb{Z}_{b}} |((\Sigma_b(i-1)+a)\setminus\Sigma_b(i-1))\cap (u+d_i\mathbb{Z}_{b})| \\
&= \sum_{u \in \mathbb{Z}_{b}/d_i\mathbb{Z}_{b}}|(S_{u}+a)\setminus S_{u}|\\
&\ge|(S_{u_{0}}+a)\setminus S_{u_{0}}|.
\end{align*}
Let $k_{i}=1+\left\lfloor \frac{4|S_{u_{0}}|}{|A_{i-1}|}\right\rfloor \in\left[2,\frac{8|S_{u_{0}}|}{|A_{i-1}|}\right]$, noting that $|S_{u_0}| > |A_{i-1}|/4$ since step $i$ is not a growth phase. 
Let $P_{i-1}\subseteq d_{i}\mathbb{Z}_{b}$ be the set of elements
$a\in d_{i}\mathbb{Z}_{b}$ such that $|(S_{u_{0}}+a)\setminus S_{u_{0}}|<\frac{1}{2k_{i}}|S_{u_{0}}|$.
Note that $0\in P_{i-1}$ and, by Lemma \ref{lem:stable-period}, for any $a\in k_{i}P_{i-1}$, 
$|(S_{u_{0}}+a)\setminus S_{u_{0}}|<\frac{1}{2}|S_{u_{0}}|$.
Therefore, by Lemma \ref{lem:double-count}, $|k_{i}P_{i-1}|\le2|S_{u_{0}}|$.
Suppose now that $A_{i-1}\subseteq P_{i-1}$. Since $d_i$ is defined so that $\langle A_{i-1}\rangle = d_i\mathbb{Z}_{b} \cong \mathbb{Z}_{b/d_i}$, $A_{i-1}$ is not a subset of a proper subgroup of $\mathbb{Z}_{b/d_i}$ and, since $0\in P_{i-1}$, $P_{i-1}$ is not contained in a
coset of a proper subgroup of $\mathbb{Z}_{b/d_{i}}$. Thus, by Lemma
\ref{lem:Cauchy-Davenport} and the fact that $|k_{i}P_{i-1}|\le2|S_{u_{0}}| < \frac{b}{2d_i} < |\mathbb{Z}_{b/d_i}|$, $|P_{i-1}|\le \frac{2|k_i P_{i-1}|}{k_i} \le \frac{4}{k_{i}}|S_{u_{0}}|$. However, by our choice of $k_{i}$, 
$|A_{i-1}|>\frac{4}{k_{i}}|S_{u_{0}}|$, which is a contradiction. Therefore, $A_{i-1}\cap P_{i-1}^{c}\ne\emptyset$ and, since $a_i$ is chosen so that $|\Sigma_b(i)|-|\Sigma_b(i-1)|=|(\Sigma_b(i-1)+a_i) \setminus \Sigma_b(i-1)|$ is maximized, 
\begin{align*}
|(\Sigma_b(i-1)+a_{i})\cup\Sigma_b(i-1)| & \ge|\Sigma_b(i-1)|+\frac{1}{2k_{i}}|S_{u_{0}}|\\
 & \ge|\Sigma_b(i-1)|+\frac{|A_{i-1}|}{16|S_{u_{0}}|}|S_{u_{0}}|\\
 & = |\Sigma_b(i-1)|+\frac{|A_{i-1}|}{16}.
\end{align*} 
Thus, over any unsaturated phase $i$, $|\Sigma_b(i)|-|\Sigma_b(i-1)|\ge\frac{|A_{i-1}|}{16}$. \qed

\vspace{2mm}

\noindent {\bf Claim 4.} For each step $i\le m/2$, $S_u = \Sigma_b(i-1)\cap(u+d_{i}\mathbb{Z}_{b})$ is non-empty for every $u\in \mathbb{Z}_b/d_i\mathbb{Z}_b$. 

\vspace{1mm}

\noindent {\it Proof.} The claim holds trivially if $d_i=1$. Assume that $d_i>1$. Since $i\le m/2$, we have $d_i|b$ and $d_i \le 4b/m$. For each $d|d_i$ with $d>1$, we have $d|b$ and $d\in [2,4b/m]$, so $A'$ contains at least $d-1$ elements which are not divisible by $d$. By Lemma \ref{lem:full-mod-d}, 
$$\Sigma_{d_{i}}(\{a_{1},\dots,a_{i-1}\}) = \Sigma_{d_{i}}(A') = \mathbb{Z}_{d_{i}},$$ 
where we used that the elements of $A'\setminus \{a_{1},\dots,a_{i-1}\}$ are all divisible by $d_{i}$. The claim follows upon noting that we can identify $\mathbb{Z}_b/d_i\mathbb{Z}_b$ with $\mathbb{Z}_{d_i}$ and, under this identification, $S_u$ is non-empty if and only if $u\in \Sigma_{d_{i}}(\{a_{1},\dots,a_{i-1}\})$. \qed

\vspace{2mm}

We now complete the proof of the lemma using these claims. First, assume that there is no saturated phase $i$ with $i\le \min(m/2,k)$. Then, among the first $\min(m/2,k)$ steps, from Claims 1 and 2, at least $\max(\min(m/2,k)-20(\log b)(\log(4b/m)), \min((\log b)^2,k))$ steps are unsaturated phases. Note that if $k=1280b/m \le 80(\log b)^2$, then $(\log b)^2 \ge k/80$ and 
$$\max(\min(m/2,k)-20(\log b)(\log(4b/m)), \min((\log b)^2,k)) \ge \min((\log b)^2,k) \ge k/80.$$ 
If $k=1280b/m \ge 80(\log b)^2$, then $20(\log b)(\log (4b/m)) \le k/2$ and $20(\log b)(\log (4b/m)) \le m/4$ by our assumption on $m$, so 
\begin{align*}
\max(\min(m/2,k)-20(\log b)(\log(4b/m)), \min((\log b)^2,k)) &\ge \min(m/2,k)-20(\log b)(\log(4b/m))\\ &\ge \min(m/4,k/2).
\end{align*} 
In either case, we have $$\max(\min(m/2,k)-20(\log b)(\log(4b/m)), \min((\log b)^2,k)) \ge \min(m/4,k/80).$$ For each step $i$ which is an unsaturated phase, we have, by Claim 3, that $|\Sigma_b(i)|- |\Sigma_b(i-1)| \geq |A_{i-1}|/16 \ge m/64$. Hence, recalling that $k=1280b/m$, we get 
\[
|\Sigma_{b}^{[k+1]}(A)|\ge|\Sigma_b(\min(m/2,k))| \ge \min\left(\frac{m}{4},\frac{k}{80}\right)\cdot\frac{m}{64}= \min\left(\frac{m^{2}}{256},\frac{b}{4}\right).
\]

If, instead, there is a saturated phase $i_0$ with $i_0 \le \min(m/2,k)$, then, for each $u \in \mathbb{Z}_b/d_{i_0}\mathbb{Z}_b$ with $S_u = \Sigma_b(i_0-1)\cap(u+d_{i_0}\mathbb{Z}_{b})$ non-empty, $|S_u| \ge \frac{b}{4d_{i_0}}$. But Claim 4 implies that $S_u$ is non-empty for all $u \in \mathbb{Z}_b/d_{i_0}\mathbb{Z}_b$, so that
\[
|\Sigma_{b}^{[k+1]}(A)| \ge |\Sigma_{b}(i_0)|\ge\sum_{u\in \mathbb{Z}_b/d_{i_0}\mathbb{Z}_b} |S_u| \ge \frac{b}{4d_{i_0}}\cdot d_{i_0} = \frac{b}{4}.
\]
Hence, the desired conclusion holds in both cases. 
\end{proof}

Let $\ell=2^{15}$. We say that a subset $A$ of $[n]$ of size $m$
is \emph{nice} if the following conditions hold:
\begin{enumerate}
\item \label{enu:cond1:divsmall}There is no $d \in [2,8\ell n/m]$ such that
all but at most $512\ell (\log n)^2 + 64\ell d$ elements of $A$ are divisible by $d$.
\item \label{enu:cond2-interval}For each dyadic interval $I_{j} = [2^{j-1},2^j)\cap [n]$, either $|A\cap I_{j}|=0$ or $|A\cap I_{j}|\ge 64\ell (\log n)$. 
\end{enumerate}
The next lemma says that any large subset of $[n]$ contains a multiple of a large nice set.  
\begin{lem}
\label{lem:nice-set}There exists a constant $C>0$ such that the following holds. Let $A$ be a subset of $[n]$ of size $m\ge Cn^{1/2}$.
Then there is an integer $d$ and a set $A'$ of integers such that 
\begin{itemize}
\item $A'$ is nice, 
\item $\{dx : x\in A'\}\subseteq A$ and 
\item $|A|-|A'|\le 1000\ell (\log n)^{3}+\frac{256\ell n}{m}$.
\end{itemize}
\end{lem}

\begin{proof}
We consider the following iteration. Let $A_{0}=A$ and $n_0 = n$. Note that $A_0 \subseteq [n_0]$. For each $i\ge0$, if $|A_{i}| < m/2$, we stop. If $A_{i}$ is nice, we let $A'=A_{i}$ and stop. Otherwise, $A_{i}\subseteq [n_i]$ is not nice and $|A_i| \ge m/2$. If (\ref{enu:cond1:divsmall})
does not hold, we let $A_{i+1}=\{x/d_i:x\in A_{i},d_i|x\}$, where
$d_i\in [2,8\ell n/|A_{i}|]$ is such that all but at most $512\ell (\log n)^2 + 64\ell d_i$
elements of $A_{i}$ are divisible by $d_i$. Note that $d_i \le 8\ell n/|A_i| \le 16\ell n/m$. Let $n_{i+1} = n_i / d_i$. Then $$|A_{i+1}|\ge|A_{i}|-512\ell (\log n)^2 - 64\ell d_i$$ and $A_{i+1}\subseteq[n_{i+1}]$. If (\ref{enu:cond1:divsmall}) holds and (\ref{enu:cond2-interval}) does
not hold, we remove all elements in $A_{i}$ which are contained
in dyadic intervals $I_{j}$ with $|A\cap I_{j}|<64\ell (\log n)$
and let $A_{i+1}$ be the resulting set. Let $n_{i+1}=n_i$, so $A_{i+1} \subseteq [n_{i+1}]$, and let $d_i = 1$. In this case, $|A_{i+1}|\ge |A_{i}| - 64\ell (\log n) (1+\log_2 n)$. 

We show that we will always stop and output a nice set with the required properties. Let $s$ be the step where we stop. Note that there can be at most $\log_2 n$ steps where (\ref{enu:cond1:divsmall}) does not hold. Furthermore, the number of steps where (\ref{enu:cond2-interval}) does not hold is at most one more than the number of steps where (\ref{enu:cond1:divsmall}) does not hold. Thus, we have 
\begin{equation}|A_s| \ge |A_0| - (1+\log_2 n)(512\ell (\log n)^2 + 64\ell (\log n)(1+\log_2 n)) - 64\ell \sum_{i\le s-1:d_i > 1} d_i.\label{eq:boundAs}
\end{equation} 
Furthermore, $n_s = n/(\prod_{i\le s-1}d_i)$, so $A_s \subseteq [n/(\prod_{i\le s-1}d_i)]$. We also have that
\begin{align*}
|A_s| &\ge |A_{s-1}| - \max(64\ell (\log n)(1+\log_2 n), 512\ell (\log n)^2 + 64\ell d_{s-1}) \\
&\ge |A_{s-1}| - 512\ell (\log n)^2 - 1024 \ell^2 n/m \\
&\ge m/4,
\end{align*} where we used that $d_{s-1} \le 16\ell n/m$, $|A_{s-1}|\ge m/2$, $m \ge C\sqrt{n}$ for a sufficiently large constant $C$ and $n$ is sufficiently large. Hence, 
$$\prod_{i\le s-1} d_i \le n/|A_s|\le 4n/m.$$ 
Since $d_i \ge 2$ for each $i \le s-1$ with $d_i > 1$, we have $\sum_{i\le s-1:d_i > 1}d_i \le \prod_{i\le s-1}d_i \le 4n/m$. Hence, combining with (\ref{eq:boundAs}), $$|A_s| \ge |A_0| - (1+\log_2 n)(512\ell (\log n)^2 + 64\ell (\log n)(1+\log_2 n)) -64\ell \cdot \frac{4n}{m} \ge m - 1000\ell (\log n)^3 - \frac{256\ell n}{m} \ge \frac{m}{2},$$ assuming that $m\ge C\sqrt{n}$ for $C$ sufficiently large and $n$ is sufficiently large. This implies that the iteration stops at step $s$ because $A_s$ is nice. The set $A'=A_s$ then satisfies all of the required properties. 
\end{proof}

We are now in a position to prove the main result of this section, Theorem~\ref{thm:Sze-Vu-2-strong}.

\begin{proof}[Proof of Theorem \ref{thm:Sze-Vu-2-strong}]
By Lemma \ref{lem:nice-set}, we can find a nice set $A'$ and an
integer $d$ such that $\{dx : x\in A'\}\subseteq A$ and 
\begin{align*}
|A|-|A'|&\le 1000 \ell (\log n)^{3}+\frac{256\ell n}{m}.
\end{align*}
In particular, $|A'|\ge 5|A|/8$ for $n$ sufficiently large. Partition $A'$ into $\ell$ sets $A_{1}',\dots,A_{\ell}'$
as follows. Let $A_{<}'$ be the set consisting of the $7|A'|/8$
smallest elements in $A'$ and let $A'_{>}$ be the remaining elements.
Partition $A'_{<}$ into $\ell$ sets $A_{<,1}',\dots,A_{<,\ell}'$, each of size $|A'_{<}|/\ell$, and partition $A'_{>}$ into $\ell$ sets $A_{>,1}',\dots,A_{>,\ell}'$, each of size $|A'_{>}|/\ell$, uniformly at random. Let $A_{i}'=A_{<,i}'\cup A_{>,i}'$.

Let $b_{i,1},b_{i,2},\dots,b_{i,|A'_{>,i}|}$ be a uniformly random enumeration of $A'_{>,i}$ and we then define two sets $B_{i,1} = \{b_{i,1},b_{i,2},\dots,b_{i,|A'_{>,i}|/2}\}$ and $B_{i,2} = \{b_{i,|A'_{>,i}|/2+1},\dots,b_{i,|A'_{>,i}|}\}$. 
Let $k'=2560\ell n/m+1$. Theorem~\ref{thm:Sze-Vu-2-strong} will follow easily from the next two claims.

\vspace{2mm}

\noindent {\bf Claim 1.} Suppose that $C$ and $n$ are sufficiently large. Then, with probability at least $3/4$, for all $i\in[\ell]$ and all $j\in[|A'_{>,i}|/2+1,|A'_{>,i}|]$,
\[
|\Sigma_{b_{i,j}}^{[k']}(A'_{<,i}\cup B_{i,1})|\ge\frac{b_{i,j}}{4}.
\]

\noindent {\bf Claim 2.} Let $M(i)$ be the sum of the largest $k'$ elements in $B_{i,2}$, let $\overline{M}(i)$ be the sum of the largest $2k'$ elements in $A'_{i}$ and let $M$ be the sum of the largest $2 \ell k'$ elements in $A'$. Then, with probability at least $3/4$, for all $i\in [\ell]$, $$M(i) \ge \frac{M}{8\ell}$$ and $$\overline{M}(i) \le \frac{4M}{\ell}.$$

\vspace{2mm}

Before proving these claims, we show how to complete the proof of Theorem \ref{thm:Sze-Vu-2-strong} assuming that their conclusions both hold, which happens with probability at least $1/2$. For any subset $J$ of $[|A'_{>,i}|/2+1,|A'_{>,i}|]$ of size $k'$, let $J=\{j_1,j_2,\dots,j_{k'}\}$ for $j_1<j_2<\dots<j_{k'}$. A straightforward adaptation of Lemma \ref{lem:modp} shows that for any set of integers $A$ and any integer $m\notin A$, we have $|\Sigma^{[h+1]}(A\cup \{m\})| \ge |\Sigma^{[h]}(A)|+|\Sigma_m^{[h]}(A)|$. For each $v \le k'$, apply this statement with $h=k'+v-1$, $m=b_{i,j_v}$ and $A = A'_{<,i}\cup B_{i,1}\cup \{b_{i,j_1},\dots,b_{i,j_{v-1}}\}$ to conclude that
\begin{align*}
&|\Sigma^{[k'+v]}(A'_{<,i}\cup B_{i,1} \cup \{b_{i,j_1},\dots,b_{i,j_v}\})| \\
&\ge |\Sigma^{[k'+v-1]}(A'_{<,i}\cup B_{i,1} \cup \{b_{i,j_1},\dots,b_{i,j_{v-1}}\})| + |\Sigma_{b_{i,j_v}}^{[k'+v-1]}(A'_{<,i}\cup B_{i,1} \cup \{b_{i,j_1},\dots,b_{i,j_{v-1}}\})| \\
&\ge |\Sigma^{[k'+v-1]}(A'_{<,i}\cup B_{i,1} \cup \{b_{i,j_1},\dots,b_{i,j_{v-1}}\})|+|\Sigma_{b_{i,j_v}}^{[k']}(A'_{<,i}\cup B_{i,1})|.
\end{align*} Thus, $$|\Sigma^{[2k']}(A'_{<,i}\cup B_{i})| \ge |\Sigma^{[2k']}(A'_{<,i}\cup B_{i,1} \cup \{b_{i,j_1},\dots,b_{i,j_{k'}}\})| \ge \sum_{v\le k'} |\Sigma_{b_{i,j_v}}^{[k']}(A'_{<,i}\cup B_{i,1})|.$$
By Claim 1, we have, for each $v\le k'$, that $|\Sigma_{b_{i,j_{v}}}^{[k']}(A'_{<,i}\cup B_{i,1})|\ge b_{i,j_{v}}/4$. Thus, $$|\Sigma^{[2k']}(A'_{<,i}\cup B_{i})| \ge \frac{1}{4}\sum_{v\le k'}b_{i,j_v} = \frac{1}{4}\sum_{j\in J}b_{i,j}.$$ By choosing $J$ to be the set of indices of the $k'$ largest elements of $B_{i,2}$, we obtain that 
$$|\Sigma^{[2k']}(A'_{i})| \ge \frac{1}{4}M(i).$$ 
Therefore, by Claim 2, we have that, for all $i\in[\ell]$, 
\[
|\Sigma^{[2k']}(A'_{i})|\ge\frac{M}{32\ell} \ge \frac{1}{128} \frac{4M}{\ell}.
\] 
Also by Claim 2, $\Sigma^{[2k']}(A_{i}')\subseteq[\overline{M}(i)]\subseteq [4M/\ell]$. 
Therefore, by Lemma \ref{lem:Lev} with $q = 4M/\ell$ and $n = q/128$, $\Sigma^{[2k']}(A'_{1})+\cdots+\Sigma^{[2k']}(A'_{\ell})$
contains an interval of length at least 
\[
\frac{1}{256}\cdot\frac{4M}{\ell}\cdot \ell=\frac{1}{64}M\ge\frac{1}{128}\ell k' m> n,
\]
where we used that $M$ is the sum of the largest $2 \ell k'<|A'|/2$ elements of $A'$, so that $M \ge 2\ell k' |A'|/2 > \ell k' m/2$, the bound $k'\ge\frac{2560 \ell n}{m}$ and $\ell=2^{15}$. Thus, $\Sigma^{[2\ell k']}(A')$ contains an interval of length at least $n$. 
\end{proof}

\begin{proof}[Proof of Claim 1]
Assume that, for some $i\in[\ell]$ and $j\in[|A'_{>,i}|/2+1,|A'_{>,i}|]$,
\begin{equation}
|\Sigma_{b_{i,j}}^{[k']}(A'_{<,i}\cup B_{i,1})|<\min\left(\frac{m^{2}}{1024\ell^2},\frac{b_{i,j}}{4}\right). \label{eq:ineq-wrong}
\end{equation}
Note that the size of the set $A'_{<,i}\cup B_{i,1}$, considered as a subset of $\mathbb{Z}_{b_{i,j}}$, is at least $\frac{7|A'|}{8\ell}>\frac{m}{2\ell}$ and at most $\frac{|A'|}{\ell}<\frac{m}{\ell}$, since all elements of $A'_{<,i}$ are smaller than $b_{i,j}$ and, hence, are distinct modulo $b_{i,j}$. But then, since $\frac{m}{2\ell} > 80(\log b_{i,j})^2$, Lemma \ref{lem:structure-1} with $b=b_{i,j}$ and $A = A'_{<,i}\cup B_{i,1}$ implies that if (\ref{eq:ineq-wrong}) holds, there must be some $d\in [2,8\ell n/m]$ such that all but at most $64(\log n)^2+8d$ elements of $A'_{<,i}\cup B_{i,1}$ are divisible by $d$.

Since $A'$ is nice, for each $d\in [2,8\ell n/m]$, at least $512\ell (\log n)^2 + 64 \ell d$
elements of $A'$ are not divisible by $d$. By the pigeonhole principle, we obtain that, for each $d \in [2,8\ell n/m]$, either $A'_>$ or $A'_<$ contains at least $256\ell (\log n)^2+32 \ell d$ elements not divisible by $d$. 

Note that $A'_{<,i}$ is distributed as a uniformly random subset of $A'_{<}$ of size $|A'_<|/\ell$ and $B_{i,1}$ is distributed as a uniformly random subset of $A'_>$ of size $|A'_>|/(2\ell)$. Consider the event ${\cal E}(i)$ that, for some $d\in [2,8\ell n/m]$, $A'_{<,i}\cup B_{i,1}$ contains at most $64(\log n)^2+8d$ elements which are not divisible by $d$. By the argument of Lemma \ref{lem:diverse} and a union bound over all $d\le 8\ell n/m\le n$, ${\cal E}(i)$ happens with probability at most $n\exp(-256\ell (\log n)^2/(16\ell)) < 1/n$. Thus, by a union bound over all $i\in [\ell]$, for sufficiently large $n$, the probability of the event $\bigcup_{i\in [\ell]}{\cal E}(i)$ is at most $1/4$. 

By our earlier observations, (\ref{eq:ineq-wrong}) cannot hold under the complement of the event $\bigcup_{i\in [\ell]}{\cal E}(i)$, so, provided $m\ge C\sqrt{n}$ for sufficiently large $C$,  
\begin{equation*}
|\Sigma_{b_{i,j}}^{[k']}(A'_{<,i}\cup B_{i,1})|\ge \min\left(\frac{m^{2}}{1024\ell^2},\frac{b_{i,j}}{4}\right)=\frac{b_{i,j}}{4} 
\end{equation*}
holds for all $i\in[\ell]$ and all $j\in[|A'_{>,i}|/2+1,|A'_{>,i}|]$ with probability at least $3/4$. 
\end{proof}

\begin{proof}[Proof of Claim 2]
Since $A'$ is nice, for each dyadic interval $I_{j}$ in $[n]$, either $A'$ is disjoint from $I_{j}$ or $A'$ intersects $I_{j}$ in at least $64\ell(\log n)$ elements. Note that there exists $j_0$ such that the dyadic intervals $I_j$ which intersect $A'$ have at least $64\ell(\log n)$ common elements with $A'_>$ for $j>j_0$, $|I_{j_0} \cap A'_>| < 64\ell (\log n)$ and $I_j \cap A'_> = \emptyset$ for $j<j_0$. As in the proof of Lemma \ref{lem:diverse}, Chernoff's inequality for hypergeometric distributions implies that the probability $|A'_{>,i} \cap I_j| > 2|A'_>\cap I_j|/\ell$ is at most $\exp(-|A'_>\cap I_j|/(3\ell)) \le \exp(-2\log n)$. Similarly, the probability that $|B_{i,2} \cap I_j| < |A'_>\cap I_j|/(4\ell)$ is at most $\exp(-|A'_>\cap I_j|/(32\ell)) \le \exp(-2\log n)$. Thus, by a union bound, with probability at least $3/4$, for each $i\in [\ell]$ and $j>j_0$, 
\begin{equation}
|B_{i,2} \cap I_j| \ge \frac{1}{4\ell} |A'_> \cap I_j| \label{eq:claim2eq}
\end{equation} and 
\begin{equation}
|A'_{>,i} \cap I_j| \le \frac{2}{\ell} |A'_> \cap I_j|.\label{eq:claim2eq2}
\end{equation} 

Assume now that (\ref{eq:claim2eq}) and (\ref{eq:claim2eq2}) hold for all $i\in [\ell]$ and $j>j_0$. Note that since $k' = 2560\ell n/m+1$ and $m \ge C\sqrt{n}$, we have $|A'_>| > 2\ell k' + 64\ell (\log n)$. Let $X$ be the set consisting of the largest $2\ell k'$ elements of $A'_>$, which is the same as the set of the largest $2\ell k'$ elements of $A'$. Observe that there is $j_1 > j_0$ 
such that, for all $j>j_1$, $X \supseteq A'_> \cap I_j$ and, for all $j<j_1$, $X \cap I_j =\emptyset$. Since, for each $j\ge j_1 > j_0$ and $i\in [\ell]$, $|B_{i,2} \cap I_j| \ge \frac{1}{4\ell} |A'_> \cap I_j| \ge \frac{1}{4\ell}|X\cap I_j|$, we have that $B_{i,2}$ contains a subset with $\left \lceil \frac{1}{4\ell} |X\cap I_j| \right\rceil$ elements in $I_j$ for each $j\ge j_1$. We next show that $\sum_{j\ge j_1} \left \lceil \frac{1}{4\ell} |X\cap I_j| \right\rceil < k'$. Indeed, let $t$ be the number of indices $j\ge j_1$ such that $X\cap I_j \ne \emptyset$. Note that $\sum_{j\ge j_1} |X\cap I_j| = |X| = 2\ell k' > 256\ell$ and, for each $j>j_1$ for which $X\cap I_j$ is non-empty, $|X\cap I_j|=|A'_>\cap I_j| \geq 64\ell (\log n)$. Thus, 
\begin{equation}
t\le 1+\frac{|X|}{64\ell (\log n)} < \frac{k'}{16}. \label{eq:bound-t}
\end{equation}Therefore, $$\sum_{j\ge j_1} \left \lceil \frac{1}{4\ell} |X\cap I_j| \right\rceil \le t + \sum_{j\ge j_1} \frac{1}{4\ell} |X\cap I_j| < \frac{k'}{16}+\frac{k'}{2} < k'.$$ 
Since $B_{i,2}$ contains a subset of size less than $k'$ with at least $\left\lceil\frac{1}{4\ell} |X\cap I_j|\right\rceil$ elements in $I_j$ for each $j\ge j_1$ and $M(i)$ is the sum of the $k'$ largest elements of $B_{i,2}$, one has 
\[
M(i) \ge \sum_{j\ge j_1} 2^{j-1} \cdot \frac{1}{4\ell} |X \cap I_j| = \frac{1}{8\ell}\sum_{j\ge j_1} 2^{j} \cdot |X \cap I_j| \ge \frac{M}{8\ell}.
\]
We also have $|A'_{>,i} \cap I_j| \leq \frac{2}{\ell}|A'_>\cap I_j| = \frac{2}{\ell}|X\cap I_j|$ for all $j>j_1$. Moreover, 
$$\sum_{j\ge j_1}\left\lfloor \frac{2}{\ell}|X\cap I_j| \right\rfloor \ge \sum_{j\ge j_1}\frac{2}{\ell}|X\cap I_j| - t \ge 4k'-\frac{k'}{16}>2k',$$
where we used the bound (\ref{eq:bound-t}). 
Thus, there exists a set of size at least $2k'$ containing $\left\lfloor \frac{2}{\ell}|X\cap I_j| \right\rfloor$ elements in $I_j$ for each $j \ge j_1$ such that the elements of this set dominate the $2k'$ largest elements of $A'_{>,i}$. 
Hence,
\[
\overline{M}(i) \le \sum_{j\ge j_1} 2^{j} \cdot \frac{2}{\ell} |X \cap I_j| \le \frac{4M}{\ell},
\]
completing the proof of Claim 2. 
\end{proof}

Both S\'{a}rk\"{o}zy~\cite{Sar2} and Szemer\'edi and Vu~\cite{SV2} also state results which apply to 
$$\Sigma^{(h)}(A)=\left\{\sum_{s\in S}s:S\subseteq A, |S|=h\right\},$$
the set of subset sums formed by adding exactly $h$ distinct elements from $A$.
By a small modification of our proof, we can also derive the following variant of Theorem \ref{thm:Sze-Vu-2-strong} that applies in this context. 

\begin{thm}
There exists an absolute constant $C > 0$ such that the following holds. For any subset $A$ of $[n]$ of size $m\ge C\sqrt{n}$, there exists $d\ge 1$ and $r\in [0,d-1]$ such that, for $A'=\{(x-r)/d\,:\,x\in A, d|(x-r)\}$ and $k\ge C n/m$, $\Sigma^{(k)}(A')$ contains an interval of length at least $n$. Furthermore, 
\[
|A|-|A'|\le C((\log n)^3+n/m).
\]
\end{thm}

That is, if $A \subset [n]$ has size $m \geq C \sqrt{n}$ and $k \geq Cn/m$, then $\Sigma^{(k)}(A)$ contains an arithmetic progression of length at least $n$.
Since we do not need this variant and the proof is rather similar to that of Theorem~\ref{thm:Sze-Vu-2-strong}, we omit the details. 

Instead, we conclude the section by proving Corollary~\ref{cor:nonavg}, that there is a constant $C$ such that $H(n)$ and $h(n)$ are both at most $C \sqrt{n}$, where we recall that $H(n)$ is the largest integer for which there are two subsets of $[n]$ of size $H(n)$ whose sets of subset sums have no non-zero common element and $h(n)$ is the size of the largest non-averaging subset of $[n]$.

\begin{proof}[Proof of Corollary \ref{cor:nonavg}]
For the bound on $H(n)$, we need to show that for any two subsets $S_1,S_2 \subset [n]$, each of size $m \geq C\sqrt{n}$, there are non-empty subsets $S_1' \subset S_1$ and $S_2'\subset S_2$ such that $\sum_{s_1\in S'_1}s_1=\sum_{s_2 \in S'_2} s_2$.
To this end, order the elements of $S_1 \cup S_2$ in increasing order and let $M$ be the median. Without loss of generality, we may assume that the smallest $m/2$ elements from $S_1$ are each at most $M$ and the largest $m/2$ elements from $S_2$ are each at least $M$. Let $R_1\subset S_1$ consist of the smallest $m/2$ elements from $S_1$ and $R_2 \subset S_2$ consist of the largest $m/2$ elements of $S_2$.

Applying Theorem \ref{thm:Sze-Vu-2-strong} to $R_1$, we see that, provided $C$ is sufficiently large, $\Sigma^{[k]}(R_1)$ with $k=2^{52}n/m$ contains a homogeneous arithmetic progression $P$ of common difference $d \leq 4M/m$ and length at least $2n$ whose minimum element is at most $kM$. Note now, by the pigeonhole principle, that any $d$ element sequence contains a subsequence (consisting of consecutive terms) whose sum is divisible by $d$. We may therefore partition $R_2$ greedily into subsets $T_1 \cup \cdots \cup T_s$, each of size at most $d$, such that for each $i<s$ the sum of the elements in $T_i$ is a multiple of $d$. Note that the sum of the elements in any $T_i$ is at most $dn$, while the sum of all the elements in $R_2 \setminus T_s$ is at least $(m/2-d)M \geq kM$. It therefore follows that, for some $j$, the sum $\sum_{i=1}^j \sum_{t \in T_i}t$, which is a sum of elements from $S_2$, lies in the homogeneous arithmetic progression $P$.

For the bound on $h(n)$, we apply Straus' inequality $h(n) \leq 2H(n) + 2$ (see~\cite{Straus}), whose proof we include for completeness. Indeed, suppose that we have a subset of $[n]$ of size $p = 2H(n) + 3$, say $\{a_1, a_2, \dots, a_p\}$ with $a_1 < a_2 < \dots < a_p$. Writing $q = H(n) + 2$, we see that $a_q$ is the median element and the sets $\{a_q - a_i : 1 \leq i < q\}$ and $\{a_j - a_q : q < j \leq p\}$ are both subsets of $[n]$ of size $H(n) + 1$. Therefore, by the definition of $H(n)$, there must be sets $I \subset [q-1]$ and $J \subset [q+1, p]$ such that $\sum_{i \in I} (a_q - a_i) = \sum_{j \in J} (a_j - a_q)$. Rearranging, we see that
$a_q = \frac{1}{|I| + |J|} \sum_{i \in I \cup J} a_i$, so the set is not non-averaging.
\end{proof}

\section{\label{sec:extremal-sum}Subsets avoiding a given subset sum}

Recall that $g(n,m)$ is the maximum size of a subset of $[n]$ with no subset sum equal to $m$. Using the results of Section \ref{sec:Sze-Vu-proof}, we now prove Theorem~\ref{thm:Alon-conjecture}, giving the precise value of $g(n,m)$. Theorem~\ref{thm:Alon-conjecture} states that there is a constant $C$ such that if $m\in\left[Cn(\log n),\frac{n^{2}}{12(\log n)^{2}}\right]$, then 
$$g(n,m)=s(n,m) := \left\lfloor \frac{n}{\textrm{snd}(m)}\right\rfloor +\textrm{snd}(m)-2,$$ 
where $\textrm{snd}(m)$ is the smallest positive integer which is not a divisor of $m$. Moreover, if $m\in \left[\frac{n^{2}}{12(\log n)^{2}},\binom{n+1}{2}\right]$, then $g(n,m)=\max\left(s(n,m),(1+o(1))\sqrt{2m}\right)$.

\begin{proof}[Proof of Theorem \ref{thm:Alon-conjecture}] We consider the cases $m\le \frac{n^{2}}{12(\log n)^{2}}$ and $m>\frac{n^{2}}{12(\log n)^{2}}$ separately.

\vspace{2mm}

\noindent {\bf Case 1.} $Cn(\log n) \le m\le\frac{n^{2}}{12(\log n)^{2}}$ for $C$ sufficiently large. 

\vspace{2mm}

Let $A\subseteq[n]$ be such
that $m\notin\Sigma(A)$. Assume that $|A|\ge s(n,m)+1$.
We claim that $\textrm{snd}(m)\le1.01\log m$ for $m$ sufficiently large. Indeed, if $\textrm{snd}(m) > 1.01\log m$, then $m \ge \textrm{lcm}(1,2,\dots,1.01\log m)$. It is easy to see that $\textrm{lcm}(1,2,\dots,1.01\log m) = \exp\left(\sum_{x\le 1.01\log m} \Lambda(x)\right)$, where $\Lambda$ is the von Mangoldt function given by $\Lambda(x)=\log p$ if $x=p^k$ is a prime power and $\Lambda(x)=0$ otherwise. But, by the prime number theorem, $\sum_{x\le 1.01\log m} \Lambda(x) \ge 1.005\log m$ for $m$ sufficiently large, so that $m \ge \textrm{lcm}(1,2,\dots,1.01\log m) \ge \exp(1.005\log m)$, a contradiction. Thus, $\textrm{snd}(m) \le 1.01\log m \le 2.02 \log n$ and, in particular, $|A|\ge\frac{n}{2.02\log n}$. 

Let $A^{*}$ be a random
subset of $A$ where each element is chosen independently with probability
$1/10$. By Hoeffding's inequality, $|A|/9\ge |A^*|\ge |A|/11$ with high probability. 
Suppose that $2\le d\le n$ is such that there are at least $(\log n)^3$ elements in $A$ which are not divisible by $d$. Again by Hoeffding's inequality, the probability that the number of elements in $A$ which are not divisible by $d$ is more than $20$ times larger than the number of elements in $A^{*}$ which are not divisible by $d$ is at most $\exp(-(\log n)^3/800)$. Thus, by the union bound, the probability that there exists $d\in [2,n]$ such that there are at least $(\log n)^3$ elements in $A$ which are not divisible by $d$ and the number of elements in $A$ which are not divisible by $d$ is at least $20$ times larger than the number of elements in $A^{*}$ which are not divisible by $d$ is at most $n\exp(-(\log n)^3/800)<1/4$. Denote this latter event by ${\cal E}$ and assume from here on that $A^*$ has been chosen so that $|A|/9\ge |A^*|\ge |A|/11$ and ${\cal E}$ does not hold.

By Theorem \ref{thm:Sze-Vu-2-strong}, there exists $d$ such that,
for $A'=\{x/d:\,x\in A^{*},d|x\}$, we have $|A'|\ge|A^{*}|-2^{30}(\log n)^3 - \frac{2^{30}n}{|A^*|}$
and $\Sigma^{[k]}(A')$ contains an interval $I$ of length at least
$n$ for $k=2^{50}n/h$, where $h=|A^{*}|$. Note that $|A'| \ge |A^{*}|-2^{30}(\log n)^3 - \frac{2^{30}n}{|A^*|} \ge |A^{*}| - 2^{40}(\log n)^{3}$, as $|A^*| \ge \frac{|A|}{11} \ge \frac{n}{22.22\log n}$. Since ${\cal E}$ does not hold, there are at most $2^{50}(\log n)^{3}$ elements
in $A$ which are not divisible by $d$. 

Let $A''=\{x/d:x\in A\setminus A^{*},d|x\}$. Since $|A^*| \le |A|/9$ and there are at most $2^{50}(\log n)^{3}$ elements in $A$ which are not divisible by $d$, the size of $A''$ is at least $|A|-|A^{*}| - 2^{50}(\log n)^{3} \ge 8|A|/9 - 2^{50}(\log n)^{3} \ge 0.87|A|$. Note that the smallest element of $I$ is at most $nk/d\le2^{50}\frac{n}{h}\cdot \frac{n}{d}$ and each element in $A''$ is at most $n/d \leq n$. Therefore, by Lemma~\ref{lem:verysimple}, $\Sigma^{[k+|A''|]}(A'\cup A'')$ contains the interval $[2^{50}\frac{n}{h}\cdot \frac{n}{d},\sum_{x\in A''}x)$.
We have 
\[
\sum_{x\in A''}x\ge\frac{|A''|^{2}}{2}\ge\frac{|A|^{2}}{2.7}\ge\frac{n^{2}}{12(\log n)^{2}}
\]
and 
\[
\frac{n}{h}\cdot n\le 30n(\log n).
\]
Hence, $\Sigma(A)$ contains all multiples $z$ of $d$ with 
$2^{60}n(\log n)\le z\le\frac{n^{2}d}{12(\log n)^{2}}$.
In particular, if $m\notin\Sigma(A)$, then $d\nmid m$. Thus, $d\ge\textrm{snd}(m)$.
Recall that at most $2^{50}(\log n)^{3}$ elements of $A$ are not
divisible by $d$. Therefore, if $d\ge\textrm{snd}(m) + 1$, then
\[
|A|\le 2^{50}(\log n)^{3}+\left\lfloor \frac{n}{d}\right\rfloor \le 2^{50}(\log n)^{3}+\frac{n}{\textrm{snd}(m)+1}<\frac{n}{\textrm{snd}(m)}-\frac{n}{4\textrm{snd}(m)^{2}}<s(n,m),
\]
a contradiction. Thus, $d=\textrm{snd}(m)$. 

Since $|A| \ge \left\lfloor \frac{n}{\textrm{snd}(m)}\right\rfloor +\textrm{snd}(m)-1$ and at most $\left\lfloor \frac{n}{\textrm{snd}(m)}\right\rfloor $
elements in $A$ are divisible by $d$, there exist at least $d-1$ elements in $A$ which are not divisible by $d$. Let $\bar{A}$ be a set of $d-1$ such elements.
Note that $\bar{A}$ is disjoint from $\{dx:x\in A' \cup A''\}$. By Lemma \ref{lem:full-mod-d},
$\Sigma_{d}(\bar{A})$ contains a non-zero subgroup $d'\mathbb{Z}_d$ of $\mathbb{Z}_{d}$.
Since $d=\textrm{snd}(m)$, $d'|m$ for any $d'|d$ and $d' \ne d$. Thus, there exists a subset of $\bar{A}$ whose
sum $y$ is congruent to $m$ modulo $d$. Furthermore, $y$ is at most $dn$ since $|\bar{A}|=d-1$ and all elements of $\bar{A}$ are at most $n$. Noting that $d \le 4\log n$, we
have $m - y \ge Cn(\log n) - nd \ge 2^{60}n(\log n)$ for sufficiently large $C$. We also have $m - y \leq m \le \frac{n^2}{12(\log n)^2}$. Hence, $m - y\in\Sigma(\{dx:x\in A' \cup A''\})$, so $m \in\Sigma(\{dx:x\in A' \cup A''\} \cup\bar{A})$. Thus, if $|A|\ge s(n,m)+1$,
then $m\in\Sigma(A)$. Hence, $g(n,m)\le s(n,m)$. Since we already noted in the introduction that $g(n,m)\ge s(n,m)$, this completes the proof in this case.

\vspace{2mm}

\noindent {\bf Case 2.} $\frac{n^{2}}{12(\log n)^{2}}\le m\le \binom{n+1}{2}$. 

\vspace{2mm}

Let $A \subseteq [n]$ be such that $m\notin\Sigma(A)$. Assume that 
$$|A|\ge1+\max\left(s(n,m),\sqrt{2m}(1+2^{50}(\log n)^{2}/n^{1/3})\right).$$
Let $A^{*}$ be a random subset of $A$ where each element is chosen independently with probability $n^{-1/3}$. By Hoeffding's inequality, $0.9|A|/n^{1/3} \le |A^*| \le 1.1|A|/n^{1/3}$ with high probability. As in the case above, we can again define an event ${\cal E}$, in this case that there exists $d \in [2,n]$ such that there are at least $n^{1/3} (\log n)^2$ elements in $A$ which are not divisible by $d$ and the number of elements in $A$ which are not divisible by $d$ is at least $2 n^{1/3}$ times larger than the number of elements in $A^{*}$ which are not divisible by $d$, and show that it happens with probability at most $1/4$. We now fix $A^*$ with $0.9|A|/n^{1/3} \le |A^*| \le 1.1|A|/n^{1/3}$ such that ${\cal E}$ does not hold.

By Theorem \ref{thm:Sze-Vu-2-strong},
there exists $d$ such that, for $A'=\{x/d:\,x\in A^{*},d|x\}$, we have $|A'|\ge|A^{*}|-2^{30}(\log n)^3 - \frac{2^{30}n}{|A^*|}$
and $\Sigma(A')$ contains an interval $I$ of length at least $n$. Note that $|A|>s(n,m) \ge \frac{n}{2.02\log n}$, so $|A^*| \ge \frac{0.9|A|}{n^{1/3}} \ge \frac{0.4n^{2/3}}{\log n}$ and $|A'| \ge |A^{*}|-2^{30}(\log n)^3 - \frac{2^{30}n}{|A^*|} \ge |A^*| - 2^{40}n^{1/3}(\log n)$. Since ${\cal E}$ does not hold, there are at most $2^{41}n^{2/3}(\log n)$ elements of $A$ which are not divisible by $d$. 

Since $|A|\ge s(n,m)+1$,
we must again have $d\le\textrm{snd}(m)$. If $d=\textrm{snd}(m)$, then,
as above, we can find at most $d-1$
elements of $A$ which are not divisible by $d$ and
whose sum is congruent to $m$ modulo $d$. If $d<\textrm{snd}(m)$,
then $d|m$. In either case, there is a (possibly empty) sum $y$ of at most $d-1$
elements of $A$ not divisible by $d$ such
that $d | (m-y)$. Therefore, to show that $m\in\Sigma(A)$, it suffices
to show that $(m-y)/d\in\Sigma(\{x/d:x\in A,d|x\})$. 

Note that $\Sigma(A')$ contains an interval $I$ where the largest
element of $I$ is at most $\sum_{x\in A'}x\le 1.1n^{5/3}/d$ and each element in $\{x/d:x\in A\setminus A^{*},d|x\}$ is at most $n/d$. By Lemma \ref{lem:verysimple}, $\Sigma(\{x/d:x\in A,d|x\})$ contains the interval $[1.1n^{5/3}/d,\sum_{z\in A\setminus A^{*},d|z}z/d]$.
The number of elements in $A\setminus A^{*}$ which are divisible
by $d$ is at least $|A|-|A^*| - 2^{41}n^{2/3}(\log n)\ge |A|-1.1|A|/n^{1/3}-2^{41}n^{2/3}(\log n)\ge |A|(1-2^{42}n^{-1/3}(\log n)^{2})>\sqrt{2m}$.
Hence, 
\[
\sum_{z\in A\setminus A^{*},d|z}z/d\ge\sum_{i=1}^{\lceil\sqrt{2m}\rceil}i\ge m.
\]
Thus, $\Sigma(\{x/d:x\in A,d|x\})$ contains $(m-y)/d$, since $(m-y)/d\ge(m-dn)/d>1.1n^{5/3}/d$
and $(m-y)/d<m$. 

Hence, 
\[
g(n,m)\le\max\left(s(n,m),\sqrt{2m}(1+2^{50}(\log n)^{2}/n^{1/3})\right).
\]
Since $g(n,m) \ge s(n,m)$ and $g(n,m) \ge \lfloor \sqrt{2m} \rfloor$ as the interval $[\lfloor \sqrt{2m} \rfloor-1]$ does not have a subset sum which equals $m$, we have $$g(n,m)=\max\left(s(n,m),\sqrt{2m}(1+O((\log n)^{2}/n^{1/3}))\right),$$
completing the proof.
\end{proof}

\appendix 
\section{Supplementary results for Section \ref{sec:density-ramsey}}

\subsection{The growth rate of $F$}\label{appendix:est-density-ramsey}
In Section \ref{sec:density-ramsey}, we consider a sequence of positive integers $F=(f_n)_{n\ge 1}$ which satisfies $f_n = \sum_{i\le \epsilon n}f_i$ for all $n \ge n_0$. Here we establish the asymptotic for $F$ claimed in the introduction.

\begin{claim}
Let $F=(f_n)_{n\ge 1}$ be a sequence of positive integers which satisfies $f_n = \sum_{i\le \epsilon n}f_i$ for all $n \ge n_0$. Then $f_n = \exp\left(\left(\frac{1}{2\log (1/\epsilon)}+o(1)\right)(\log n)^2\right)$.
\end{claim}

\begin{proof}
We first show by induction that there is a constant $C$ for which $f_n \le \exp\left(\frac{1}{2\log(1/\epsilon)}((\log n)^2+C)\right)$ for all positive integers $n$, which would imply the upper bound in the claim. We can choose $C$ sufficiently large so that this holds for all $n \le \max(n_0,10/\epsilon)$. Let $m> \max(n_0,10/\epsilon)$. If $f_n\le \exp\left(\frac{(\log n)^2+C}{2\log(1/\epsilon)}\right)$ for all $n\le m-1$, then \begin{align*} 
f_m &= \sum_{i\le \epsilon m}f_i \le  \epsilon m \cdot \exp\left(\frac{1}{2\log(1/\epsilon)}((\log ( \epsilon m))^2+C)\right) \\ &= \exp\left(\frac{1}{2\log(1/\epsilon)}((\log m -\log (1/\epsilon))^2+ C + 2\log(1/\epsilon)\log(\epsilon m))\right) \\ &\le \exp\left(\frac{1}{2\log(1/\epsilon)}((\log m)^2 + C)\right),
\end{align*} 
completing the induction proof of the desired upper bound on $f_n$.

We now turn to proving the desired lower bound on $f_n$ in the claim. Let $C' = 100\log(1/\epsilon)$. Let $\tilde{g}(x) = \exp\left(\frac{(\log x)^2 - C' \log x \log \log x}{2\log(1/\epsilon)}\right)$. Note that there is $x_0 > 0$ depending only on $\epsilon$ such that $(\log x)^2 - C' \log x \log \log x$ is increasing for all $x \ge x_0$. Let $m(\epsilon)$ be the least positive integer such that, for all $m\ge m(\epsilon)$, $$\epsilon \tilde{g}(m) \le \tilde{g}(m-1/\epsilon)-\tilde{g}(x_0/\epsilon).$$ It is easy to verify that such $m(\epsilon)$ exists. Let $g(x)=\exp\left(\frac{(\log x)^2 - C' \log x \log \log x - C}{2\log(1/\epsilon)}\right)$, where $C$ is a sufficiently large constant to be chosen later. We next show by induction that, for an appropriate choice of $C$, $f_n \ge g(n)$ for all $n\ge 2$. We choose $C$ sufficiently large that the above claim holds for all $n \le \max(n_0,x_0,m(\epsilon))$. 
Let $m\ge \max(n_0,x_0,m(\epsilon))$. If $f_n\ge g(n)$ for all $n\le m-1$, then \begin{align*} f_m &= \sum_{i\le \epsilon m}f_i \ge \sum_{i\le \epsilon m} g(i) \\ &\ge \int_{x_0}^{\epsilon m - 1} g(x) dx = \int_{x_0/\epsilon}^{m-1/\epsilon} \epsilon g(\epsilon y)dy, \end{align*} where in the last step we used the change of variable $y=x/\epsilon$. Note now that 
\begin{align*}g'(y)&=\exp\left(\frac{1}{2\log(1/\epsilon)}((\log y)^2 - C' \log y \log \log y - C)\right)\cdot \frac{2 \log y - C' - C' \log \log y}{2 y \log(1/\epsilon)}\\& \le \exp\left(\frac{1}{2\log(1/\epsilon)}\left[(\log y)^2 - C' \log y \log \log y - C - 2\log(1/\epsilon)\log y + 2\log(1/\epsilon) \log(\log y/\log (1/\epsilon)) \right]\right).\end{align*} 
Thus,
\begin{align*} 
&g'(y)/g(\epsilon y) \\
&\leq \exp\left(\frac{1}{2\log(1/\epsilon)}\left[(\log y)^2 - C' \log y \log \log y - C - 2\log(1/\epsilon)\log y + 2\log(1/\epsilon) \log(\log y/\log (1/\epsilon)) \right]\right) \\ 
&\qquad \qquad  \cdot \exp\left(\frac{1}{2\log(1/\epsilon)}\left[-(\log(\epsilon y))^2+C'\log(\epsilon y)\log \log (\epsilon y)+C\right]\right)\\
&\le \exp\left(\frac{1}{2\log(1/\epsilon)}\left[2\log(1/\epsilon) \log(\log y/\log (1/\epsilon)) + C'\log(\epsilon)\log\log y\right]\right) \\ &\le \exp(-C'(\log\log y)/4) \le \epsilon^2, 
\end{align*} 
where in the last inequality we used the fact that $C' = 100\log(1/\epsilon)$. Then $$f_m \ge \int_{x_0/\epsilon}^{m-1/\epsilon} \epsilon g(\epsilon y)dy \ge \int_{x_0/\epsilon}^{m-1/\epsilon} \frac{g'(y)}{\epsilon} dy \ge \frac{g(m-1/\epsilon)-g(x_0/\epsilon)}{\epsilon} \ge g(m),$$ where in the last inequality we used the definition of $m(\epsilon)$ and the fact that $m\ge m(\epsilon)$. This completes the induction. 
\end{proof}

\subsection{Proof of Lemma \ref{lem:ramsey-density}\label{appendix:ramsey-density}}

In this subsection, we give the proof of Lemma \ref{lem:ramsey-density}, which is a key component in the proof of Theorem \ref{epsilonthm2}. First, we recall the setting and the statement of the lemma. Let $\epsilon_0>0$ be a sufficiently small constant. 
Let $B=(b_n)_{n\ge 1}$ be an $\epsilon$-friendly sequence. For $j$ sufficiently large, we choose $a_j$ to be a uniform random integer in $[b_j,b_{j+1})$ which has no prime factor at most $(\max(1/\epsilon, 1/\epsilon_0))^{4000}$ and let $A=(a_j)_{j \geq 1}$. For small $j$, we choose $a_j$ to be an arbitrary integer in $[b_j,b_{j+1})$. We let $h(i)$ be the smallest integer for which $b_{h(i)}\ge 2^i$ and $A_i = A\cap [b_{h(i)},b_{h(i+1)-1})$. 

\begin{customlemma}{4.4}
There exist positive constants $\epsilon_0$, $C_1$ and $C_2$ such that the following holds. For $i$ sufficiently large, with positive probability,
the set $A_{i}$ has the property that, for any subset $A'_{i}\subset A_{i}$
with $|A'_{i}|\geq(\min(\epsilon,\epsilon_0)/4)|A_{i}|$, $A'_{i}$ contains
a subset $A''_{i}$ with $|A''_{i}|\leq C_{1}i$ such that $\Sigma(A''_{i})$
contains every integer in the interval $[y,2y]$, where $y=C_{2}2^{i}i$. 
\end{customlemma}

The proof of this lemma has been consigned to an appendix because of its similarity to the proof of Lemma \ref{lem:main-ramsey}. Indeed, the difference between the two proofs consists mainly of minor modifications to account for the non-uniformity in the distribution of the elements of $A_i$. However, for completeness, we give the proof in full, beginning with the following lemma, which is the analogue of Lemma \ref{lem:cover} in this context. 

\begin{lem}\label{lem:ramsey-appendix}
For a sufficiently large positive constant $C_0$, the following holds. Assume that $\epsilon>0$ is sufficiently small. Let $i$ be sufficiently large and let $m$ be an integer in $[2^i,2^{i+1})$ with no prime factor at most $\epsilon^{-4000}$. If $S$ is a uniformly chosen random subsequence of $A_i$ of size $C_0 i$, then $|\Sigma_m(S)| < 2^{i-2}$ with probability less than $\epsilon^{1000C_0 i}$. 
\end{lem}

\begin{proof}
Let $w=\epsilon^{-4000}$. Denote by $X$ the set of integers $[b_{h(i)},b_{h(i+1)-1})$ with no prime divisor at most $w$. For each $j \in [h(i),h(i+1)-1)$, let $X_{j}$ be the set of integers in $[b_j,b_{j+1})$ with no prime divisor at most $w$. Let $t = h(i+1)-h(i)-1$, which is the number of intervals $[b_j,b_{j+1})$ in $[b_{h(i)},b_{h(i+1)-1})$. Note that for each interval $I$ of integers of sufficient length, the number of elements in the interval which are coprime to all the primes at most $w$ is $(\tau+o(1))|I|$, where $\tau = \phi(W)/W$ with $W$ being the product of all primes at most $w$. Since $B$ is a friendly sequence, we have $b_{j+1}-b_j \le (b_{j'+1}-b_{j'})/c$ for all $j\in [h(i)-1,h(i+1)-1]$ and $j'\in [h(i),h(i+1)-2]$. Thus, 
$$2^i \le b_{h(i+1)} - b_{h(i)-1} \le (1+2/(ct))(b_{h(i+1)-1}-b_{h(i)}).$$ 
Since $t$ tends to infinity as $i$ tends to infinity, we have that $b_{h(i+1)-1} - b_{h(i)} \ge \frac{7}{8} 2^i$ for sufficiently large $i$. Hence, for large $i$, we have that
\begin{equation}
|X| \ge \frac{2^i\tau}{2}. \label{eq:bound-X} 
\end{equation}
Again by properties of friendly sequences, the length of the intervals $[b_j,b_{j+1})$ for $j\in [h(i),h(i+1)-1)$ are within a factor $1/c$ of each other and the minimum length of an interval $[b_j,b_{j+1})$ with $j\in [h(i),h(i+1)-1)$ tends to infinity as $i$ tends to infinity. We thus obtain that all $|X_j|$ with $j\in [h(i),h(i+1)-1)$ are within a factor $2/c$ of each other for $i$ sufficiently large. Hence, 
\begin{equation}
|X_j| \ge \frac{c|X|}{2t}. \label{eq:bound-Xj}
\end{equation}

Let ${\cal D}$ be the distribution of a random integer in $[b_{h(i)},b_{h(i+1)-1})$, where the probability that an element $a\in X_j$ is chosen is $\frac{1}{|X_j|t}$. Observe that the random sequence $S$ is a sequence of $C_0 i$ random integers with distribution ${\cal D}$, subject to the condition that no two elements come from the same interval $[b_j,b_{j+1})$. 

Let $q = C_0 i$. Let $S=(s_{1},s_{2},\dots,s_{q})$. Let $S_j=(s_1, s_2, \ldots,s_j)$ denote the sequence consisting of the first $j$ elements of $S$. Let $\delta=1/w$. Call $j \in [2,q]$ {\it bad} if 
\begin{itemize} 
\item $|\Sigma_{m}(S_j)|\leq \frac{3}{2}|\Sigma_{m}(S_{j-1})|$ and $|\Sigma_{m}(S_{j-1})|\leq \frac{2^i}{w}$ or 
\item $|\Sigma_{m}(S_j)|\leq (1+\delta)|\Sigma_{m}(S_{j-1})|$ and $\frac{2^i}{w} < |\Sigma_{m}(S_{j-1})|<2^{i-2}$. 
\end{itemize}
The following two claims will allow us to complete the proof of the lemma.

\vspace{2mm}

\noindent {\bf Claim 1.} The probability that $j$ is bad conditioned on the choice of $S_{j-1}$ is at most $p:=\frac{16}{cw\tau}$.

\vspace{2mm}

\noindent {\bf Claim 2.} If $|\Sigma_m(S)|<2^{i-2}$, then all but fewer than $2i$ integers in $[2,q]$ are bad. 

\vspace{2mm}

Assuming Claim 1, for any $B \subset [2,q]$, the probability that all elements in $B$ are bad is at most $p^{|B|}$. By Merten's third theorem, we have $\tau = (e^{-\gamma}+o(1))/\log w \ge 1/(2\log w)$ for sufficiently small $\epsilon$, so 
\begin{equation}
p \le \frac{32\log w}{cw}. \label{eq:bound-p}
\end{equation}
From Claim 2, if $|\Sigma_m(S)|<2^{i-2}$, then there is a set $B$ of $q-2i$ integers $i \in [2,q]$ which are bad. Taking a union bound over all such choices of $B$, the probability that $|\Sigma_m(S)|<2^{i-2}$ is at most 
\[{q \choose q-2i}p^{|B|}={q \choose 2i}p^{|B|}<
(eC_0)^{2i} \left(\frac{32\log w}{cw}\right)^{C_0 i-2i}<\epsilon^{1000C_0 i},
\] where in the first inequality we used (\ref{eq:bound-p}) and in the second inequality we assume a sufficiently large choice of $C_0$ and note that $w=\epsilon^{-4000}$ with $\epsilon$ sufficiently small.
\end{proof}

To complete the proof, it remains to verify Claims 1 and 2. 

\vspace{2mm}

\noindent {\it Proof of Claim 1.} Fix $S_{j-1}=(s_1,\ldots,s_{j-1})$. Conditioned on this choice of $S_{j-1}$, we bound the probability that $j$ is bad. Let $T$ be the set of $k$ such that $[b_k,b_{k+1})$ contains at least one of $s_1,\ldots,s_{j-1}$. Observe that conditioned on $s_1,\dots,s_{j-1}$, the distribution of $s_j$ is supported on $\bigcup_{k\in [h(i),h(i+1)-1)\setminus T} X_k$ and, for $x_k \in X_k$ with $k\in [h(i),h(i+1)-1)\setminus T$, the conditional probability that $s_j$ is equal to $x_k$ is 
$$\frac{1}{|X_k|(t-|T|)}\le \frac{2}{|X_k|t} \le \frac{4}{c|X|},$$ 
where we used (\ref{eq:bound-Xj}). 

If $|\Sigma_{m}(S_{j-1})| \geq 2^{i-2}$, then $j$ cannot be bad (so the event that $j$ is bad has probability zero). We may therefore restrict attention to the two cases $|\Sigma_{m}(S_{j-1})| \leq 2^i/w$ and $2^i/w <|\Sigma_{m}(S_{j-1})| < 2^{i-2}$. 

For the first case, note, by Lemma \ref{lem:double-count}, that the number of $s$ with $|\Sigma_{m}(S_{j-1}\cup\{s\})|\le\frac{3}{2}|\Sigma_{m}(S_{j-1})|$ is at most $\frac{|\Sigma_{m}(S_{j-1})|^{2}}{|\Sigma_{m}(S_{j-1})|/2}=2|\Sigma_{m}(S_{j-1})|$. Therefore, if $|\Sigma_{m}(S_{j-1})| \leq 2^i/w$, the probability that $j$ is bad conditioned on $S_{j-1}$ is at most $$2|\Sigma_{m}(S_{j-1})| \cdot \frac{4}{c|X|} \le \frac{16}{cw\tau} = p,$$ where in the inequality we used (\ref{eq:bound-X}). 

Suppose now that $2^i/w <|\Sigma_{m}(S_{j-1})| < 2^{i-2}$. For a positive integer $D$, let $G_{D}$ be the set of $s$ such that $|\Sigma_{m}(S_{j-1}\cup\{s\}) |\le |\Sigma_{m}(S_{j-1})|+D$. Let $d = \lfloor \delta |\Sigma_{m}(S_{j-1})|\rfloor$, so $j$ is bad in this case if and only if $s_j \in G_d$. Let $k=\lfloor \frac{1}{2\delta}\rfloor$, so $kd \leq |\Sigma_{m}(S_{j-1})|/2$. By Lemma~\ref{lem:stable-period}, $kG_{d}\subseteq G_{kd}$, so 
$\left|kG_{d}\right|\le|G_{kd}| \le 2|\Sigma_{m}(S_{j-1})| < 2^{i-1}$, where the middle inequality is again by the consequence of Lemma \ref{lem:double-count} noted above. 

If  $|G_{d}|\le\frac{m}{w}$, then $|G_{d}|\le \frac{m}{w} \leq \frac{2^{i+1}}{w} = 2\delta 2^i$. 
Otherwise, $|G_{d}|>\frac{m}{w}$. In this case, since $m$ has no prime divisor
at most $w$, no subgroup of $\mathbb{Z}_{m}$ has size larger than
$\frac{m}{w}$. Thus, $G_{d}$ cannot be contained in a coset
of a non-trivial subgroup. By Lemma \ref{lem:Cauchy-Davenport}, since $\left|kG_{d}\right| \le 2^{i-1} < m$, we must have 
$\left|kG_{d}\right|\ge (k+1)|G_{d}|/2 \geq |G_d|/(4\delta)$. Hence, $|G_{d}| \leq 4\delta|kG_d| \leq 4\delta 2^{i-1} =2\delta 2^i$. Thus, in either case, conditioned on the choice of $S_{j-1}$, the probability that $j$ is bad, which is the same as the probability that $s_j \in G_d$, is at
most $$|G_d|\cdot \frac{4}{c|X|} \leq \frac{16\delta}{c\tau} = p,$$ where we again used (\ref{eq:bound-X}). \qed  

\vspace{2mm}

\noindent {\it Proof of Claim 2.} As $S_{j-1}\subset S_j$ for $j\in[2,q]$, $\Sigma_m(S_{j-1}) \subset \Sigma_m(S_{j})$ and, hence, $1 \leq |\Sigma_m(S_1)| \leq \cdots \leq |\Sigma_m(S_q)|=|\Sigma_m(S)| < 2^{i-2}$. Therefore, the number of $j$ which are not bad with $|\Sigma_m(S_{j-1})| \leq 2^i/w$ and $|\Sigma_m(S_{j})| \geq \frac{3}{2}|\Sigma_m(S_{j-1})|$ is at most $\log_{3/2} 2^i = i\frac{\log 2}{\log (3/2)}$, as we get a factor of $3/2$ for each such $j$. Moreover, since $(1+\delta)^{\delta^{-1} \log_2w} \geq 2^{\log_2 w} = w$, the number of elements $j$ which are not bad with $2^{i-2}>|\Sigma_m(S_{j-1})| >  2^i/w$ and $|\Sigma_m(S_j)| \geq (1+\delta)|\Sigma_m(S_{j-1})|$ is at most $\delta^{-1}\log_2 w = w\log_2 w$, as we get a factor of $1+\delta$ for each such $j$. Therefore, the number of $j \in [2,q]$ which are not bad is at most $i\frac{\log 2}{\log (3/2)}+w \log_2 w < 2i$ for sufficiently large $i$. \qed

\vspace{2mm}

We are now ready to prove Lemma \ref{lem:ramsey-density}.

\begin{proof}[Proof of Lemma \ref{lem:ramsey-density}]
By replacing $\epsilon$ with $\min(\epsilon,\epsilon_0)$, we only need to prove Lemma \ref{lem:ramsey-density} for $\epsilon\le \epsilon_0$. Thus, by choosing $\epsilon_0$ sufficiently small, it suffices to prove that the following holds for sufficiently small $\epsilon$. If $i$ is sufficiently large, then any subset $A_i'$ of $A_i$ with $|A_i'|\ge (\epsilon/4)|A_i|$ contains a subset $A_i''$ with $|A_i''|\le C_1 i$ such that $\Sigma(A_i'')$ contains every integer in the interval $[y,2y]$, where $y=C_2 2^i i$. 

For a given $C_1$ and $i$ sufficiently large, we have that $|A_i| \ge 400\epsilon^{-1}C_1 i$. Consider a random partition of $A_{i}$ into subsets of size $4\epsilon^{-1} C_1 i$ and consider a uniform random ordering of each subset as a sequence of integers. Let the obtained sequences be $A_{i,1},\dots,A_{i,u}$. 

We will show that for an appropriate choice of $C_1$, there exists a positive constant $C_2$ such that, with positive probability, the following event ${\cal E}$ holds. For all $k\le u$ and all subsequences $A'_{i,k}$ of $A_{i,k}$ of size $(\epsilon/4)|A_{i,k}|$, $\Sigma(A'_{i,k})$ contains the interval $[y,2y]$ for $y=C_2 2^i i$. 

Fix $k\le u$ and fix a subset $I'$ of $[4\epsilon^{-1}C_1i]$ of size $C_1 i$. Let $A'_{i,k}$ be the subsequence of $A_{i,k}$ consisting of elements with index in $I'$. Let $\ell$ be a constant to be chosen later. We partition $I'$ into a subset $I''$ of size $7C_1 i/8$ and $\ell$ subsets $I''_{k,1},\dots,I''_{k,\ell}$ of equal size such that each subset in the partition consists of consecutive terms from $I'$. Let $A''_{i,k}$ be the elements with index in $I''$ and, for each $j\in [\ell]$, let $S_{k,j}$ be the elements with index in $I''_{k,j}$. Let $J'_{k,j}$ be the first $|I''_{k,j}|/2$ elements of $I''_{k,j}$ and let $J''_{k,j}$ be the remaining elements. Let $S'_{k,j}$ be the elements with index in $J'_{k,j}$ and let $S''_{k,j}$ be the elements with index in $J''_{k,j}$. Then $S'_{k,j}$ has the same distribution as a random subsequence of $A_i$ of length $|S_{k,j}|/2$. We choose $C_1 = 16\ell C_0$, so that $|S'_{k,j}| = C_0 i$, where $C_0$ is the constant defined in Lemma \ref{lem:ramsey-appendix}. By Lemma \ref{lem:ramsey-appendix} and a union bound, we have that $|\Sigma_m(S'_{k,j})| \ge 2^{i-2}$ for all $m\in [2^i,2^{i+1})$ with no prime factor at most $\epsilon^{-4000}$ with probability at least $1-2^i\epsilon^{1000C_0 i} > 1-\epsilon^{800C_0i}$, assuming that $\epsilon$ is sufficiently small. 
Thus, by another union bound, with probability at least $1-\ell\epsilon^{800C_0i}$, $|\Sigma_m(S'_{k,j})| \ge 2^{i-2}$ for all $j\le \ell$ and all $m\in S''_{k,j}$. By repeated application of Lemma \ref{lem:modp}, we have that 
$$|\Sigma(S_{k,j})| = |\Sigma(S'_{k,j}\cup S''_{k,j})| \ge |S''_{k,j}|2^{i-2}\ge \frac{1}{16} 2^{i}|A'_{i,k}|/(4\ell).$$ 
We also have that $\Sigma(S_{k,j})$ is a subset of the interval $[0, 2^{i}|A'_{i,k}|/(4\ell)]$. Furthermore, $\Sigma(S_{k,j})$ is not contained in any arithmetic progression with common difference greater than $1$, as otherwise there exists $1<d\le 16$ such that all elements of $S_{k,j}$ are divisible by $d$, contradicting the fact that elements of $A_i$ do not have prime factors at most $\epsilon^{-4000}$. Thus, choosing $\ell = 33$, 
by Lemma \ref{lem:Lev}, we have that $S_{k,1}+\ldots+S_{k,\ell}$ contains an interval of length at least $2^{i}|A'_{i,k}|/(4\ell)$. Hence, $\Sigma(A'_{i,k}\setminus A''_{i,k})$ contains an interval $[a,b]$ of length at least $2^{i}|A'_{i,k}|/(4\ell) > 2^{i+1}$. Note that $a < b \le 2^{i-2}|A'_{i,k}|$. By Lemma \ref{lem:verysimple}, we then have that $\Sigma(A'_{i,k})$ contains the interval $[a,b+\sum_{x\in A''_{i,k}}x] \supset [y,2y]$ for $y = \frac{1}{4}2^{i}|A'_{i,k}| = \frac{C_1}{4}2^i i$. Let $C_2 = \frac{C_1}{4}$. 

By taking a union bound over all possible choices of $I'$, the probability that there exists a subsequence $A'_{i,k}$ of $A_{i,k}$ of size $(\epsilon/4)|A_{i,k}|$ such that $\Sigma(A'_{i,k})$ does not contain the interval $[y,2y]$ for $y=C_2 2^i i$ is at most ${|A_{i,k}| \choose (\epsilon/4)|A_{i,k}|} \ell \epsilon^{800C_0 i}$. By a union bound over all $k\le u$, we then obtain that the event ${\cal E}$ holds with probability at least 
$$1 - u {|A_{i,k}| \choose (\epsilon/4)|A_{i,k}|} \ell \epsilon^{800C_0 i} \ge 1 - 2^{i+7} (4e/\epsilon)^{C_1 i}\epsilon^{800C_0 i}  \ge 1 - \epsilon^{200C_0 i} > 0,$$ 
where we used that $\epsilon$ is sufficiently small, $C_1 = 16 \ell C_0$ with $C_0$ sufficiently large, $\ell = 33$ and $u\le 2^{i+1}$.

Assume now that the event ${\cal E}$ holds. For any subset $A'_i$ of $A_i$ such that $|A'_i|\ge (\epsilon/4)|A_i|$, there exists $k\le u$ such that $|A'_i \cap A_{i,k}| \ge (\epsilon/4)|A_{i,k}|$. Thus, defining $A''_i$ to be an arbitrary subset of $A'_i \cap A_{i,k}$ of size $(\epsilon/4)|A_{i,k}| = C_1 i$, we have that $\Sigma(A''_i)$ contains the interval $[y,2y]$ for $y=C_2 2^i i$, as required. 
\end{proof}

\section{Supplementary results for Section \ref{sec:Monochromatic-subset-sums}}

\subsection{Number-theoretic estimates}\label{sec:number-theoretic}

This short section contains the proofs of some number-theoretic estimates which were used in Section~\ref{sec:Monochromatic-subset-sums}.
We will need the following simple lemma. 

\begin{lem}\label{lem:sum-div}
One has 
\[
\frac{m}{\zeta(2)\phi(m)}\le\prod_{p|m}\left(1+\frac{1}{p}\right)\le\sum_{u|m}\frac{1}{u}\le\prod_{p|m}\left(1+\frac{1}{p-1}\right)=\frac{m}{\phi(m)}.
\]
\end{lem}

\begin{proof}
By considering the squarefree divisors of $m$, we have 
\[
\sum_{u|m}\frac{1}{u}\ge\prod_{p|m}\left(1+\frac{1}{p}\right).
\]
On the other hand, 
\[
\sum_{u|m}\frac{1}{u}\le\prod_{p|m}\left(1+\frac{1}{p}+\frac{1}{p^{2}}+\cdots\right)=\prod_{p|m}\left(1+\frac{1}{p-1}\right).
\]
Furthermore, 
\[
\prod_{p|m}\left(1+\frac{1}{p}\right)=\prod_{p|m}\left(1+\frac{1}{p-1}\right)\cdot\prod_{p|m}\left(1-\frac{1}{p^{2}}\right)=\frac{m}{\phi(m)}\prod_{p|m}\left(1-\frac{1}{p^2}\right) \ge\frac{m}{\zeta(2)\phi(m)}. \qedhere
\] 
\end{proof}

Our first aim is to prove Lemma \ref{lem:pseudoprimes-count}, which gives upper and lower bounds on the number of integers in an interval with certain number-theoretic properties. The following lemma, of a similar flavor, is a key component in the proof. Recall that $W(r)=\prod_{i=1}^{r}p_i$, where $p_i$ is the $i^{\textrm{th}}$ prime, and $\tau(r,m)=\phi(W(r)m)/(W(r)m)=\prod_{p|W(r)m}(1-1/p)$.

\begin{lem}\label{31appendixextra}
Let $r$, $n$ and $m$ be positive integers such that $m\in [n,\binom{n}{2}]$, $r\le n$ and $r$ is sufficiently large. For any interval $I=[x,2x)$ with $x\ge n^{1/6}$, the number of integers in $I$ which are coprime to $W(r)m$ is at most $8\tau(r,m)x$. If also $x \ge r^{1.5}$, then the number of integers in $I$ which are coprime to $W(r)m$ is at least $\frac{1}{4}\tau(r,m)x$. 
\end{lem}

\begin{proof}
By \cite[Theorem 7.11]{MV}, for each interval $I=[x,2x)$ with $x \ge p_r/2$, 
the number of integers in $I$ which are coprime to $W(r)$ is at most $(1+o(1))\frac{x}{\log p_r} \le 2\tau(r,1)x$, where we used that $p_r=(1+o(1))r\log r$ and (\ref{eq:bound-Wrho}). For $x<p_r/2$, the number of integers in $I$ which are coprime to $W(r)$ is $0 < 2\tau(r,1)x$. If also $x \ge r^{1.5} > p_r$, then the number of integers in $I$ which are coprime to $W(r)$ is at least $\left(\frac{1}{2}-o(1)\right)\frac{x}{\log p_r} \ge \frac{1}{2}\tau(r,1)x$, again using (\ref{eq:bound-Wrho}).

Consider the case $r\ge (\log m)/100$. Then $\tau(r,1)\le 1/\log r \le 4\tau(r,m)$ by (\ref{eq:bound-Wrho}) and (\ref{eq:tau-log}), 
so the number of integers in $I$ which are coprime to $W(r)m$ is at most $8\tau(r,m)x$. For $x\ge r^{1.5}$, we have seen that there are at least $\frac{1}{2}\tau(r,1)x$ integers in $I$ which are coprime to $W(r)$. For each prime factor $p$ of $m$ that is larger than $p_r>r$, the number of integers in $I$ divisible by $p$ and coprime to $W(r)$ is the same as the number of integers in $[x/p,2x/p)$ coprime to $W(r)$, which is at most $2\tau(r,1)x/p$. Since there are at most $(\log m)/(\log r)$ such prime factors, the number of integers in $I$ which are coprime to $W(r)m$ is at least 
\begin{align*}
\frac{1}{2}\tau(r,1)x - \frac{\log m}{\log r} \cdot \frac{2\tau(r,1)x}{r} 
&\ge \frac{1}{2}\tau(r,1)x - \frac{200r\tau(r,1)x}{r\log r}\\
&> \tau(r,m)x/4,
\end{align*} 
where, in the first inequality, we used the assumption $r\ge (\log m)/100$ and, in the second inequailty, we used that $r$ is sufficiently large and $\tau(r,m) \le \tau(r,1)$ by (\ref{eq:mrho}). 

Next, consider the case $r<(\log m)/100$. By the inclusion-exclusion principle, the number of integers in $I=[x,2x)$ which are coprime to $W(r)m$ is
\[
x+\sum_{k=1}^{M}(-1)^{k}\sum_{\substack{p_1<p_2<\dots<p_k,\\p_1,p_2,\dots,p_k|W(r)m}}\frac{x}{p_1p_2\cdots p_k} + O(2^M),
\] 
which is within an additive $O(2^M)$ of $x\prod_{p|W(r)m}(1-1/p)$, where $M$ is the number of distinct primes that divide $W(r)m$. Since $M\le r+2(\log m)/(\log \log m) < (\log n)/10$ and $x\ge n^{1/6}$, the number of integers in $I$ coprime to $W(r)m$ is at least $\frac{1}{2}\tau(r,m)x$ and at most $2\tau(r,m)x$. 
\end{proof}

\begin{customlemma}{5.2}
Let $r$, $n$ and $m$ be positive integers such that $m\in [n,\binom{n}{2}]$, $r\le n$ and $r$ is sufficiently large. For any interval $I=[x,2x)$ with $x\ge n^{1/4}$, there are at most $8(m/\phi(m))\tau(r,m)x$ integers in $I$ of the form $qu$, where $u|m$, $u\le x^{1/16}$ and $q$ is coprime to $W(r)m$. If also $x \ge r^2$, then there are at least $\frac{1}{8}(m/\phi(m))\tau(r,m)x$  integers in $I$ of this form.
\end{customlemma}

\begin{proof}
Observe that for $x\ge n^{1/4}$ and each fixed $u|m$ with $u\le x^{1/16}$, Lemma \ref{31appendixextra} implies that the number of integers in $I$ of the form $qu$ where $q$ is coprime to $W(r)m$, which is the same as the number of integers in $[x/u,2x/u)$ which are coprime to $W(r)m$, is at most $8\tau(r,m)x/u$, where we used that $x/u \ge n^{1/6}$. If also $x\ge r^2$, then Lemma \ref{31appendixextra} similarly implies that the number of integers in $I$ of the form $qu$ where $q$ is coprime to $W(r)m$ is at least $\frac{1}{4}\tau(r,m)x/u$, where we used that $x/u \ge r^{1.5}$. 

Hence, the number of integers in $I$ of
the form $qu$, where $u|m$, $u\le x^{1/16}$ and $q$ is coprime
to $W(r)m$, is at least 
\begin{align*}
\frac{1}{4}\tau(r,m)\sum_{u|m,\,u\le x^{1/16}}\frac{x}{u} & \ge\frac{1}{4}\tau(r,m)\left(\sum_{u|m}\frac{1}{u}-\frac{\sigma(m)}{x^{1/16}}\right)x\\
 & \ge\frac{1}{4}\tau(r,m)\cdot\frac{m}{2\phi(m)}x\\
 & \ge\frac{1}{8}(m/\phi(m))\tau(r,m)x,
\end{align*}
where $\sigma(m)$ is the number of positive divisors of $m$, which is smaller
than $m^{1/100}$ for $m$ sufficiently large, and we used Lemma \ref{lem:sum-div} in the second inequality. Similarly, the number
of integers in $I$ of the form $qu$, where $u|m$, $u\le x^{1/16}$
and $q$ is coprime to $W(r)m$, is at most 
\[
8\tau(r,m)\sum_{u|m,\,u\le x^{1/16}}\frac{x}{u}\le8(m/\phi(m))\tau(r,m)x,
\] 
where we again used Lemma \ref{lem:sum-div}. 
\end{proof}

We now prove Lemma \ref{lem:selberg}, which gives an upper bound on the number of integers in an arithmetic progression which are coprime to $W(r)/\gcd(W(r),m)$. The proof employs the Selberg sieve. 

\begin{customlemma}{5.9} Let $r$ and $n$ be sufficiently large positive integers and $m\in [n,\binom{n}{2}]$. Let $X$ be an arithmetic progression of size $|X|\ge r^{1/16}$ with common difference $b\le n$. Then the number of elements
of $X$ which are coprime to $W(r)/\gcd(W(r),m)$ is at most 
\[
\frac{256|X|\log\log n}{\log r}.
\]
Furthermore, when $b=1$, the number of elements
of $X$ which are coprime to $W(r)/\gcd(W(r),m)$ is at most 
\[
256|X|\prod_{p|W(r),p\nmid m}(1-1/p). 
\]
\end{customlemma}

\begin{proof}
First, we prove the lemma in the case where the elements of the arithmetic progression are coprime to $b$. By the Selberg sieve \cite[Theorem 3.8]{MV}, applied with $q = b$ and $P = W(r)/\gcd(W(r),bm)$, which is coprime to $b$, the number of integers coprime to $W(r)/\gcd(W(r),bm)$ contained in any arithmetic progression of length $k\ge r^{1/16}$ and common
difference $b$ is at most 
\begin{equation*}
2k\prod_{p|(W(r)/\gcd(W(r),bm)),\,p\le \sqrt{k}}\frac{p-1}{p}\le 2k\prod_{p|(W(r)/\gcd(W(r),bm)),\,p\le r^{1/32}}\frac{p-1}{p}.
\end{equation*}
Since each prime $p\le r^{1/32}<r$ is either a divisor of $\gcd(W(r),bm)$ or a divisor of $W(r)/\gcd(W(r),bm)$, for $r$ sufficiently large, we have that
$$\left(\prod_{p|\gcd(W(r), bm)}\frac{p-1}{p}\right)\cdot \left(\prod_{p|(W(r)/\gcd(W(r),bm)),\,p\le r^{1/32}}\frac{p-1}{p}\right) \le \prod_{p\le r^{1/32}}\frac{p-1}{p} \le \frac{32}{\log r},$$ 
where we used Mertens' third theorem. 
Since $W(r)/\gcd(W(r),bm) \,|\,W(r)/\gcd(W(r),m)$, the number of integers coprime to $W(r)/\gcd(W(r),m)$ contained in any arithmetic progression of length $k\ge r^{1/16}$ and common
difference $b$ is at most 
\begin{equation}
\left(\prod_{p|\gcd(W(r), bm)}\frac{p}{p-1}\right)\cdot\frac{64k}{\log r}=\frac{\gcd(W(r), bm)}{\phi(\gcd(W(r), bm))}\cdot \frac{64k}{\log r}, \label{eq:Selberg}
\end{equation}assuming that the elements of the arithmetic progression are coprime to $b$. 

If the elements of $X$ are not coprime to $b$, let $d$ be the greatest common divisor of $b$ and the elements of $X$. Let $Y=\{x/d:x\in X\}$. Then $Y$ is an arithmetic progression of size $|X|$ and common difference $b/d$ whose elements are coprime to $b/d$. Furthermore, the number of elements of $X$ coprime to $W(r)/\gcd(W(r),m)$ is at most the number of elements of $Y$ coprime to $W(r)/\gcd(W(r),m)$. By (\ref{eq:Selberg}), the number of elements of $Y$ coprime to $W(r)/\gcd(W(r),m)$ is at most 
\[
\frac{\gcd(W(r), bm/d)}{\phi(\gcd(W(r), bm/d))}\cdot \frac{64|Y|}{\log r} \le \frac{\gcd(W(r), bm)}{\phi(\gcd(W(r), bm))}\cdot \frac{64|X|}{\log r},
\]
where we used that 
$$\frac{\gcd(W(r), bm/d)}{\phi(\gcd(W(r), bm/d))} = \prod_{p|\gcd(W(r), bm/d)}\frac{p}{p-1} \le \prod_{p|\gcd(W(r),bm)}\frac{p}{p-1} = \frac{\gcd(W(r), bm)}{\phi(\gcd(W(r), bm))}.$$Thus, for any arithmetic progression $X$ with common difference $b$, the number of integers coprime to $W(r)/\gcd(W(r),m)$ in $X$ is at most 
\begin{equation}
\frac{\gcd(W(r), bm)}{\phi(\gcd(W(r), bm))}\cdot \frac{64|X|}{\log r}.\label{eq:Selberg2}
\end{equation}
The first claim in the lemma follows immediately upon noticing that $bm\le n^3$, so $\frac{\gcd(W(r), bm)}{\phi(\gcd(W(r), bm))} < 2\log \log (bm) < 4\log \log n$. 

The second claim in the lemma follows from (\ref{eq:Selberg2}) by observing that when $b=1$, 
\begin{align*}
\frac{\phi(\gcd(W(r), bm))}{\gcd(W(r), bm)} \cdot \log r &\ge \prod_{p|\gcd(W(r),m)} \frac{p-1}{p} \cdot  \frac{1}{2\tau(r,1)} \\
&= \frac{1}{2} \prod_{p|\gcd(W(r),m)} \frac{p-1}{p} \prod_{p|W(r)}\frac{p}{p-1}\\
&= \frac{1}{2} \prod_{p|W(r),p\nmid m}(1-1/p)^{-1},
\end{align*}
where we used (\ref{eq:bound-Wrho}) in the first inequality. 
\end{proof}

\subsection{Further estimates for Subsection \ref{subsec:monochromatic-subset-sums}\label{appendix:monochromatic}}

In this subsection, we collect several important estimates that are used throughout Subsection \ref{subsec:monochromatic-subset-sums}.
To this end, let $n$ be a sufficiently large positive integer and $m \in [n,\binom{n}{2}]$. For a positive integer $\rho$, recall that $W(\rho)=\prod_{i=1}^{\rho}p_i$ and $\tau(\rho,m)=\phi(W(\rho)m)/(W(\rho)m)=\prod_{p|W(\rho)m}(1-1/p)$. We define $\rho(n,m)$ to be the smallest positive integer $\rho$ such that $$\rho/\tau(\rho,m) \ge n^2/\phi(m).$$ Note that $\rho/\tau(\rho,m)$ is increasing as a function of $\rho$ and 
\begin{equation}
\frac{m}{\phi(m)}\cdot \frac{W(\rho)}{\phi(W(\rho))} \ge \frac{1}{\tau(\rho,m)} \ge \max\left(\frac{m}{\phi(m)}, \frac{W(\rho)}{\phi(W(\rho))}\right). \label{eq:mrho}
\end{equation}
For sufficiently large $\rho$, we have, by Mertens' third theorem, that 
\begin{equation}
\frac{1}{\tau(\rho,1)} = \frac{W(\rho)}{\phi(W(\rho))} \in [1.6\log \rho, 1.8\log \rho].\label{eq:bound-Wrho}
\end{equation}Thus, 
\begin{equation}
\tau(\rho,m) \ge \frac{\phi(m)/m}{2\log \rho}. \label{eq:tau-log2}
\end{equation}
Hence, for $\rho$ sufficiently large with $\rho \le n$, 
\begin{equation}\label{eq:ineq-tau}
\frac{1}{\tau(\rho,m)} \le 4\log \rho \log \log m \le 8\log n\log \log n.
\end{equation}Furthermore, for $\rho\ge 10\log m/\log \log m$, noting that $m$ has at most $2\log m/\log \log m$ distinct prime factors larger than $10\log m/\log \log m$, we have 
\begin{equation*}
\prod_{p|m,p>p_\rho} (1-1/p)^{-1} \le \left(1-\frac{\log \log m}{10\log m}\right)^{-2\log m/\log \log m} \le 2, \label{eq:bound-tau-large}
\end{equation*} 
so 
\begin{equation}
\frac{1}{\tau(\rho,m)} \in [\log \rho, 4\log \rho]. \label{eq:tau-log}
\end{equation} 

The next claim gives the order of $\rho(n,m)$ when $m\le n^2/(\log n)^2$. 

\begin{claim}\label{claim:orderrho}
For $m\le n^2/(\log n)^2$,  
\begin{equation*}
\rho(n,m) = \Theta\left(\frac{n^2/\phi(m)}{\log(n^2/\phi(m))}\right).
\end{equation*}
\end{claim}

\begin{proof}
Since $m\le n^2/(\log n)^2$, we have $n^2/\phi(m) > n^2/m \ge (\log n)^2$. 
Moreover, if $\rho$ is a positive integer such that $\rho < 10\log m/\log \log m$, then, by (\ref{eq:ineq-tau}), we have that 
\[
\rho/\tau(\rho,m) \le \frac{10\log m}{\log \log m} \cdot 4\log \rho \log \log m < (\log n)^2 < n^2/\phi(m).
\]
Thus, we must have $\rho(n,m) \ge 10\log m/\log \log m$. 

If now $\rho$ is a positive integer such that $\rho \ge 10\log m/\log \log m$, we have $\tau(\rho,m)^{-1} \in [\log \rho,4\log \rho]$ by (\ref{eq:tau-log}). Therefore, if $\rho \ge 16\frac{n^2/\phi(m)}{\log(n^2/\phi(m))}$, then, by monotonicity of $\rho \mapsto \frac{\rho}{\tau(\rho,m)}$, 
\[
\frac{\rho}{\tau(\rho,m)} \ge \frac{16\frac{n^2}{\phi(m)}\cdot \log\left(\frac{16n^2/\phi(m)}{\log(n^2/\phi(m))}\right)}{\log(n^2/\phi(m))} \ge 8\frac{n^2}{\phi(m)}
\]
and so $\rho(n,m) \le 16\frac{n^2/\phi(m)}{\log(n^2/\phi(m))}$. On the other hand, if $10\log m/\log \log m\le \rho \le  \frac{1}{16}\frac{n^2/\phi(m)}{\log(n^2/\phi(m))}$, then 
\[
\frac{\rho}{\tau(\rho,m)} \le \frac{\frac{n^2}{\phi(m)}\cdot 4\log\left(\frac{n^2/\phi(m)}{16\log(n^2/\phi(m))}\right)}{16\log(n^2/\phi(m))} \le \frac{1}{4}\frac{n^2}{\phi(m)}
\]
and so $\rho(n,m) \ge \frac{1}{16}\frac{n^2/\phi(m)}{\log(n^2/\phi(m))}$, as required. 
\end{proof}

Recall that $\psi(n,m)=\frac{m^{1/3}(m/\phi(m))}{(\log n)^{1/3}(\log\log n)^{2/3}}$ and ${\cal R}(n,m)=\min\left(\psi(n,m),\rho(n,m)\right)$. Using Claim \ref{claim:orderrho}, it is easy to show that ${\cal R}(n,m) = \Theta\left(\psi(n,m)\right)$ when $m = O\left( \frac{n^{3/2}(\log\log n)^{1/2}}{(\log n)^{1/2}}\right)$ and ${\cal R}(m,n) = \Theta(\rho(n,m))$ when $m=\Omega\left(\frac{n^{3/2}(\log\log n)^{1/2}}{(\log n)^{1/2}}\right)$.

Recall that in Subsection \ref{subsec:monochromatic-subset-sums}, we define $r=c{\cal R}(n,m)$ for a sufficiently small absolute constant $c$. The next claim establishes the existence of the integer $y$ used in Lemma \ref{lem:lem-choosey}. 

\begin{claim}\label{claim:cond-for-y}
Let $n$ and $m\in [n,\binom{n}{2}]$ be positive integers such that $n$ and $\rho(n,m)$ are sufficiently large. Let $r=c{\cal R}(n,m)$, where $c>0$ is sufficiently small. Then there exists an integer $y<n/2$ with 
$$m\in\left[\frac{y^{2}(m/\phi(m))\tau(r,m)}{25r},\frac{y^{2}(m/\phi(m))\tau(r,m)}{15r}\right].$$
Moreover, one may choose $y$ such that
\begin{equation} 
y\ge \max(r^2,n^{3/5}) \label{eq:ycond}
\end{equation}
and
\begin{equation} 
64 (m/\phi(m)) \tau(r, m) \frac{y}{r\log r} > n^{1/4}. \label{eq:ycond2}
\end{equation}
\end{claim}

\begin{proof}
We consider the cases $n\le m\le \frac{n^{3/2}(\log\log n)^{1/2}}{(\log n)^{1/2}}$ and $\frac{n^{3/2}(\log\log n)^{1/2}}{(\log n)^{1/2}}<m\le \binom{n}{2}$ separately.

\vspace{2mm}

{\noindent \bf Case 1:} $n\le m\le\frac{n^{3/2}(\log\log n)^{1/2}}{(\log n)^{1/2}}$.  

\vspace{2mm}

In this case, we have
$\frac{n^{1/3}}{(\log n)^{2/3}}\le r\le\frac{cCm^{1/3}(m/\phi(m))}{(\log n)^{1/3}(\log\log n)^{2/3}}$, where $C$ is some absolute constant independent of all other parameters. Since $r\ge \frac{n^{1/3}}{(\log n)^{2/3}}$, by (\ref{eq:tau-log}), we have $1/\tau(r,m) \in [\log r,4\log r]$. We also have 
\[
\sqrt{\frac{mr\log r}{m/\phi(m)}}\le\sqrt{\frac{cCm^{4/3}\log n}{(\log n)^{1/3}(\log\log n)^{2/3}}}\le n\sqrt{cC}.
\] Thus, for sufficiently small $c$, there exists an integer $y$ such that $y< \sqrt{\frac{25mr}{(m/\phi(m))\tau(r,m)}}<\frac{n}{2}$ and $y > \sqrt{\frac{15mr}{(m/\phi(m))\tau(r,m)}}$. 
This integer $y$ then satisfies $y<n/2$ and 
$$m\in\left[\frac{y^{2}(m/\phi(m))\tau(r,m)}{25r},\frac{y^{2}(m/\phi(m))\tau(r,m)}{15r}\right].$$
Furthermore, we have 
\begin{equation}
\frac{r^{3}(m/\phi(m))}{m(\log n)}\le\frac{(cC)^{3}m(m/\phi(m))^{4}}{m(\log n)^{2}(\log\log n)^{2}}<1/2.\label{eq:bound-r^2}
\end{equation} 
Since $r\ge \frac{n^{1/3}}{(\log n)^{2/3}}>n^{3/10}$, we also have 
\[
y > \sqrt{\frac{15mr\log r}{m/\phi(m)}} > \sqrt{\frac{mr\log n}{m/\phi(m)}} >  r^{2} > n^{3/5},
\] 
where we used (\ref{eq:bound-r^2}) in the third inequality. 
\medskip

\vspace{2mm}

{\noindent \bf Case 2:} $\frac{n^{3/2}(\log\log n)^{1/2}}{(\log n)^{1/2}}<m\le \binom{n}{2}$.

\vspace{2mm}

In this case, we have $cC^{-1}\rho(n,m)\le r\le cC\rho(n,m)$, where $C$ is again an absolute constant and we assume that $\rho(n,m)$ is sufficiently large. By the definition of $\rho(n,m)$, 
\[
\sqrt{\frac{mr/\tau(r,m)}{m/\phi(m)}}\le\sqrt{\frac{2cCmn^{2}/\phi(m)}{m/\phi(m)}}\le n\sqrt{2cC}.
\]
Thus, for sufficiently small $c$, there exists an integer $y$ such that $y< \sqrt{\frac{25mr}{(m/\phi(m))\tau(r,m)}}<\frac{n}{2}$ and $y > \sqrt{\frac{15mr}{(m/\phi(m))\tau(r,m)}}$. 
This integer $y$ then satisfies $y<n/2$ and 
$$m\in\left[\frac{y^{2}(m/\phi(m))\tau(r,m)}{25r},\frac{y^{2}(m/\phi(m))\tau(r,m)}{15r}\right].$$

If $\frac{n^{3/2}(\log\log n)^{1/2}}{(\log n)^{1/2}}<m<n^{7/4}$, we have $\rho(n,m) \ge n^{1/8}$, so $\tau(\rho(n,m),m) \le 1/\log \rho(n,m) \le 10/\log n$ by (\ref{eq:tau-log}). Using this, we obtain $$\phi(m)^2 \ge \frac{m^2}{4(\log \log m)^2} \ge \frac{n^3}{16(\log \log n)(\log n)} \ge 4\tau(\rho(n,m),m)^2 n^3.$$ If $m\ge n^{7/4}$, we also easily have $$\tau(\rho(n,m),m)^2 n^3 \le n^3 \le \frac{1}{4}\phi(m)^2.$$ Since $\frac{\rho(n,m)}{\tau(\rho(n,m),m)} \in [\frac{n^2}{\phi(m)},\frac{2n^2}{\phi(m)}]$ by the definition of $\rho(n,m)$, we obtain $\rho(n,m) \le \sqrt{n}$ in both ranges $\frac{n^{3/2}(\log\log n)^{1/2}}{(\log n)^{1/2}}<m<n^{7/4}$ and $n^{7/4}\le m \le \binom{n}{2}$. 
Moreover, since $r\ge cC^{-1}\rho(n,m)$, 
$$\tau(r,m) =\tau(\rho(n,m),m) \prod_{i \in (r,\rho(n,m)],p_i\nmid m}(1-1/p_i)^{-1} \le \frac{1.8\log \rho(n,m)}{1.6\log r}\tau(\rho(n,m),m)\le 25\tau(\rho(n,m),m),$$ where we used that $\rho(n,m)$ is sufficiently large. 
Hence, 
\begin{equation}
y \ge \sqrt{\frac{mr}{25(m/\phi(m))\tau(r,m)}}\ge \frac{\sqrt{cC^{-1}}}{25} \sqrt{\frac{m\rho(n,m)}{(m/\phi(m))\tau(\rho(n,m),m)}} \ge \frac{n\sqrt{cC^{-1}}}{25}\ge c^2 C^2 \rho(n,m)^2 \ge r^{2}, \label{eq:yr^2}
\end{equation}
where we used the definition of $\rho(n,m)$, the bound $\rho(n,m)\le \sqrt{n}$ and assumed $c$ is sufficiently small. Furthermore, from (\ref{eq:yr^2}), for $n$ sufficiently large, we have $$y\ge \frac{n\sqrt{cC^{-1}}}{25} > n^{3/5}.$$

\vspace{2mm}

Thus, in both Case 1 and Case 2, there exists a choice of $y < n/2$ with $$m\in\left[\frac{y^{2}(m/\phi(m))\tau(r,m)}{25r},\frac{y^{2}(m/\phi(m))\tau(r,m)}{15r}\right]$$
such that (\ref{eq:ycond}) also holds. Moreover, (\ref{eq:ycond2}) holds, since
\begin{equation*} 
64 (m/\phi(m)) \tau(r, m) \frac{y}{r\log r} \ge \frac{32y^{1/2}}{(\log n)^2} > n^{1/4},
\end{equation*}
where we used that $y\ge \max(r^2,n^{3/5})$, $\log r\le \log n$ and, by (\ref{eq:tau-log2}), $(m/\phi(m))\tau(r,m) \ge \frac{1}{2\log r} \ge \frac{1}{2\log n}$.
\end{proof}


\begin{thebibliography}{10}

\bibitem{A}N.~Alon, Subset sums, \textit{J. Number Theory} \textbf{27}
(1987), 196--205.

\bibitem{AEr}N.~Alon and P.~Erd\H{o}s, Sure monochromatic subset sums,
\textit{Acta Arith.} \textbf{74} (1996), 269--272. 

\bibitem{AF}N.~Alon and G.~Freiman, On sums of subsets of a set of
integers, \textit{Combinatorica} \textbf{8} (1988), 297--306.

\bibitem{BP} R.~Balasubramanian and P.~P.~Pandey, On a theorem of Deshouillers and Freiman, \textit{European J. Combin.} \textbf{70} (2018), 284--296. 

\bibitem{BL96}
V.~Bergelson and A.~Leibman, Polynomial extensions of van der Waerden's and Szemer\'edi's theorems, {\it J. Amer. Math. Soc.} {\bf 9} (1996), 725--753.

\bibitem{B}B.~J.~Birch, Note on a problem of Erd\H{o}s, \textit{Proc.
Cambridge Philos. Soc.} \textbf{55} (1959), 370--373. 

\bibitem{Bosz} A. P. Bosznay, On the lower estimation of non-averaging sets, {\it Acta Math. Hungar.} {\bf 53} (1989), 155--157. 

\bibitem{BE1} S.~A.~Burr and P.~Erd\H{o}s, Completeness properties of
perturbed sequences, \textit{J. Number Theory }\textbf{13} (1981), 446--455. 

\bibitem{BE2}S.~A.~Burr and P.~Erd\H{o}s, A Ramsey-type property in additive
number theory, \textit{Glasgow Math. J. }\textbf{27} (1985), 5--10. 

\bibitem{BEGL96}
{S.~A.~Burr, P.~Erd\H{o}s, R.~L.~Graham and W.~Li,} {Complete sequences of sets of integer powers,} {\it Acta Arith.} {\bf 77} (1996), 133--138. 

\bibitem{Cassels}
J.~W.~S.~Cassels, On the representation of integers as sums of distinct summands taken from a fixed set, {\it Acta Sci. Math. (Szeged)} {\bf 21} (1960), 111--124.

\bibitem{COS}T.~Cochrane, M.~Ostergaard and C.~Spencer, Cauchy--Davenport
theorem for abelian groups and diagonal congruences, {\it Proc. Amer. Math. Soc.} {\bf 147} (2019), 3339--3345.

\bibitem{DF}J.~M.~Deshouillers and G.~Freiman, A step beyond Kneser's
theorem for abelian finite groups, \textit{Proc. London Math. Soc.}
\textbf{86} (2003), 1--28.

\bibitem{Er62}P.~Erd\H{o}s, On the representation of large integers as sums of distinct summands taken from a fixed set, \textit{Acta. Arith.} \textbf{7} (1962), 345--354.

\bibitem{Er81}
P.~Erd\H{o}s, Many old and on some new problems of mine in number theory, {\it Congr. Numer.} {\bf 30} (1981), 3--27.

\bibitem{Er82}
P.~Erd\H{o}s, Miscellaneous problems in number theory, {\it Congr. Numer.} {\bf 34} (1982), 25--45.

\bibitem{Er822}
P.~Erd\H{o}s, Some new problems and results in number theory, in Number theory (Mysore, 1981), 50--74,
Lecture Notes in Math., 938, Springer, Berlin-New York, 1982.

\bibitem{Er89}P.~Erd\H{o}s, Some problems and results on combinatorial number theory, in Graph theory and its applications: East and West
(Jinan, 1986), 132--145, Ann. New York Acad. Sci., 576, New York Acad. Sci., New York, 1989.

\bibitem{Er95}
P.~Erd\H{o}s, Some of my favourite problems in number theory, combinatorics, and geometry, \textit{Resenhas} \textbf{2}
(1995), 165--186. 

\bibitem{Er952}
P.~Erd\H{o}s, Some of my recent problems in combinatorial number theory, geometry and combinatorics, in Graph theory, combinatorics, and algorithms, Vol. 1, 2 (Kalamazoo, MI, 1992), 335--349, Wiley-Intersci. Publ., Wiley, New York, 1995.

\bibitem{EG}P.~Erd\H{o}s and R.~L.~Graham, {\bf Old and new problems and results in combinatorial number theory}, 
Monographies de L'Enseignement Math\'ematique, 28, Universit\'e de Gen\`eve, L'Enseignement Math\'ematique, Geneva, 1980.

\bibitem{ES} P.~Erd\H{o}s and A.~S\'{a}rk\"{o}zy, On a problem of Straus, in Disorder in Physical Systems, 55--66, Oxford Univ. Press, New York, 1990. 

\bibitem{EStr} P.~Erd\H{o}s and E.~G.~Straus, Nonaveraging sets II, in Combinatorial theory and its applications, II (Proc. Colloq., Balatonf\"ured, 1969),
405--411, North-Holland, Amsterdam, 1970. 

\bibitem{Folk}J.~Folkman, On the representation of integers as sums of distinct terms from a fixed sequence, \textit{Canadian J. Math.} \textbf{18} (1966), 643--655.

\bibitem{F2}G.~A.~Freiman, New analytical results in subset-sum problem, \textit{Discrete Math. }\textbf{114} (1993), 205--218.

\bibitem{G}R.~L.~Graham, Complete sequences of polynomial values,
\textit{Duke Math. J.} \textbf{31} (1964), 275--285. 

\bibitem{GuRo98} D. S. Gunderson and V. R\"odl, Extremal problems for affine cubes of integers, {\it Combin. Probab. Comput.} {\bf 7} (1998), 65--79. 

\bibitem{Helfgott} H.~A.~Helfgott, The ternary Goldbach problem, to appear in {\it Ann. of Math. Stud.}

\bibitem{Hoeff}W.~Hoeffding, Probability inequalities for sums of bounded random variables, {\it J. Amer. Statist. Assoc.} {\bf 58} (1963), 13--30.

\bibitem{Lev}V.~Lev, Consecutive integers in high-multiplicity sumsets,
\textit{Acta Math. Hungar.} \textbf{129} (2010), 245--253. 

\bibitem{Lip}E.~Lipkin, On representation of $r$th powers by subset
sums, \textit{Acta Arith.} \textbf{52} (1989), 353--365.

\bibitem{MV}H.~L.~Montgomery and R.~C.~Vaughan, {\bf Multiplicative Number Theory I: Classical Theory},
Cambridge Studies in Advanced Mathematics, 97, Cambridge University Press, Cambridge, 2007.

\bibitem{R53} K.~F.~Roth, On certain sets of integers, \textit{J.
London Math. Soc.} \textbf{28} (1953), 104--109.

\bibitem{RoSz54} K.~F.~Roth and G.~Szekeres, Some asymptotic formulae in the theory of partitions, {\it Quart. J. Math.} {\bf 5} (1954), 241--259. 

\bibitem{Sar2}A.~S\'{a}rk\"{o}zy, Finite addition theorems II, \textit{J.
Number Theory} \textbf{48} (1994), 197--218.

\bibitem{Sar3}
A. S\'ark\"ozy, On finite addition theorems, {\it Ast\'erisque} {\bf 258} (1999), 109--127.

\bibitem{Sp81}
J. Spencer, Suresums, {\it Combinatorica} {\bf 1} (1981), 203--208.

\bibitem{Sprague} R.~Sprague, \"Uber Zerlegungen in $n$-te Potenzen mit lauter verschiedenen Grundzahlen, {\it Math. Z.} {\bf 51} (1948), 466--468. 

\bibitem{Straus} E.~G.~Straus, Nonaveraging sets, in Combinatorics (Proc. Sympos. Pure Math., Vol. XIX, Univ. California, Los Angeles, Calif., 1968), 215--222, Amer. Math. Soc., Providence, R.I., 1971. 

\bibitem{SV}E.~Szemer\'edi and V.~H.~Vu, Finite and infinite arithmetic
progressions in sumsets, \textit{Ann. of Math. }\textbf{163} (2006),
1--35. 

\bibitem{SV2}E.~Szemer\'edi and V.~H.~Vu, Long arithmetic progressions
in sumsets: Thresholds and bounds, \textit{J. Amer. Math. Soc. }\textbf{19}
(2006), 119--169.

\bibitem{TVW}L.~Tran, V.~H.~Vu and P.~M.~Wood, On a conjecture
of Alon, \textit{J. Number Theory} \textbf{129} (2009),
2801--2807.

\bibitem{V}V.~H.~Vu, Some new results on subset sums, \textit{J. Number
Theory} \textbf{124} (2007), 229--233. 
\end{thebibliography}
\end{document}